\documentclass[a4paper]{article}

\usepackage[english]{babel}
\usepackage[utf8]{inputenc}
\usepackage{amsmath}
\usepackage{amsfonts}
\usepackage{amsthm}
\usepackage{tikz}
\usepackage{graphicx}
\usepackage[sort,numbers]{natbib}
\usepackage{color}
\usepackage[colorlinks]{hyperref}
\usepackage{cleveref}
\usepackage[affil-it]{authblk}

\def\mb{\mathbf}

\newtheorem{mydef}{Definition}
\newtheorem{lem}{Lemma}
\newtheorem{cor}{Corollary}
\newtheorem{thm}{Theorem}

\newtheorem{rem}{Remark}
\theoremstyle{empty}
\newtheorem{duplicate}{Appendix Theorem}

\title{\vspace{-20pt}Degree switching and partitioning for enumerating graphs to arbitrary orders of accuracy }

\author[a]{David E. Burstein
\thanks{Electronic address: \texttt{dburstein@pitt.edu}}}

\author[a]{Jonathan E. Rubin  \thanks{Electronic address: \texttt{jonrubin@pitt.edu}; Corresponding author. }}

\affil[a]{Department of Mathematics\\University of Pittsburgh \thanks{The Dietrich School of Arts and Sciences. 301 Thackeray Hall.  Pittsburgh, PA 15260}}

\begin{document}
\maketitle
\begin{abstract}
We provide a novel method for constructing asymptotics (to arbitrary accuracy) for the number of directed graphs that realize a fixed bidegree sequence $\mathbf{d=(a,b)}\in\mathbb{Z}^{N\times2}$ with maximum degree $d_{max}=O(S^{\frac{1}{2}-\tau})$ for an arbitrarily small positive number $\tau$, where $S$ is  the number edges specified by $\mathbf{d}$.  Our approach is based on two key steps,  graph partitioning and degree preserving switches.  The former idea allows us to relate enumeration results for given sequences to those for sequences that are especially easy to handle, while the latter facilitates expansions based on numbers of shared neighbors of pairs of nodes.  While we focus primarily on directed graphs allowing loops, our results can be extended to other cases, including bipartite graphs, as well as directed and undirected graphs without loops.  In addition, we can relax the constraint that $d_{max}=O(S^{\frac{1}{2}-\tau})$ and replace it with  $a_{max}b_{max}=O(S^{1-\tau})$, where $a_{max}$ and $b_{max}$ are the maximum values for $\mathbf{a}$ and $\mathbf{b}$ respectively.  The previous best results, from Greenhill et al. 2006 \cite{Greenhill06}, only allow for $d_{max}=o(S^{\frac{1}{3}})$ or alternatively $a_{max}b_{max}=o(S^{\frac{2}{3}})$.  Since in many real world networks, $d_{max}$ scales larger than $o(S^{\frac{1}{3}})$, we expect that this work will be helpful for various applications.
\end{abstract}  

\noindent Sequential Importance Sampling  
\newline\newline

\section{Introduction}
Given a degree sequence, a fundamental question to ask is whether it is graphic; that is, does there exist a graph for which the degrees of the nodes are exactly the elements of the sequence?
A more refined view goes beyond simply considering graphicality as a yes-or-no property and recognizes that there may be very different numbers of graphs that realize different graphic degree sequences.
Our main goal in this work is to develop formulas that approximate the numbers of graphs that realize degree sequences with certain properties, which are valid asymptotically as the number of nodes in the degree sequence and in the corresponding graphs goes to infinity. 

The problem of counting graphs that realize a given degree sequence can be recast as a problem of counting $0-1$ binary matrices.
Specifically, counting the number of rectangular $0-1$ binary matrices with fixed row and column sums is equivalent to counting the number of bipartite graphs with a fixed bidegree sequence.  Alternatively, counting the number of square $0-1$ binary matrices with fixed row and column sums is equivalent to counting the number of directed graphs with loops that realize a given bidegree sequence.  Since  any rectangular $0-1$ binary matrix can be arbitrarily extended to a square $0-1$ binary matrix by adding either rows or columns of $0's$, we focus here on square binary matrices.  

In terms of applications, the aforementioned counting problem is an important step for uniformly generating $0-1$ matrices (contingency tables) with fixed row and column sums.
Uniform generation of $0-1$ binary matrices has many applications, from detecting statistically significant subgraphs (motifs) in a network in data mining \cite{Itzkowitz03,Gionis07,Chindelevitch} to determining the impact of the degree sequence on emergent dynamics in a network of nodes with temporally varying states \cite{Durrett10,Zhao11,Vespignani12}.  While the methods proposed for (almost) uniformly generating graphs from a given degree sequence are quite diverse, many (if not most) of the methods can be classified into two categories, Markov Chain Monte Carlo (MCMC) methods and Sequential Importance Sampling (SIS).

The traditional MCMC method involves swapping edges many times to generate an approximately uniformly random sample.  The main drawback of this method is that there is an unknown mixing time.  Naturally,  by knowing precise asymptotics for the number of graphs with a prescribed degree sequence, we can deduce asymptotic probabilities for the likelihood two nodes share an edge.  Consequently, we could use these asymptotic probabilities as a criteria to empirically help us determine the mixing time. There are quite a few interesting technical issues for implementing this method and we refer the reader to the literature for details \cite{Rao96,Milo03,Itzkowitz04,Berger10,Gionis07,Ray12,Verhelst}.  We also mention that there are MCMC methods that have probable bounds for almost uniformly sampling graphs with a prescribed degree sequence, but they are also computationally expensive \cite{Bezakova06}.

Alternatively, SIS methods sample the number of graphs with a prescribed degree sequence in a biased way. A large sample of graphs is taken from a biased distribution and a Law of Large Numbers argument is used to construct a new (approximately) uniform distribution based on the output of the biased sampling procedure \cite{Blitzstein,Chen05,DelGenio}.  The biggest drawback of SIS is that we often do not know how large a sample of graphs we need from our biased distribution to reliably construct the approximately uniform distribution.  Indeed, past work has shown that certain constraints on degree sequences may be required to ensure the computational efficiency of SIS methods \cite{Bayati,Blanchet}; without such constraints,  an exponentially large sample size may be required to attain meaningful estimates for approximating the uniform distribution \cite{Bezakova12}.   Statistical arguments show that to increase the speed of convergence, we want an initial biased distribution that is quite close to the uniform distribution \cite{Asmussen,Blanchet}.  Incorporating asymptotically accurate graph enumeration formulas in constructing a biased distribution could conceivably improve the performance of such methods.  

The derivation of asymptotic formulas for the number of $0-1$ binary matrices with fixed row and column sums has a rich history spanning at least as far back as 1958 with Read \cite{Read58}.   
Progress in this area has been made by restricting to matrices that are sparse \cite{Bender}, in the sense that row and column sums grow at most as a fractional power of the norm of the matrix or by restricting to matrices that are dense but have limited variation among the row and column sums \cite{Barvinok,Canfield08}.  We focus on the sparse case.  In this setting, McKay \cite{McKay84} developed a formula to count such matrices that is valid in an asymptotic sense in the limit as the number of edges $S$ becomes arbitrarily large, assuming that the maximum row sum or column sum grows as $o(S^{\frac{1}{4}})$.  More recently, Greenhill et al. \cite{Greenhill06} generalized McKay's formula, obtaining a result that holds if the maximum row sum or column sum is $o(S^{\frac{1}{3}})$.  Our motivation for considering this counting problem comes from the construction of directed graphs representing neuronal networks in the analysis of the spread of neural activity in certain brain regions.  Given that such networks can consist of many thousands of neurons, commonly with connectivities of up to 10 $\%$, we aim to count matrices with row and column sums, corresponding to in- and out-degrees, that exceed $O(S^{\frac{1}{3}})$.  
In fact, since we are asymptotically guaranteed graphicality
provided that the maximum degree (row sum or column sum) is $O(S^{\frac{1}{2}-\tau})$, as demonstrated in \cite{B15}, we expect that we should be able to generalize the above asymptotic formulas to allow for a maximum degree of $O(S^{\frac{1}{2}-\tau})$.

To attain this generality, we initially estimate the ratio $\|G_{\mathbf{d_1}}\|/\|G_{\mathbf{d_2}}\|$, where $\|G_{\mathbf{d_i}}\|$ is the number of directed graphs with loops that realize the bidegree sequence $\mathbf{d}_i$, $i=1,2$, \textcolor{black}{under the constraint that the maximum degree for both of these bidegree sequences is $O(S^{\frac{1}{2}-\tau})$ for any $\tau\in(0,\frac{1}{2}]$.}  We can estimate this ratio accurately when the taxicab norm of the difference of the bidegree sequences, $\| \mathbf{d_1}-\mathbf{d_2}\|_{1}$, equals 2, and this relation is assumed in the statements of the theorems that we prove.
We can apply the theorems in this work to estimate $\|G_{\mathbf{d_{k+1}}}\|/\|G_{\mathbf{d_0}}\|$, where $\| \mathbf{d_{k+1}}-\mathbf{d_0}\|_{1} > 2$, by considering a product  
$\Pi_{i=0}^{k} \|G_{\mathbf{d_{i+1}}}\|/\|G_{\mathbf{d_{i}}}\| = \|G_{\mathbf{d_{k+1}}}\|/\|G_{\mathbf{d_{0}}}\|$, where for all $i$, $\| \mathbf{d_i}-\mathbf{d_{i+1}}\|_{1} = 2$. (For a rigorous proof that we can construct a sequence of $\mathbf{d_{i}}$'s in this fashion, we refer the reader to a result by Muirhead that can be found in \cite{Berger14}.) 

Our main results can be summarized as follows:
 \begin{itemize}
\item In Theorem \ref{thm:partition}, we derive an expansion for $\|G_{\mathbf{d}}\|$ that holds in general for all degree sequences. 
\item In Corollary \ref{cor:ratio}, we exploit sparsity constraints to prove that the terms in the expansion of $\|G_{\mathbf{d_1}}\|/\|G_{\mathbf{d_2}}\|$ based on Theorem \ref{thm:partition} decrease geometrically.
\item Starting with Corollary \ref{cor:expand}, we establish an asymptotic approximation for $\|G_{\mathbf{d_1}}\|/\|G_{\mathbf{d_2}}\|$ up to errors of size $O(S^{-2\tau})$ for small $\tau>0$, where $S$ is the number of edges in the graph.
\item Under modest assumptions regarding the asymptotic approximation for  $\|G_{\mathbf{d_1}}\|/\|G_{\mathbf{d_2}}\|$ with $O(S^{-2w\tau})$ error terms for some positive integer $w$, in Theorem \ref{thm:arborder1} we provide a general method that yields an approximation for $\|G_{\mathbf{d_1}}\|/\|G_{\mathbf{d_2}}\|$ up to errors of size $O(S^{-2\tau-2w\tau})$.
\item Next, in Theorem \ref{thm:arborder2}, we demonstrate that if we know that our approximation for $\|G_{\mathbf{d_1}}\|/\|G_{\mathbf{d_2}}\|$ is correct up to $O(S^{-\gamma})$ errors, where $\frac{1}{2}\leq \gamma$, then we can derive a sharper approximation of  $\|G_{\mathbf{d_1}}\|/\|G_{\mathbf{d_2}}\|$ where our new error term is now $O(S^{-\gamma-2\tau})$ without making the  ``modest assumptions'' previously made in Theorem \ref{thm:arborder1}. (As the proofs of Theorems \ref{thm:arborder1} and \ref{thm:arborder2} are very similar, we place the proof of Theorem \ref{thm:arborder2} in Appendix C.)
\item In Theorem \ref{thm:graphcount}, we demonstrate how we can recover an (arbitrarily) accurate asymptotic approximation of $\|G_{\mathbf{d_1}}\|$ with knowledge of an (arbitrarily) accurate approximation of the ratio  $\|G_{\mathbf{d_1}}\|/\|G_{\mathbf{d_2}}\|$.
\item Several subsequent theorems in Section 5 extend our $O(S^{-2\tau})$ approximations up to several successively higher order error terms.
\item Theorem \ref{thm:assumption}, which builds on Theorems \ref{thm:alg}, \ref{thm:alg2}, then verifies that the ``modest assumptions'' of Theorem \ref{thm:arborder1} indeed hold if we allow errors up to  $O(S^{-1/2})$.  Hence, we can use Theorem \ref{thm:arborder1} as many times as needed to reduce errors to $O(S^{-1/2})$ and then use Theorem  \ref{thm:arborder2} to reduce errors to arbitrary order.
\item  In Appendix A, we explain how to generalize our results to the case where the product of the maximet estimae um in-degree and maximum out-degree is of $O(S^{1-\tau})$.
\item  In Appendix B, we illustrate how to extend our results to undirected and directed graphs, including the case where loops are prohibited,  as well as graphs where we prohibit edges between certain nodes.
\end{itemize}

We now detail our proof strategy.  First, in Section 2, we use {\em graph partitioning}, inspired by \cite{Miller13}, to construct novel expansions for the number of graphs that realize a bidegree sequence.  That is, for a particular realization of a bidegree sequence, we can partition our adjacency matrix into two submatrices, one submatrix containing just the $i$th and $j$th rows (or columns), and another submatrix containing the remaining $N-2$ rows.  In turn, we obtain two ``smaller" bidegree sequences for both of our submatrices.  Once we demonstrate how to count the number of graphs that realize the smaller bidegree sequence corresponding to the two-row submatrix, we obtain the following expression for $\|G_{\mathbf{d}}\|$ in terms of the number of graphs where two arbitrarily chosen nodes $i$ and $j$  (with degrees $a_{i}$,$a_{j}$) have $k$ common neighbors:
\begin{equation}
\label{eq:prelim}
\|G_{\mathbf{d}}\|=\sum_{k=0}^{a_{j}}\{\binom{a_{i}+a_{j}-2k}{a_{j}-k}\sum_{\mathbf{r}\in X_{k}}\|G_{\mathbf{r}}\|\}
\end{equation}
where $\mathbf{r}$ is a (residual) bidegree sequence of the $N-2$ remaining rows and $X_{k}$ is a set of (residual) bidegree sequences corresponding to graphs where the $i$th and $j$th nodes have exactly $k$ common neighbors.

Next, in Section 3, we introduce the idea of {\em degree preserving switches}, as discussed in \cite{McKay91,McKay10,Milo03,Rao96}, in which we make a single edge replacement to eliminate a common neighbor of two nodes without changing any nodes' degrees.  Counting graphs with common images or pre-images under degree preserving switches allows us to prove that for sparse graphs, the dominant term in the expansion only involves instances where there are no common neighbors between the two nodes $i$ and $j$.  That is, in the notation of (\ref{eq:prelim}), $$\|G_{\mathbf{d}}\|\approx \binom{a_{i}+a_{j}}{a_{j}}\sum_{\mathbf{r}\in X_{0}}\|G_{\mathbf{r}}\|.$$  
Moreover, it turns out that the set $X_{0}$ in the prior expression does not change if we consider $\mathbf{d_*}$ where $\mathbf{d_*=d}$ except that node $i$ in $\mathbf{d_*}$ has degree $a_{i}-1$ instead of $a_i$ for $\mathbf{d}=(\mathbf{a},\mathbf{b})$ with $\mathbf{a}=(a_1, \ldots, a_N)^T$.
Consequently, we find that $$\frac{\|G_{\mathbf{d}_*}\|}{\|G_{\mathbf{d}}\|} \approx  \binom{a_{i}+a_{j}-1}{a_{i}-1}/\binom{a_{i}+a_{j}-1}{a_{j}-1}=\frac{a_{i}}{a_{j}}.$$  

A subtlety of the proof is that we first establish the above statement where node $i$ or $j$ has bounded (in)-degree.  We then show that the above relationship still holds even for degree sequences that lack a node of bounded degree by using the fact that such degree sequences are very close in taxicab norm to a degree sequence that contains a node of bounded degree.  

We then proceed in Section 4 by proving (under modest assumptions) a recursion that allows for more refined approximations of $\|G_{d}\|$.  In particular we demonstrate that the higher order error terms from our approximation of $\frac{\|G_{\mathbf{d}_*}\|}{\|G_{\mathbf{d}}\|}$ in Section 3, consist of the number of graphs that realize some set of auxiliary degree sequences.  Iteratively approximating the number of graphs that realize these auxiliary degree sequences yields the recursion.  We then show how asymptotics for $\| G_\mathbf{d} \|$ follow from those obtained for the ratio $\| G_\mathbf{d} \| / \| G_{\mathbf{d_*}} \|$, where $\mathbf{d_*}$ is a degree sequence designed such that both this ratio and $\| G_{\mathbf{d_*}} \|$  itself can be estimated.  More precisely, we achieve this result by working with a sequence of  intermediary degree sequences, starting from a sequence for which it is easy to compute the number of graphs that realize it and, from there, progressing successively closer to $\mathbf{d}$.

Subsequently in Section 5, we put the technique proposed by Section 4 into practice, computationally verifying that the 'modest assumptions' do indeed hold for certain specific orders of error and deriving various asymptotic approximations for the ratio $\| G_\mathbf{d} \| / \| G_{\mathbf{d_*}} \|$.  Finally, in Section 6, we prove that the assumptions hold up to $O(S^{-1/2})$ errors, which together with the previous results shows that our method yields arbitrarily accurate asymptotic approxiations for the ratio $\| G_\mathbf{d} \| / \| G_{\mathbf{d_*}} \|$.  Establishing these results relies on deriving approximations for sums of the form $\sum_{x_1\neq...\neq x_r}^{N} \Pi_{i=1}^{r} f(x_i)$  that come up in our expansions.  Successively improved approximations are obtained by relating such expressions with the inequality of arguments enforced to similar expressions where certain arguments are either allowed to be equal or are forced to be equal.  When such terms appear in our ratios, these relations can be used to represent numerators in the right way to achieve cancellation with corresponding denominators, such that terms that remaining after cancellaton are of higher order.

\section{Using Partitioning to Count Graphs}

In this paper, we will consider {\em bidegree sequences}, each denoted by $\mathbf{d}=(\mathbf{a},\mathbf{b})\in \mathbb{Z}^{N\times 2}$, where for concreteness we specify that $\mathbf{a}$ lists the in-degrees of the nodes in the sequence and $\mathbf{b}$ the out-degrees.  We are interested in graphs that realize such sequences, where we allow either 0 or 1 connection between each pair of nodes as well as single self-loops within each graph.   Throughout the paper, we use $S$ to denote the number of edges in the graphs that realize $\mathbf{d}$.
Since we do not need to distinguish between degree sequences and bidegree sequences in this work, we will henceforth simply refer to these as degree sequences.
In one special case, it is particularly easy to count
the number of graphs that realize a specified degree sequence.  

\begin{lem} 
\label{lem:prelim}
Suppose that   
\[
\mathbf{a}=\{a_1,\ldots,a_k,0,\ldots,0\}, \mathbf{b}=\{k,..,k,1,..,1,0,..,0\}, \; \; \mbox{and} \; \; \sum_{i=1}^{k}a_{i}=\sum_{i=1}^{N}b_{i}=:S.
\] 
Let $q$ denote the number of times $k$ appears in $\mathbf{b}$ and let $p = \min \{ a_1, \ldots, a_k \}$.
If $p< q$, then there does not exist a graph that realizes this bidegree sequence.
If $p \geq q$ then there are $$\frac{(S-qk)!}{\Pi_{i} (a_{i}-q)!}$$ graphs that realize the bidegree sequence.  Similarly, if $\mathbf{a}=\{k,..,k,1,..,1,0,..,0\}$ and  $\mathbf{b}=\{b_1,\ldots,b_k,0,\ldots,0\}$, with corresponding definitions of $p$ and $q$, then there are 
$$\frac{(S-qk)!}{\Pi_{i} (b_{i}-q)!}$$ graphs that realize the bidegree sequence.
\end{lem}
\begin{proof}
We present the proof for the first case, since the second is completely analogous.
We first note that the $q$ nodes in $\mathbf{b}$ with out-degrees equal to $k$ must connect to all of the nodes with nonzero degree in $\mathbf{a}$.  Of the $S-qk$ remaining outward edges, start with the node that corresponds to $a_{1}$.  There are $S-qk$ choices for outward edges that can supply the $a_1-q$ unconnected inward edges to this node, such that there are $\binom{S-qk}{a_{1}-q}$ possible ways to link this node into the graph.  Once  these $a_1-q$ edges have been connected, there are $\binom{S-qk-a_{1}+q}{a_{2}-q}$ ways to link the node  that corresponds to $a_{2}$ into the graph.  Notice that 
\[
\binom{S-qk-a_{1}+q}{a_{2}-q}\binom{S-qk}{a_{1}-q}=\frac{(S-qk)!}{(a_{1}-q)!(a_{2}-q)!(S-qk-a_{1}-a_{2}+2q)!}.
\]
  Multiplying inductively, the $(S-qk-a_{1}-a_{2}+2q)!$ term disappears and we obtain the desired result.
\end{proof}

\begin{rem} Lemma \ref{lem:prelim} generalizes immediately to arbitrary permutations of the given $\mathbf{a}$ and $\mathbf{b}$.
\end{rem}

With the above lemma in hand, we can construct an inductive characterization for the number of graphs that realize a specified degree sequence.   First, though, we must define some notation.

\begin{mydef}
Let $\|G_{\mathbf{d}}\|$ denote the number of graphs that realize a degree sequence $\mathbf{d}$.  Furthermore, for a set $X$ of degree sequences, let $\|G_{X}\|$ be the number of graphs that realize any degree sequence in $X$.
\end{mydef}

Now, consider an arbitrary adjacency matrix $\textbf{A}\in{\{0,1\}}^{M\times N}$.  We can write $\textbf{A}$ as 
\[
\textbf{A}=\begin{bmatrix} \mathbf{A}_{1} & \mathbf{A}_{2} &\cdots & \mathbf{A}_{N} \\ \delta_{M-1,1} & \delta_{M-1,2} & \cdots & \delta_{M-1,N}  \\ \delta_{M,1} & \delta_{M,2} & \cdots & \delta_{M,N} \end{bmatrix}
\]
where for each $i$, $\mathbf{A}_{i}\in\{0,1\}^{(M-2)\times 1}$.
Of course, letting  $\textbf{0}$ denote a column vector of $M-2$ zeros, we have 

$$\textbf{A} =\begin{bmatrix} \mathbf{A}_{1} & \mathbf{A}_{2} &\cdots &\mathbf{A}_{N} \\ 0 & 0 & \cdots & 0 \\ 0 & 0 & \cdots & 0 \end{bmatrix}+\begin{bmatrix} \textbf{0} & \textbf{0} &\cdots &\textbf{0} \\ \delta_{M-1,1} & \delta_{M-1,2} & \cdots & \delta_{M-1,N}  \\ \delta_{M,1} & \delta_{M,2} & \cdots & \delta_{M,N}  \end{bmatrix}.$$   

Call the first and second matrices in this equation $\mathbf{A}_{l}$ and $\mathbf{A}_{r}$, respectively.  Now, if $\textbf{A}$ realizes a given degree sequence, $(\mathbf{a},\mathbf{b})$, then 
there is a vector $\{s_{1},...,s_{N}\}$ such that the degree sequence of (the graph corresponding to) $\mathbf{A}_{r}$ is $(\{0,...,0,a_{M-1},a_{M}\}, \{s_{1},....,s_{N}\})$ and the degree sequence of $\mathbf{A}_{l}$ is $(\{a_{1},...,a_{M-2},0,0\},\{b_{1}-s_{1},....,b_{N}-s_{N}\})$, with the constraint that none of the $s_{i}$ (i.e., the column sums of $\mathbf{A}_{r}$) can exceed 2 and the $s_i$ must sum to $a_{M-1}+a_{M}$.  If we know the number of $s_{i}$ that equal 2, we can invoke Lemma 1 to count the number of realizations of the degree sequence of $(\{0,...,0,a_{M-1},a_{M}\}, \{s_{1},....,s_{N}\})$.  This idea, extended to a partition of the $i$th and $j$th rows rather than specifically the $(M-1)$st and $M$th rows, motivates a useful theorem.  
To state the theorem, we define the set of degree sequences
\begin{equation}
\label{eq:Xk}
\begin{array}{rl}
X_{k}(i,j;\mathbf{d})=\{(\mathbf{a}-a_{i}\mathbf{e_{i}}-a_{j}\mathbf{e_{j}} , \mathbf{b}-\mathbf{s}): & \#(s_{n}=2)=k, \; \#(s_{n}\geq 3)=0, \\
& \mbox{and} \; \; \sum_{n=1}^{N}s_{n}=a_{i}+a_{j}\}.
\end{array}
\end{equation}
Note that in equation (\ref{eq:Xk}), we explicitly represent the positions $(i,j)$  from which edges will be partitioned out as well as the base degree sequence $\mathbf{d} = (\mathbf{a},\mathbf{b})$.  $X_k(i,j;\mathbf{d})$ is a set of degree sequences, even with $\mathbf{d}$ fixed, because different choices of $\mathbf{s}$ can be made that fit the definition in (\ref{eq:Xk}).   
The  notation $X_k(i,j;\mathbf{d})$ is cumbersome and the arguments of $X_k$ will be dropped when possible, but this notation will be needed to make  results precise later in the paper.  

\begin{thm} \label{thm:partition} Fix a degree sequence $\mathbf{d}=(\mathbf{a},\mathbf{b})$.  Pick an arbitrary pair of nodes, say with indices $i, j$ and corresponding in-degrees $a_{i},a_{j}$, where $a_{j}\leq a_{i}$, and define $X_k = X_k(i,j;\mathbf{d})$ as in (\ref{eq:Xk}).  Then $$\|G_{\mathbf{d}}\|=\sum_{k=0}^{a_{j}}\binom{a_{i}+a_{j}-2k}{a_{j}-k}\|G_{X_{k}}\|.$$
\end{thm}
\begin{proof}
Any adjacency matrix that realizes $\mathbf{d}$ can be partitioned into two adjacency matrices as we have discussed. 
Picking any two nodes, with in-degrees denoted by $a_{i},a_{j}$, we note that any realization of our degree sequence $\mathbf{d}$ must also be a realization of some degree sequence $(\mathbf{a}-a_{i}\mathbf{e_{i}}-a_{j}\mathbf{e_{j}} , \mathbf{b}-\mathbf{s}) \in X_{k}$ for some $k \leq N$, combined with a realization of the degree sequence $(a_{i}\mathbf{e_{i}}+a_{j}\mathbf{e_{j}} , \mathbf{s})$.  
In order for $(a_{i}\mathbf{e_{i}}+a_{j}\mathbf{e_{j}}, \mathbf{s})$ to be graphic, we require that $\#(s_{i}=2)\leq a_{j}$ (see Lemma 1), hence only graphs that realize degree sequences in $X_{k}$ for $k\leq a_{j}$ can correspond to our adjacency matrix partition.  If $(\mathbf{a}-a_{i}\mathbf{e_{i}}-a_{j}\mathbf{e_{j}} ,\mathbf{b}-\mathbf{s}) \in X_{k}$ for $k \leq a_j$, then  Lemma 1 implies that the number of graphs that realize $(a_{i}\mathbf{e_{i}}+a_{j}\mathbf{e_{j}}, \mathbf{s})$  is precisely $\binom{a_{i}+a_{j}-2k}{a_{j}-k}$.  By multiplying this quantity by the number of graphs that can be generated by the residual degree sequence $(\mathbf{a}-a_{i}\mathbf{e_{i}}-a_{j}\mathbf{e_{j}} ,\mathbf{b}-\mathbf{s})$, namely $\|G_{X_{k}}\|$, and summing over all $X_{k}$ for $k \leq a_{j}$, we obtain the desired result.     \end{proof}

\begin{rem} 
\textcolor{black}{Based on Definition 1,  $\|G_{X_k}\|$ represents the expression $\sum_{\mathbf{r}\in X_{k}}\|G_{\mathbf{r}}\|$, such that Theorem \ref{thm:partition} agrees with equation (\ref{eq:prelim}).}
\end{rem}

At this point, we introduce some additional notation.  The rationale is that we will want to compare numbers of graphs realizing two different degree sequences that are identical except that the in-degrees of two particular nodes in the sequences differ by 1.    A succint way to think of this relation is to start with a degree sequence $\mathbf{d} = (\mathbf{a},\mathbf{b})$ for which $\sum_k a_k = \sum_k b_k+1$.  
Now, for any entry $a_k \geq 1$ in $\mathbf{a}$, let $\mathbf{a_{-k}}=\mathbf{a}-\mathbf{e_{k}}$,  $\mathbf{d_{-i}}=(\mathbf{a_{-i}},\mathbf{b})$, and $\mathbf{d_{-j}}=(\mathbf{a_{-j}}, \mathbf{b})$.  Note that the sums of the in-degrees for the sequences $\mathbf{d_{-i}}$ and $\mathbf{d_{-j}}$ are identical and are equal to the sum of their out-degrees, and $\mathbf{d_{-i}}, \mathbf{d_{-j}}$ are related as desired.  
Given Theorem \ref{thm:partition}, we can attain a nontrivial approximation for the ratio of the number of graphs that realize the  degree sequence $\mathbf{d_{-i}}$ compared to the number of graphs that realize the degree sequence $\mathbf{d_{-j}}$.  To do so, we will need to work with $X_k(i,j;\mathbf{d_{-i}})$ and $X_k(i,j;\mathbf{d_{-j}})$.  But note that since $\mathbf{d_{-i}}, \mathbf{d_{-j}}$ are both defined from the same base sequence $\mathbf{d}$ and have the same in-degree sum, these two $X_k$ are identical; hence, we will simply refer to this set as $X_k$ in the proof below, and we will similarly drop the arguments of $X_k$ in most subsequent parts of the paper.

\begin{cor} If $a_{i}\geq a_{j}>0$, then 
$$\|G_{\mathbf{d_{-i}}}\|\geq \frac{a_{i}}{a_{j}}\|G_{\mathbf{d_{-j}}}\|.$$
\end{cor}
\begin{proof}
First define the variable $\delta$ such that if $a_{j}+1 \leq a_{i}$, then $\delta=1$ and otherwise, $\delta=0$.  The statement of the Corollary holds trivially if $\mathbf{d_{-j}}$ is not graphic.  If $\mathbf{d_{-j}}$ is graphic, then so is $\mathbf{d_{-i}}$ and we can 
apply Theorem 1 and note that  $$\|G_{\mathbf{d_{-i}}}\|=\sum_{k=0}^{a_{j}-1+\delta}\binom{(a_{i}-1)+a_{j}-2k}{a_{j}-1+\delta-k}\|G_{X_{k}}\|\geq \sum_{k=0}^{a_{j}-1}\binom{a_{i}+a_{j}-2k-1}{a_{j}-k}\|G_{X_{k}}\|.$$  Similarly, $\|G_{\mathbf{d_{-j}}}\|=\sum_{k=0}^{a_{j}-1}\binom{a_{i}+a_{j}-2k-1}{a_{j}-k-1}\|G_{X_{k}}\|$.

If we show that for all relevant natural numbers $k$, 
\[
\binom{a_{i}+a_{j}-2k-1}{a_{j}-k}\geq \frac{a_{i}}{a_{j}}\binom{a_{i}+a_{j}-2k-1}{a_{j}-k-1},\]
then the result will follow.
Note that 
\[
\binom{a_{i}+a_{j}-2k-1}{a_{j}-k}/\binom{a_{i}+a_{j}-2k-1}{a_{j}-k-1}=\frac{a_{i}-k}{a_{j}-k} \geq \frac{a_i}{a_j},
\]
since $a_i \geq a_j$, 
and the proof is complete.

\end{proof}

We next seek a compact expression for $\|G_{\mathbf{d_{-i}}}\|/\|G_{\mathbf{d_{-j}}}\|$ that will enable us to analyze this ratio with minimal difficulty without having to explicitly worry about the number of terms in the formulas for $\|G_{\mathbf{d_{-i}}}\|$ or $\|G_{\mathbf{d_{-j}}}\|$ found in the summation in Theorem 1.  Before we do that we introduce two additional notational conventions: 
\begin{equation}
\label{eq:convention}
\Pi_{k=r}^{r-1}\omega_{k}=1 \; \; \; \mbox{and} \; \; \; \Pi_{k=0}^{0}\omega_{k}=\omega_{0}.
\end{equation}

\begin{cor}
\begin{equation}
\label{eq:cor2o}
\frac{\|G_{\mathbf{d_{-i}}}\|}{\|G_{\mathbf{d_{-j}}}\|}=\frac{a_{i}}{a_{j}} \left(\frac{\sum_{k=0}^{a_j}\Pi_{l=0}^{k-1}(a_{j}-l)\Pi_{l=1}^{k}(a_{i}-l)\eta_{k}}{\sum_{k=0}^{a_j-1}\Pi_{l=1}^{k}(a_{j}-l)\Pi_{l=0}^{k-1}(a_{i}-l)\eta_{k}} \right)
\end{equation}
where 
\begin{equation}
\label{eq:etak}
\eta_{k}=\frac{\|G_{X_{k}}\|}{[\Pi_{l=0}^{2k-1}(a_{i}+a_{j}-l-1)]\|G_{X_{0}}\|}
\end{equation}
when $k\leq \lfloor\frac{a_{i}+a_{j}-1}{2}\rfloor$, and $\eta_{a_j} \equiv \eta_{a_i} =1$ when $k=a_j$ if $a_j=a_i$.
% is any arbitrary finite number otherwise. \textcolor{blue}{It would be nice to be able to write this without the arbitrary $\eta_k$ aspect.}
\end{cor}

\begin{proof}
Without loss of generality, assume $a_{j}<a_{i}-1$.  That is, if $a_j \in \{ a_i, a_i-1 \}$, then we can adjust the calculations with $\delta=1$ as in the proof of Corollary 1 and they will go through; furthermore, if $a_j=a_i$, then the $k=a_j=a_i$ term in the numerator of the right hand side of (\ref{eq:cor2o}) is 0, so we can set $\eta_{a_i=a_j}=1$  (or any finite number) and the result will still hold.
 
By invoking Theorem 1, using $X_k = X_k(i,j; \mathbf{d_{-i}})=X_k(i,j;\mathbf{d_{-j}})$, and dividing the numerator and denominator by $\|G_{X_{0}}\|$, we obtain 

\begin{equation}
\label{eq:doubleratio}
\begin{array}{ll}
\|G_{\mathbf{d_{-i}}}\| / \|G_{\mathbf{d_{-j}}}\|  =  \\ \\
\left( \sum_{k=0}^{a_{j}} \binom{a_{i}+a_{j}-2k-1}{a_{j}-k} \|G_{X_{k}}\| /  \|G_{X_{0}}\| \right) / 
 \left(\sum_{k=0}^{a_{j}-1} \binom{a_{i}+a_{j}-2k-1}{a_{j}-k-1} \|G_{X_{k}}\|/\|G_{X_{0}}\| \right).
\end{array}
\end{equation}
$\newline$
Now we multiply both the numerator and denominator by $\frac{a_{i}!a_{j}!}{(a_{i}+a_{j}-1)!}$, which yields
\begin{equation}
\begin{array}{ll}
\|G_{\mathbf{d_{-i}}}\| / \|G_{\mathbf{d_{-j}}}\| = \\ \\ \left( \sum_{k=0}^{a_{j}}\frac{\Pi_{l=0}^{k-1}(a_{j}-l)\Pi_{l=0}^{k}(a_{i}-l)\|G_{X_{k}}\|}{\Pi_{l=0}^{2k-1}(a_{i}+a_{j}-l-1)\|G_{X_{0}}\|}\right) / \left(\sum_{k=0}^{a_{j}-1}\frac{\Pi_{l=0}^{k}(a_{j}-l)\Pi_{l=0}^{k-1}(a_{i}-l)\|G_{X_{k}}\|}{\Pi_{l=0}^{2k-1}(a_{i}+a_{j}-l-1)\|G_{X_{0}}\|} \right).
\end{array}
\end{equation}
Substituting in $\eta_{k}$ as defined in equation (\ref{eq:etak}), we obtain
\begin{equation}
\begin{array}{ll}
\|G_{\mathbf{d_{-i}}}\|/\|G_{\mathbf{d_{-j}}}\|= \\ \\ \frac{\sum_{k=0}^{a_{j}}\Pi_{l=0}^{k-1}(a_{j}-l)\Pi_{l=0}^{k}(a_{i}-l)\eta_{k}}{\sum_{k=0}^{a_{j}-1}\Pi_{l=0}^{k}(a_{j}-l)\Pi_{l=0}^{k-1}(a_{i}-l)\eta_{k}}=\frac{a_{i}}{a_{j}} \left( \frac{\sum_{k=0}^{a_{j}}\Pi_{l=0}^{k-1}(a_{j}-l)\Pi_{l=1}^{k}(a_{i}-l)\eta_{k}}{\sum_{k=0}^{a_{j}-1}\Pi_{l=1}^{k}(a_{j}-l)\Pi_{l=0}^{k-1}(a_{i}-l)\eta_{k}} \right),
%\end{equation}
\end{array}
\end{equation}
which agrees with equation (\ref{eq:cor2o}).
\end{proof}

\begin{rem}
\label{rem:eta}
Obviously the right hand side of (\ref{eq:cor2o}) does not change if we add zero terms to the sums in the numerator and denominator.
The summation in the numerator of  this expression can be rewritten using
 $$\sum_{k=0}^{a_{j}}\Pi_{l=0}^{k-1}(a_{j}-l)\Pi_{l=1}^{k}(a_{i}-l)\eta_{k}=\sum_{k=0}^{a_{j}+1}\Pi_{l=0}^{k-1}(a_{j}-l)\Pi_{l=1}^{k}(a_{i}-l)\eta_{k},$$ which follows because $$\sum_{k=0}^{a_{j}+1}\Pi_{l=0}^{k-1}(a_{j}-l)\Pi_{l=1}^{k}(a_{i}-l)\eta_{k}-\sum_{k=0}^{a_{j}}\Pi_{l=0}^{k-1}(a_{j}-l)\Pi_{l=1}^{k}(a_{i}-l)\eta_{k}=$$ $$\Pi_{l=0}^{a_{j}}(a_{j}-l)\Pi_{l=1}^{a_{j}+1}(a_{i}-l)\eta_{a_{j}+1}=0.$$
Inductively, it follows that $$\sum_{k=0}^{a_{j}}\Pi_{l=0}^{k-1}(a_{j}-l)\Pi_{l=1}^{k}(a_{i}-l)\eta_{k}=\sum_{k=0}^{\infty}\Pi_{l=0}^{k-1}(a_{j}-l)\Pi_{l=1}^{k}(a_{i}-l)\eta_{k}.$$ 
Analogously, $$\sum_{k=0}^{a_{j}-1}\Pi_{l=1}^{k}(a_{j}-l)\Pi_{l=0}^{k-1}(a_{i}-l)\eta_{k}=\sum_{k=0}^{\infty}\Pi_{l=1}^{k}(a_{j}-l)\Pi_{l=0}^{k-1}(a_{i}-l)\eta_{k}.$$ 
Therefore, we can extend the statement of Corollary 2 to read
\begin{equation}
\label{eq:cor2}
\frac{\|G_{\mathbf{d_{-i}}}\|}{\|G_{\mathbf{d_{-j}}}\|}=\frac{a_{i}}{a_{j}} \left(\frac{\sum_{k=0}^{\infty}\Pi_{l=0}^{k-1}(a_{j}-l)\Pi_{l=1}^{k}(a_{i}-l)\eta_{k}}{\sum_{k=0}^{\infty}\Pi_{l=1}^{k}(a_{j}-l)\Pi_{l=0}^{k-1}(a_{i}-l)\eta_{k}} \right)
\end{equation}
where $\eta_k$ is defined in the statement of Corollary 2 for $k\leq \lfloor\frac{a_{i}+a_{j}-1}{2}\rfloor$ and $\eta_k=1$ (or any finite number) for all other $k$.  Expression (\ref{eq:cor2}) will be useful for providing flexibility in subsequent calculations.
\end{rem}

The results we have proved so far apply to any graphic degree sequence.  Under additional assumptions, at the cost of generality, we can go farther and obtain a power series representation for the ratio $\|G_{\mathbf{d_{-i}}}\|/\|G_{\mathbf{d_{-j}}}\|$.   Specifically, if our graph is in some sense ``sparse",  equation (\ref{eq:cor2}) suggests that $\|G_{\mathbf{d_{-i}}}\| / \|G_{\mathbf{d_{-j}}}\|=a_{i}/a_{j}+O(\epsilon)$ for $\epsilon$  ``small".   We will elaborate on this idea in the following section.

\section{Degree Preserving Switches to Count Graphs}
We will use the idea of {\em degree preserving switches} to estimate the likelihood that two nodes have a common neighbor.  This estimation will help us to derive a power series expansion of  (\ref{eq:cor2}) that is valid if we let the number of nodes, or correspondingly the number of elements in each degree sequence, be sufficiently large.   To obtain this result, it is important to set notation to specify how degrees behave as sequence length grows.
\begin{mydef}
Consider a sequence of degree sequences, $\{ \mathbf{d}^N \}_{N\in\mathbb{N}}$, where each $\mathbf{d}^N  \in\mathbb{Z}^{N\times 2}$.  We say that $d_{max}( \{ \mathbf{d}^N \}_{N\in\mathbb{N}})=O(S^{p})$, where $p\in\mathbb{R}$ and $S=S(\mathbf{d}^{N})$ is the sum of the edges for the degree sequence $\mathbf{d}^{N}$, if and only if there exists a fixed constant $C\in\mathbb{R}$ such that $\varlimsup_{N\rightarrow\infty}\frac{max(\mathbf{d}^N)}{S^{p}} \leq C$, where the maximum is taken over all components of $\mathbf{d}^N$ and  $\varlimsup$ denotes the limit supremum.  
\end{mydef}  

For notational simplicity, we will omit explicit reference to the sequence of degree sequences when we write $d_{\max}=O(S^{p})$.  The use of the $O(S)$ notation is meant to indicate that we are not referring to the maximum degree of a fixed degree sequence.   Analogously, we can say that for a sequence of degree sequences, the total number of nodes is $O(S^{p})$ for some $p\in\mathbb{R}$. In addition, we say that the $i$th node has bounded (in)-degree in the limit of a sequence of degree sequences, if $\varlimsup_{N\rightarrow\infty}(\mathbf{a}^{N})_{i}\leq C$, where $C\in\mathbb{R}$ and $(\mathbf{a}^{N})_{i}$ denotes the $i$th element of the vector $\mathbf{a}^{N}$. At this juncture, we are ready to prove a result about the asymptotic likelihood
of obtaining a graph, based on a uniform sampling of graphs realizing a degree sequence, in which two fixed nodes form edges with a common target node.

\begin{figure}[h!]
\begin{center}
\includegraphics[scale=0.25]{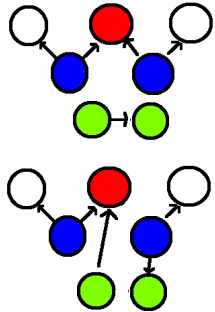}
\end{center}
\caption{An example of the type of degree preserving switch described in the proof of Theorem 2.  The top image is the input graph (i.e., the graph before the switch is applied) and the bottom is the output graph (generated by the switch). Note that in the output graph, all nodes have the same in and out degrees as in the input graph but the  two blue nodes no longer have a common neighbor.} \label{fig:switch}
\end{figure}
\begin{thm}
Consider a sequence of degree sequences such that \textcolor{black}{$d_{max}=O(S^{\frac{1}{2}-\tau})$} for some $\tau>0$. Pick an arbitrary node $x$ and another node $y$  such that $y$ has a bounded number of edges in the limit of the sequence of degree sequences.  In the limit of the sequence of degree sequences, the ratio of the number of realizations of graphs where there does not exist a node that receives an outward edge from (or supplies an inward to) both $x$ and $y$ relative to the number of graphs where exactly one node receives an outward edge from (or supplies an inward edge to) both $x$ and $y$  is $O(S^{2\tau})$.  
\end{thm}
\begin{proof}
We will focus on the case of outward edges from $x$ and $y$, as the inward case is analogous. Consider a realization of a fixed degree sequence where a node (red in Figure~\ref{fig:switch}) receives outward edges from nodes $x$ and $y$ (both blue).  We can then define a degree preserving switch by choosing a remaining edge whose end points are not a red or blue node and redirecting it to the red node.  Then to preserve the degree, we replace the edge that connects the right blue node ($x$) to the red node with an edge from $x$ to the node that is now missing an edge.  Unfortunately, this operation is not 1:1, since different input graphs can yield the same output graph, so we need to account for such repeats in order to accurately count the likelihood of a node receiving edges from both $x$ and $y$.

Let us refer to the nodes receiving edges from a fixed node as the out-neighbors of that node and to the nodes supplying edges to a fixed node as the in-neighbors of that node.  In the  output graph of the degree preserving switch operator in Figure \ref{fig:switch}, denote the out-neighbors of node $x$ by $\mathbf{S}=\{ s_{1},...,s_{l}\}$ and the out-neighbors of node $y$ by $\mathbf{T}=\{t_{1},...,t_{m}\}$, where $l,m$ are the out-degrees of nodes $x,y$, respectively.  Since the desired result only pertains to input graphs where $x$ and $y$ share exactly one common out-neighbor,  we have $\mathbf{S}\cap \mathbf{T}=\emptyset$.  Importantly, for each input graph that maps to this output graph, $x$ must have $l-1$ out-neighbors that are still its out-neighbors in the output graph, while the out-neighbors of $y$ in the input and output graphs are the same. 

Without loss of generality, suppose that in the input graph, $t_{1}$ is the node that has edges from both $x$ and $y$.  In the output graph denote the in-neighbors of $t_{1}$ as $\mathbf{U}=\{u_1,....,u_n\}$ where $n$ denotes the in-degree of $t_{1}$.  In an input graph that is mapped by the degree preserving switch to the desired output graph, $t_{1}$ must have edges with $n-1$ of the nodes in $U$.  Naturally, there are $\binom{n-1}{1}$ ways for this to happen.  Since $l-1$ of the out-neighbors of $x$ in the input graph must also be out-neighbors of $x$ in the output graph, we conclude that there are fewer than $nl=O(S^{1-2\tau})$ ways to generate the same output if $t_{1}$ is the node that has edges with both $x$ and $y$ in the input graph.

We now repeat the same argument for each out-neighbor of $y$ in the output graph.  Since $y$ has bounded degree, it follows that under the  degree preserving operation, there are at most $O(S^{1-2\tau})$ ways to generate the same output.  

Now that we know that our degree preserving operation is a $O(S^{1-2\tau}):1$ function, to finish off the proof, we need to identify how many degree preserving switches are possible from a single graph.  The  edges that are eligible to be switched connect nodes such that the source is not an in-neighbor of the common out-neighbor of $x$ and $y$ and the target does not already receive an edge from $x$ (see Figure \ref{fig:switch}). 
Since $x$ has degree at most $O(S^{\frac{1}{2}-\tau})$, that means the total number of edges corresponding to the neighbors of $x$ is at most $O(S^{1-2\tau})$.   Analogously, the common out-neighbor of $x$ and $y$ has  degree at most $O(S^{\frac{1}{2}-\tau})$ and the number of edges corresponding to its in-neighbors  is $O(S^{1-2\tau})$. Consequently, we have $O(S)-O(S^{1-2\tau})=O(S)$ edges that we can choose from for the degree preserving switch operator to switch.  Hence, for every graph where $x$ and $y$ have a common out-neighbor, the number of unique graphs where they do not have a common out-neighbor is at least the ratio of $O(S)$ to the number of input graphs that can map to each output graph, namely $O(S)/O(S^{1-2\tau})=O(S^{2\tau})$, which is our desired result. \end{proof}

\begin{cor} \label{cor:ratio}
Fix a sequence of degree sequences such that $d_{max}=O(S^{\frac{1}{2}-\tau})$.  Let $k\in\mathbb{N}$. For any two nodes $x$ and $y$, where $y$ has bounded degree in the limit of the sequence of degree sequences, the ratio of the number of graphs where $k$ nodes receive edges from both $x$ and $y$ compared to the number of graphs where $k+1$ nodes receive edges from both $x$ and $y$ is $O(S^{2\tau})$.
\end{cor}
\begin{proof}
Perform the same switch technique as in Theorem 2 on a particular node that has an edge with both $x$ and $y$, and the result follows analogously.
\end{proof}

Next, we apply Corollary \ref{cor:ratio} to begin to expand the terms in equation (\ref{eq:cor2o}).
\begin{cor}
\label{cor:expand}
 If $d_{max}=O(S^{\frac{1}{2}-\tau})$, then $$\|G_{\mathbf{d_{-i}}}\|= \frac{a_{i}}{a_{j}}\|G_{\mathbf{d_{-j}}}\|[1+O(S^{-2\tau})].$$
\end{cor}
\begin{proof}
First suppose that either node $i$ or node $j$ has bounded degree. 
By Corollary 3, we know that for all $k\geq 1$, 
\[
 \frac{\binom{a_{i}+a_{j}-2k-1}{a_{j}-k} \|G_{X_{k}}\|}{\binom{a_{i}+a_{j}-1}{a_{j}}  \|G_{X_{0}}\|}=O(S^{-2k\tau})  \; \; \mbox{and} \; \;   \frac{\binom{a_{i}+a_{j}-2k-1}{a_{j}-k-1} \|G_{X_{k}}\|}{\binom{a_{i}+a_{j}-1}{a_{j}-1}  \|G_{X_{0}}\|}=O(S^{-2k\tau}),
 \]
  as these terms  represent the number of graphs with $k$ common neighbors of the two nodes $i$ and $j$ divided by the number of graphs where nodes $i$ and $j$ have no common neighbors.

 So we conclude from equation (\ref{eq:doubleratio}) in the proof of Corollary 2 that 

\[
\begin{array}{rcl} \frac{\|G_{\mathbf{d_{-i}}}\|}{\|G_{\mathbf{d_{-j}}}\|} &=& \frac{\frac{\binom{a_{i}+a_{j}-1}{a_{j}}}{\binom{a_{i}+a_{j}-1}{a_{j}-1}}+\sum_{k=1}^{a_{j}}\frac{\binom{a_{i}+a_{j}-2k-1}{a_{j}-k}}{\binom{a_{i}+a_{j}-1}{a_{j}-1}}\frac{\|G_{X_{k}}\|}{\|G_{X_{0}}\|}}{1+\sum_{k=1}^{a_{j}-1}\frac{\binom{a_{i}+a_{j}-2k-1}{a_{j}-k-1}}{\binom{a_{i}+a_{j}-1}{a_{j}-1}}\frac{\|G_{X_{k}}\|}{\|G_{X_{0}}\|}}=\frac{\frac{\binom{a_{i}+a_{j}-1}{a_{j}}}{\binom{a_{i}+a_{j}-1}{a_{j}-1}}+\sum_{k=1}^{a_{j}}\frac{\binom{a_{i}+a_{j}-2k-1}{a_{j}-k}}{\binom{a_{i}+a_{j}-1}{a_{j}-1}}\frac{\|G_{X_{k}}\|}{\|G_{X_{0}}\|}}{1+O(S^{-2\tau})} \vspace{0.1in} \\

&=& \frac{\binom{a_{i}+a_{j}-1}{a_{j}}}{\binom{a_{i}+a_{j}-1}{a_{j}-1}}\left[ \frac{1+\sum_{k=1}^{a_{j}}\frac{\binom{a_{i}+a_{j}-2k-1}{a_{j}-k}}{\binom{a_{i}+a_{j}-1}{a_{j}}}\frac{\|G_{X_{k}}\|}{\|G_{X_{0}}\|}}{1+O(S^{-2\tau})}\right] =\frac{a_{i}}{a_{j}}[\frac{1+O(S^{-2\tau})}{1+O(S^{-2\tau})}]=\frac{a_{i}}{a_{j}}[1+O(S^{-2\tau})].
\end{array}
\]

To extend this relationship to the general case where nodes $i$ and $j$ both have degree $O(S^{\frac{1}{2}-\tau})$, we consider the degree sequences $\mathbf{d^{(i)}}=(\{\mathbf{a},1\}-\mathbf{e_{i}}, \{\mathbf{b},0\}), 
\mathbf{d^{(j)}}=(\{\mathbf{a},1\}-\mathbf{e_{j}},\{\mathbf{b},0\}) \in\mathbb{Z}^{(N+1)\times 2}$, for which 

\begin{equation}
\label{eq:ratio}
\frac{\|G_{\mathbf{d_{-i}}}\|}{\|G_{\mathbf{d_{-j}}}\|}=\frac{\|G_{\mathbf{d^{(i)}_{-(N+1)}}}\|}{\|G_{\mathbf{d^{(i)}_{-j}}}\|}\frac{\|G_{\mathbf{d^{(j)}_{-i}}}\|}{\|G_{\mathbf{d^{(j)}_{-(N+1)}}}\|},
\end{equation}
 where node $N+1$ has bounded degree and $\mathbf{d^{(j)}_{-i}}=\mathbf{d^{(i)}_{-j}}$. Consequently,   \[
 \frac{\|G_{\mathbf{d_{-i}}}\|}{\|G_{\mathbf{d_{-j}}}\|}=\frac{a_{i}}{a_{j}}[1+O(S^{-2\tau})].
 \]
\end{proof}

We conclude this section noting that the proof technique can be extended to more general cases.  We can relax the constraint that $d_{max}=O(S^{\frac{1}{2}-\tau})$ and still attain that two nodes in general should not share common neighbors.  For more details we refer the reader to Appendix A.  On a similar note, the notion that a node of bounded degree and an arbitrary node should not share common neighbors extends to cases beyond directed graphs with loops, including directed and undirected graphs without loops.  We discuss in Appendix B how to extend these arguments to these cases, noting that similar ideas also apply to graphs with other edges prohibited besides loops.  Section 4 explains how to iteratively attain more refined approximations from the relatively crude approximation given by Corollary 4.  These results also generalize  in ways that are discussed in the appendices.  

\section{Asymptotic Enumeration for Arbitrary Orders of Accuracy}
Since we have established some fundamental results in the prior section, we can now derive our general asymptotic enumeration.

We start by establishing an alternative expression for  $\|G_{\mathbf{d_{-i}}}\|/\|G_{\mathbf{d_{-j}}}\|$.
For this statement, define $X_{0_{i}}$ as the set of all residual degree sequences $(\mathbf{a}-a_{i}\mathbf{e_{i}}-\mathbf{e_{j}},\mathbf{b-s})$ with $\mathbf{s}$ constructed by removing one (outgoing) edge from each of  $a_{i}$ nodes; define $X_{1_{i}}$ as the set of all residual degree sequences constructed by removing two outgoing edges from one node and one edge from $a_{i}-2$ nodes; and define $X_{0_{j}}$,$X_{1_{j}}$ analogously 
from $(\mathbf{a}-\mathbf{e_{i}}-a_j\mathbf{e_j},\mathbf{b-s})$.

\begin{cor} \label{cor5}
\begin{equation}
\label{eq:cor5}
\frac{\|G_{d_{-i}}\|}{\|G_{d_{-j}}\|}=\frac{a_{i}}{a_{j}}\left( \frac{1+\|G_{X_{1_{i}}}\|/(a_{i}\|G_{X_{0_{i}}}\|)}{1+ \|G_{X_{1_{j}}}\|/(a_{j}\|G_{X_{0_{j}}}\|)} \right)
\end{equation}
\end{cor}

\begin{proof}
The proof is in the spirit of Corollary 4.  
Let $\mathbf{d^{(i)}}= ( \{\mathbf{a},1\}-\mathbf{e_{i}}, \{\mathbf{b},0\}), \mathbf{d^{(j)}}=( \{\mathbf{a},1\}-\mathbf{e_{j}} , \{\mathbf{b},0\}) \in\mathbb{Z}^{(N+1)\times 2}$.  

We previously established equation (\ref{eq:ratio}).

But applying Theorem 1 yields $$\frac{\|G_{\mathbf{d^{(i)}_{-(N+1)}}}\|}{\|G_{\mathbf{d^{(i)}_{-j}}}\|}=\frac{\|G_{X_{0_{j}}}\|}{a_{j}\|G_{X_{0_{j}}}\|+\|G_{X_{1_{j}}}\|}=\frac{1}{a_{j}}\left( \frac{1}{1+\|G_{X_{1_{j}}}\|/(a_{j}\|G_{X_{0_{j}}}\|)} \right). $$
Similarly, $$\frac{\|G_{\mathbf{d^{(j)}_{-i}}}\|}{\|G_{\mathbf{d^{(j)}_{-(N+1)}}}\|}=\frac{a_{i}\|G_{X_{0_{i}}}\|+\|G_{X_{1_{i}}}\|}{\|G_{X_{0_{i}}}\|}=a_{i}({1+\frac{\|G_{X_{1_{i}}}\|}{a_{i}\|G_{X_{0_{i}}}\|}}).$$
\end{proof}

Corollary \ref{cor5} will be useful to us when combined with the observation that for an arbitrary degree sequence $\mathbf{m}$, 
\begin{equation}
\label{eq:GX}
\frac{\|G_{X_{1_{i}}}\|}{\|G_{X_{0_{i}}}\|}=\frac{\sum_{\mathbf{x}\in X_{1_{i}}}\|G_{\mathbf{x}}\|}{\sum_{\mathbf{x}\in X_{0_{i}}}\|G_{\mathbf{x}}\|}=\frac{\sum_{\mathbf{x}\in X_{1_{i}}} \|G_{\mathbf{x}}\|/\|G_{\mathbf{m}}\|} {\sum_{\mathbf{x}\in X_{0_{i}}} \|G_{\mathbf{x}}\|/\|G_{\mathbf{m}}\|}.
\end{equation} 
That is, equation (\ref{eq:cor5}) represents a recursion that expresses the ratio of the number of graphs of two different degree sequences as a function of the ratios of the numbers of graphs of various other degree sequences.  

We now state the first of three theorems that will enable us to reach the desired asymptotic enumeration results of arbitrary order.  

\begin{thm}
\label{thm:arborder1}
Let $\mathbf{d}=(\mathbf{a},\mathbf{b})$ with $\sum a_n= \sum b_n \pm 1$ and  
$d_{max} = O(S^{\frac{1}{2}-\tau})$.  Consider an approximation that satisfies the equation 
$$\frac{\|G_{\mathbf{d_{-i}}}\|}{\|G_{\mathbf{d_{-j}}}\|}=f(\mathbf{e}_{i},\mathbf{e}_{j},\mathbf{d},\mathbf{\sigma})(1+O(S^{-2w\tau}))$$ 
where $w\geq 1$, $2w\tau\leq \gamma$ (\textcolor{black}{we define $\gamma$ below}),
and the last argument $\mathbf{\sigma}$ either equals $\mathbf{a}$, which denotes that $\sum a_n= \sum b_n + 1$ and the two degree sequences in the ratio differ in their in-degree sequences (i.e., the in-degree sequence is being used to define $\mathbf{d_{-i}}$), or $\mathbf{\sigma}$  equals $\mathbf{b}$, which has the analogous connotation with respect to out-degree sequences, with $\sum a_n= \sum b_n - 1$. Assume that $f(\mathbf{e}_{i},\mathbf{e}_{j},\mathbf{d},\mathbf{\sigma})=h(\mathbf{e}_i,\mathbf{d},\mathbf{\sigma})/h(\mathbf{e}_j,\mathbf{d},\mathbf{\sigma})$ for some function $h$.
Furthermore, suppose that for $m=O(S^{\frac{1}{2}-\tau})$, 
\begin{equation}
\label{eq:assumption}
h(\mathbf{e}_{i},\mathbf{d}_1,\mathbf{\sigma}_1)= h(\mathbf{e}_{i},\mathbf{d}_0,\mathbf{\sigma}_0)(1+O(S^{-\gamma}))
\end{equation}
where $\|\mathbf{d}_1-\mathbf{d}_0\|_1\leq m$, $\|\mathbf{d}_1-\mathbf{d}_0\|_\infty\leq 1$, $\sigma_i$ is either $\mathbf{a}_i$ or $\mathbf{b}_i$ (i.e., the in- or out-degree sequence of $\mathbf{d}_i$), and the following equalities of dot products hold:  $\sigma_1 \cdot \mathbf{e}_i = \sigma_0 \cdot \mathbf{e}_i$, $\sigma_1 \cdot  \mathbf{e}_j = \sigma_0 \cdot \mathbf{e}_j$.  
If $\sigma=\mathbf{a}$, then there exists a sharper approximation 
$$\frac{\|G_{\mathbf{d_{-i}}}\|}{\|G_{\mathbf{d_{-j}}}\|}=g(\mathbf{e}_{i},\mathbf{e}_{j},\mathbf{d},\mathbf{a})(1+O(S^{-2(w+1)\tau}))$$
 where
$$g(\mathbf{e}_{i},\mathbf{e}_{j},\mathbf{d},\mathbf{a})=\frac{a_{i}}{a_{j}}\exp(\log(1+\frac{(a_{i}-1)\sum_{x_{1}\neq...\neq x_{a_{i-1}}=x_{a_i}} \Pi_{k=1}^{a_i} f(\mathbf{e}_{x_{k}},\mathbf{e}_{u_{k}},\mathbf{d}_{k,i,j},\mathbf{b}_k) }{\sum_{x_{1}\neq...\neq x_{a_{i-1}}\neq x_{a_i}} \Pi_{k=1}^{a_i} f(\mathbf{e}_{x_{k}},\mathbf{e}_{u_{k}},\mathbf{d}_{k,i,j},\mathbf{b})})-$$ $$\log(1+\frac{(a_{j}-1)\sum_{x_{1}\neq...\neq x_{a_{j-1}}=x_{a_j}} \Pi_{k=1}^{a_j} f(\mathbf{e}_{x_{k}},\mathbf{e}_{u_{k}},\mathbf{d}_{k,j,i},\mathbf{b}) }{\sum_{x_{1}\neq...\neq x_{a_{j-1}}\neq x_{a_j}} \Pi_{k=1}^{a_j} f(\mathbf{e}_{x_{k}},\mathbf{e}_{u_{k}},\mathbf{d}_{k,j,i},\mathbf{b})}))$$
for any arbitrary  choice of indices $\{ u_{k} \}$, where $\mathbf{d}_{k,i,j} = (\mathbf{a}-a_i\mathbf{e}_i-\mb{e}_j,\mathbf{b}-\sum_{j=1}^{a_i-k+1}\mathbf{e}_{u_j}-\sum_{j=1}^{k-1}\mathbf{e}_{x_j})$. A similar sharpened approximation, with $g$ depending on $\mathbf{b}$, holds if $\sigma=\mathbf{b}$.
\end{thm}
We will postpone the motivation for the assumptions that $f(\mathbf{e}_{i},\mathbf{e}_{j},\mathbf{d},\mathbf{\sigma})=h(\mathbf{e}_i,\mathbf{d},\mathbf{\sigma})/h(\mathbf{e}_j,\mathbf{d},\sigma)$ for some function $h$ and that perturbing  the third argument of $f$ (the degree sequence) only results in relatively small changes in $h$ until Section 6.  In Section 5, when we use the above theorem to explicitly compute asymptotics of the ratio $\| G_{\mathbf{d_{-i}}}\|/\|G_{\mathbf{d_{-j}}}\|$, we will be able to verify these assumptions directly.  Even at this stage, we already know that the approximation $\|G_{\mathbf{d_{-i}}}\|/\|G_{\mathbf{d_{-j}}}\|=\frac{a_i}{a_j}(1+O(S^{-2\tau}))$ does not depend on the degrees of the  nodes in the degree sequence other than nodes $i$ and $j$.  Similarly, this approximation can be expressed in terms of a decomposition like that assumed in Theorem \ref{thm:arborder1}, given by $\|G_{\mathbf{d_{-i}}}\|/\|G_{\mathbf{d_{-j}}}\|=[h(\mathbf{e}_i,\mathbf{d},\mathbf{a})/h(\mathbf{e}_j,\mathbf{d},\mathbf{a})](1+O(S^{-2\tau}))$, where $h(\mathbf{e}_i,\mathbf{d},\mathbf{a})=a_i$. We now proceed with the proof. 

\begin{proof}
Consider $\mathbf{d}$ such that $\sum a_n = \sum b_n + 1$.
We know from equation (\ref{eq:cor5}) that $$\frac{\|G_{\mathbf{d_{-i}}}\|}{\|G_{\mathbf{d_{-j}}}\|}=\frac{a_{i}}{a_{j}}\exp(\log(1+\frac{\|G_{X_{1_i}}\|}{a_{i}\|G_{X_{0_i}}\|})-\log(1+\frac{\|G_{X_{1_j}}\|}{a_{j}\|G_{X_{0_j}}\|})).$$
Let $$\|G_{\mathbf{d_{-i}}}\|/\|G_{\mathbf{d_{-j}}}\| = f(\mathbf{e}_{i},\mathbf{e}_{j},\mathbf{d},\mathbf{a})(1+O(S^{-2w\tau}))$$ for some $w \geq 1$. Our goal is to show that substituting the decomposition for $f$ into equation  (\ref{eq:cor6})  yields a sharper approximation $g$, as specified in the theorem statement (the proof with $\sigma=\mathbf{b}$ is analogous).

To prove this claim rigorously, 
we show that using $f$ to approximate $\frac{\|G_{X_{1_i}}\|}{a_{i}\|G_{X_{0_i}}\|}$  can allow us to derive  an improved  approximation of $\frac{\|G_{d_{-i}}\|}{\|G_{d_{-j}}\|}$, in a sense that will be made precise later.   Since $\frac{\|G_{X_{1_i}}\|}{a_{i}\|G_{X_{0_i}}\|}$ is equivalent to $\frac{\|G_{X_{1_j}}\|}{a_{j}\|G_{X_{0_j}}\|}$, we can carry over our results to the latter expression to complete the derivation.

For any $\mb{u} \in X_{0_i}$, equation (\ref{eq:GX}) yields
\begin{equation}
\label{eq:arbderive2}
\frac{\|G_{X_{1_i}}\|}{a_{i}\|G_{X_{0_i}}\|}=\frac{\sum_{\mathbf{x}\in X_{1_{i}}} \|G_{\mathbf{x}}\|/\|G_{\mathbf{u}}\|}{a_i\sum_{\mathbf{x}\in X_{0_{i}}} \|G_{\mathbf{x}}\|/\|G_{\mathbf{u}}\|}.
\end{equation}
Now, for any $\mathbf{x}\in X_{0_i}\cup X_{1_i}$, the bidegree sequences $\mb{u}$ and $\mb{x}$ include 
 the same number of edges and identical in-degree sequences.
  By definition we can write $\mathbf{x}=(\mathbf{a}-a_i\mathbf{e}_i-\mb{e}_j, \mathbf{b}-\sum_{j=1}^{a_i}\mathbf{e}_{x_j})$ and $\mathbf{u}=(\mathbf{a}-a_i\mathbf{e}_i-\mb{e}_j,\mathbf{b}-\sum_{j=1}^{a_i}\mathbf{e}_{u_j})$.  
  Let $\mb{d}_1=\mb{u}, \mb{d}_{a_{i+1}}=\mb{x}$, and define intermediate degree sequences 
   $\mathbf{d}_k=(\mathbf{a}-a_i\mathbf{e}_i-\mb{e}_j, \mathbf{b}-\sum_{j=1}^{a_i-k+1}\mathbf{e}_{u_j}-\sum_{j=1}^{k-1}\mathbf{e}_{x_j})$ for each $k=2, \ldots, a_i$.  Note that the sum of the in-degrees equals the sum of the out-degrees in each $\mb{d}_k$, since by assumption, $\sum a_n = \sum b_n + 1$.

Letting 

\begin{equation}
\label{eq:arbderive1}
\phi(\mathbf{e_i},\mathbf{e_j},\mathbf{d},\sigma):=\|G_{\mb{d_{-i}}}\|/\|G_{\mb{d_{-j}}}\|
\end{equation}
for any choice of $\mathbf{d}$, these definitions imply that for our particular $\mathbf{d}$ under consideration, 
\begin{equation}
\label{eq:arbderive3} \frac{\|G_{\mathbf{x}}\|}{\|G_{\mathbf{u}}\|}=\Pi_{k=1}^{a_i}\frac{\|G_{\mathbf{d}_{k+1}}\|}{\|G_{\mathbf{d}_{k}}\|}=\Pi_{k=1}^{a_i} \phi(\mathbf{e}_{x_{k}},\mathbf{e}_{u_{a_i-k+1}},(\mathbf{a}-a_i\mb{e}_i-\mb{e}_j,\mathbf{b}-\sum_{j=1}^{k-1} \mathbf{e}_{x_{j}}-\sum_{j=1}^{a_i-k}\mathbf{e}_{u_{j}},\mathbf{b}),
\end{equation}
where now the sum of the in-degrees in the third argument of $\phi$ in (\ref{eq:arbderive1}) is {\em one less than} the sum of the out-degrees.  Recall that the final argument $\mb{b}$ of $\phi$ in (\ref{eq:arbderive1}) implies that in the numerator and denominator of the right hand side of (\ref{eq:arbderive1}), a degree is being subtracted off of the {\em out-degree sequence} of the bidegree sequence given in the third argument of $\phi$, which means that the in- and out-degree sequences in $\mathbf{d_{-i}}, \mathbf{d_{-j}}$ end up with the same sums.
 The particular components of the out-degree sequence from which a degree is being subtracted are specified in the first and second arguments of $\phi$ for the numerator and denominator, respectively. 

Now, define $\Delta_i$ to be the difference between  $\frac{\|G_{X_{1_i}}\|}{a_{i}\|G_{X_{0_i}}\|}$  evaluated using the exact ratio $\phi$ and the same quantity evaluated using the approximation $f$.  
Recall that each $\mathbf{x}\in X_{0_i}$ was defined by removing the $a_i$ incoming edges to node $i$ along with one outgoing edge from each of $a_i$ distinct nodes.  The number of resulting bidegree sequences is the same as the number of bidegree sequences in which the $a_i$ outgoing edges are directed to node $i$ instead of being removed.  There is an analogous equivalence for each $\mathbf{x}\in X_{1_i}$.
Hence, we can write $\Delta_i$ as

$$ \Delta_i = \frac{(a_{i}-1)\sum_{x_{1}\neq...\neq x_{a_{i-1}}=x_{a_i}} \Pi_{k=1}^{a_i} \phi(\mathbf{e}_{x_{k}},\mathbf{e}_{u_{a_i-k+1}},(\mathbf{a}-a_i\mb{e}_i-\mb{e}_j,\mathbf{b}-\sum_{j=1}^{k-1} \mathbf{e}_{x_{j}}-\sum_{j=1}^{a_i-k}\mathbf{e}_{u_{j}}),\mathbf{b}) }{\sum_{x_{1}\neq...\neq x_{a_{i-1}}\neq x_{a_i}} \Pi_{k=1}^{a_i} \phi(\mathbf{e}_{x_{k}},\mathbf{e}_{u_{a_i-k+1}},(\mathbf{a}-a_i\mb{e}_i-\mb{e}_j,\mathbf{b}-\sum_{j=1}^{k-1} \mathbf{e}_{x_{j}}-\sum_{j=1}^{a_i-k}\mathbf{e}_{u_{j}}),\mathbf{b})}-
$$
\begin{equation} \label{eq:delta} \frac{(a_{i}-1)\sum_{x_{1}\neq...\neq x_{a_{i-1}}=x_{a_i}} \Pi_{k=1}^{a_i} {f}(\mathbf{e}_{x_{k}},\mathbf{e}_{u_{a_i-k+1}},(\mathbf{a}-a_i\mb{e}_i-\mb{e}_j,(\mathbf{b}-\sum_{j=1}^{k-1} \mathbf{e}_{x_{j}}-\sum_{j=1}^{a_i-k}\mathbf{e}_{u_{j}}),\mathbf{b}) }{\sum_{x_{1}\neq...\neq x_{a_{i-1}}\neq x_{a_i}} \Pi_{k=1}^{a_i} {f}(\mathbf{e}_{x_{k}},\mathbf{e}_{u_{a_i-k+1}},(\mathbf{a}-a_i\mb{e}_i-\mb{e}_j,(\mathbf{b}-\sum_{j=1}^{k-1} \mathbf{e}_{x_{j}}-\sum_{j=1}^{a_i-k}\mathbf{e}_{u_{j}}),\mathbf{b})}.
\end{equation}

Since the choice for $u_{a_i-k+1}$ is arbitrary and assumption (\ref{eq:assumption}) in the theorem statement  holds, 
we can simplify the notation by using $f_k(x_k)$ in place of  the full expression for $f$.  (We will defer a technical point regarding this simplification to the end of the proof.)     In contrast , $\phi$ does depend on the degree sequence and on .$x_1,...,x_{k-1}$.   But for simplicity, we abuse notation and write $\phi_k(x_k)$.  This reduces to a more tractable (but slightly misleading) notation:

$$\Delta_i = \frac{(a_{i}-1)\sum_{x_{1}\neq...\neq x_{a_{i-1}}=x_{a_i}} \Pi_{k=1}^{a_i} \phi_k({x_{k}}) }{\sum_{x_{1}\neq...\neq x_{a_{i-1}}\neq x_{a_i}} \Pi_{k=1}^{a_i} \phi_k({x_{k}})}-\frac{(a_{i}-1)\sum_{x_{1}\neq...\neq x_{a_{i-1}}=x_{a_i}} \Pi_{k=1}^{a_i} {f_k}({x_{k}}) }{\sum_{x_{1}\neq...\neq x_{a_{i-1}}\neq x_{a_i}} \Pi_{k=1}^{a_i} {f_k}({x_{k}})}.$$

Denote  $D_{0}$ as the set of sets of $a_i$ distinct indices in $\{ 1, \ldots, N \}$ and $D_1$ as the set of sets of $a_i$ indices in $\{ 1, \ldots, N \}$ such that the first $a_{i}-2$ are distinct and the final two are equal. Writing $\Delta_i$ as a single fraction, we obtain 

\[
\begin{array}{lcr}
\Delta_i &=& \frac{\textstyle (a_{i}-1)[\sum_{D_1} \Pi_{k=1}^{a_i} \phi_k({x_{k}})\sum_{D_0} \Pi_{k=1}^{a_i} {f_k}({x_{k}})-\sum_{D_1} \Pi_{k=1}^{a_i} {f_k}({x_{k}})\sum_{D_0} \Pi_{k=1}^{a_i} \phi_k({x_{k}}) ]}{\textstyle \sum_{D_0} \Pi_{k=1}^{a_i} \phi_k({x_{k}})\sum_{D_0} \Pi_{k=1}^{a_i} {f_k}({x_{k}})} \vspace{0.1in} \\
&=& \frac{\textstyle (a_{i}-1)[\sum_{D_1,D_0} \left( \Pi_{x_k\in D_1} \phi_k({x_{k}})\Pi_{x_k\in D_0} {f_k}({x_{k}})- \Pi_{x_k\in D_1}{f_k}({x_{k}}) \Pi_{x_k \in D_0} \phi_k({x_{k}})\right) ]}{\textstyle \sum_{D_0} \Pi_{k=1}^{a_i} \phi_k({x_{k}})\sum_{D_0} \Pi_{k=1}^{a_i} {f_k}({x_{k}})}.
\end{array}
\]

We now write $\phi_k=f_k(1+\xi_k)$ where $\xi_k$ depends only on $x_1,...,x_{k}$ (but we omit the dependence) and $\xi_k=O(S^{-2w\tau})$ from the definition of $f$. Furthermore, denote $\delta_k =0$ if $k = a_i$ or $k=a_i -1$ and $\delta_k = 1$ otherwise.
% NOT USED:  and define $\delta_{k,x}=1$ if $k=x$ and $0$ otherwise. 
These steps yield  

\begin{equation}
\label{eq:deltastar}
\Delta_i =\frac{(a_{i}-1)[\sum_{D_1,D_0} \left( \Pi_{x_k\in D_1} f_k({x_{k}})(1+\xi_k\delta_k)\Pi_{x_k\in D_0} {f_k}({x_{k}})- \Pi_{x_k\in D_1}{f_k}({x_{k}}) \Pi_{x_k \in D_0} f_k({x_{k}})(1+\xi_k\delta_k)\right)+\epsilon ]}{\sum_{D_0} \Pi_{k=1}^{a_i} \phi_k({x_{k}})\sum_{D_0} \Pi_{k=1}^{a_i} {f_k}({x_{k}})} 
\end{equation}
where $\epsilon$ is the compensatory term for zeroing out certain terms by inserting the $\delta_k$ into equation (\ref{eq:deltastar}), which we can express as  $\epsilon = \epsilon_1 + \epsilon_2 - \epsilon_3 - \epsilon_4$ for 

\begin{equation} \label{eq:epsilon}
\epsilon_1 =  \sum_{D_1}\xi_{a_i}f_{a_i}(x_{a_i})f_{a_i-1}(x_{a_i-1})\Pi_{k\neq a_i-1,a_i}f_k({x_{k}})(1+\xi_k)\sum_{D_0}\Pi_{x_k\in D_0} {f_k}({x_{k}}), 
\end{equation}

\begin{equation}
\epsilon_2 =  \sum_{D_1}\xi_{a_i-1}f_{a_i-1}(x_{a_i-1})\Pi_{k\neq a_i-1}f_k({x_{k}})(1+\xi_k)\sum_{D_0}\Pi_{x_k\in D_0} {f_k}({x_{k}}), 
\end{equation}

\begin{equation}
\epsilon_3 =  \sum_{D_1}\Pi_{x_k\in D_1} {f_k}({x_{k}})\sum_{D_0}\xi_{a_i-1}f_{a_i}(x_{a_i})f_{a_i-1}(x_{a_i-1})\Pi_{k\neq a_i-1,a_i}f_k({x_{k}})(1+\xi_k), 
\end{equation}

\begin{equation}
\epsilon_4 =  \sum_{D_1}\Pi_{x_k\in D_1} {f_k}({x_{k}})\sum_{D_0}\xi_{a_i}f_{a_i}(x_{a_i})\Pi_{k\neq a_i}f_k({x_{k}})(1+\xi_k).
\end{equation}

Now, for $k=a_i$ or $k=a_i-1$, $f_k(x_k)(1+\xi_k\delta_k)=f_k(x_k)$ by definition. Factoring $f_k(x_k)$ out for those choices of $x_k$ in both $D_0$ and $D_1$, applying a version of the mean value theorem  and using the fact that $\xi_k$ only depends on $x_1,...,x_{k}$ enables us to integrate out the last two variables of $D_0$ and $D_1$.  Since the first $a_{i}-2$ indices are distinct in each element of both $D_0$ and $D_1$, if we define $D_*$ as the set of sets of $a_{i}-2$ distinct indices, then the expression for $\Delta_i$ can now be written as  

\[
\begin{array}{rcl} 
\Delta_i &=&\frac{\textstyle (a_{i}-1)[\lambda \{\sum_{D_*,D_*}\left( \Pi_{x_k\in D_*} f_k({x_{k}})(1+\xi_k)\Pi_{x_k\in D_*} {f_k}({x_{k}})- \Pi_{x_k\in D_*}{f_k}({x_{k}}) \Pi_{x_k \in D_*} f_k({x_{k}})(1+\xi_k)\right) \}+\epsilon ]}{\textstyle \sum_{D_0} \Pi_{k=1}^{a_i} \phi_k({x_{k}})\sum_{D_0} \Pi_{k=1}^{a_i} {f_k}({x_{k}})} \vspace{0.1in} \\
&=&  \frac{\textstyle (a_{i}-1)[\epsilon ]}{\textstyle \sum_{D_0} \Pi_{k=1}^{a_i} \phi_k({x_{k}})\sum_{D_0} \Pi_{k=1}^{a_i} {f_k}({x_{k}})}
\end{array}
\]
where  $\lambda$ is the constant from the application of the mean value theorem.  

To bound $\Delta_i$,  we use the crude approximation that for $k=a_i-1,a_i$, both $f_k(x_k,\mathbf{b})=\frac{\mb{b}_{x_k}}{\mb{b}_{u_k}}(1+O(S^{-2\tau}))$ and $\phi_k(x_k)=\frac{b_{x_k}}{b_{u_k}}(1+O(S^{-2\tau}))$,  
and we multiply the numerator and denominator of $\Delta_i$ by $(\Pi_{k=a_{i}-1}^{a_{i}}b_{u_k})^{2}$.
From equation (\ref{eq:epsilon}), we have that

$$(a_i-1)(\Pi_{k=a_{i}-1}^{a_{i}}b_{u_k})^{2}\epsilon_1 \leq $$
$$(a_i-1) \sum_{D_1}(\Pi_{k=a_{i}-1}^{a_{i}}b_{u_k})\xi_{a_i}f_{a_i}(x_{a_i})f_{a_i-1}(x_{a_i-1})\Pi_{k\neq a_i-1,a_i}f_k({x_{k}})(1+\xi_k)\sum_{D_0}(\Pi_{k=a_{i}-1}^{a_{i}}b_{u_k})\Pi_{x_k\in D_0} {f_k}({x_{k}}). $$
Note that $\xi_{a_i}=O(S^{-2w\tau})$ and $a_i\leq d_{max}$. Moreover, using the relationship that $\sum_{m=1}^{N}b_{u_{k}}\phi_k(x_m,\mathbf{b})=S(1+O(S^{-2\tau}))$ to integrate out the last two variables $x_{a_i-1}$ and $x_{a_i}$ from $D_0$, we can obtain the bound 
\[
\sum_{D_0}(\Pi_{k=a_{i}-1}^{a_{i}}b_{u_k})\Pi_{x_k\in D_0}^{a_i} {f_k}({x_{k}}) \leq \sum_{D_0*}(S+O(S^{1-2\tau}))^{2}\Pi_{x_k\in D_0*}{f_k}({x_{k}}).
\]
Hence, 

$$(a_i-1)(\Pi_{k=a_{i}-1}^{a_{i}}b_{u_k})^{2}\epsilon_1 \leq$$ $$ d_{max} O(S^{-2w\tau}) \sum_{D_1}(\Pi_{k=a_{i}-1}^{a_{i}}b_{u_k})f_{a_i}(x_{a_i})f_{a_i-1}(x_{a_i-1})\Pi_{k\neq a_i-1,a_i}f_k({x_{k}})(1+\xi_k)\sum_{D_0*}(S+O(S^{1-2\tau}))^{2}\Pi_{x_k\in D_0*}{f_k}({x_{k}}). $$

Similarly, we conclude that

$$ (a_i-1)(\Pi_{k=a_{i}-1}^{a_{i}}b_{u_k})^{2}\epsilon_1 \leq O(S^{-2w\tau})d_{max}^{2}[S+O(S^{1-2\tau})]^{3}\sum_{D_1*}\Pi_{x_k\in D_1*}f_k({x_{k}})\sum_{D_0*}\Pi_{x_k\in D_0*}{f_k}({x_{k}})$$
by replacing the summation over $D_1$ with a summation over $D_{1*}$, again integrating out the last two variables $x_{a_i-1}$ and $x_{a_i}$ from $D_1$ using $b_{u_{k}}\phi_k(x_m,\mathbf{b})=b_{x_{m}}(1+O(S^{-2\tau})).$

Repeating this argument for all $\epsilon_i$ $i=2,3,4$, we obtain that $$(a_{i}-1)(\Pi_{k=a_i-1}^{a_i} b_{u_k})^{2}\epsilon \leq 4d_{max}^{2}(S+O(S^{1-2\tau}))^{3}O(S^{-2w\tau})\sum_{D_1*} \Pi_{x_k\in D_1*} f_k({x_{k}})(1+\xi_k)\sum_{D_0*}\Pi_{x_k\in D_0*} {f_k}({x_{k}}).$$
Similarly, $$(\Pi_{k=a_{i}-1,a_{i}}b_{u_k})^{2}\sum_{D_0} \Pi_{k=1}^{a_i} \phi_k({x_{k}})\sum_{D_0} \Pi_{k=1}^{a_i} {f_k}({x_{k}})\geq (S-O(S^{1-2\tau}))^{4}\sum_{D_1*} \Pi_{x_k\in D_1*} f_k({x_{k}})(1+\xi_k)\sum_{D_0*}\Pi_{x_k\in D_0*} {f_k}({x_{k}}).$$

Simplifying, using the fact that $d_{max} = O(S^{\frac{1}{2}-\tau})$, we obtain 
$$\Delta_i =\frac{(a_{i}-1)(\Pi_{k=a_i-1}^{a_i} b_{u_k})^{2}\epsilon }{(\Pi_{k=a_{i}-1,a_{i}}b_{u_k})^{2}\sum_{D_0} \Pi_{k=1}^{a_i} \phi_k({x_{k}})\sum_{D_0} \Pi_{k=1}^{a_i} {f_k}({x_{k}})}=O(S^{-2(w+1)\tau}).$$
%Note that $\mid \log(1+x)-\log(1+y)\mid \leq \mid x-y\mid$ for $x,y>0$.   
Therefore, we conclude that  $g$ as defined in the theorem statement can be approximated as $g(\mathbf{e}_{i},\mathbf{e_{j}},\mathbf{d},\mathbf{a})=\frac{\|G_{\mathbf{d_{-i}}}\|}{\|G_{\mathbf{d_{-j}}}\|}(1+O(S^{-2(w+1)\tau})),$
even though our approximation for $f$ is $f(\mathbf{e}_{i},\mathbf{e}_{j},\mathbf{d},\mathbf{x})=\frac{\|G_{\mathbf{d_{-i}}}\|}{\|G_{\mathbf{d_{-j}}}\|}(1+O(S^{-2w\tau}))$.

We have nearly completed the proof.  We argued above that the choice of $u_k$ should not have any impact in evaluating expressions such as 
 $$(a_{i}-1)\sum_{x_{1}\neq...\neq x_{a_{i-1}}=x_{a_i}} \Pi_{k=1}^{a_i} {f}(\mathbf{e}_{x_{k}},\mathbf{e}_{u_{k}},(\mathbf{a}-a_i\mb{e}_i,\mathbf{b}-\sum_{j=1}^{k-1} \mathbf{e}_{x_{j}}-\sum_{j=1}^{a_i-k}\mathbf{e}_{u_{j}}),\mathbf{b})=\sum_{\mathbf{x}\in X_{1_i}}\frac{\|G_{\mathbf{x}}\|}{\|G_{\mathbf{u}}\|} $$ 
 By our assumptions on $f$, the dependence on the degree sequence is ignored (i.e., represents a higher order term) unless  one of the $u_{a_i-m+1}$ equals one of the $x_k$'s.  (That is, when we evaluate $\Pi_{k}f(\mathbf{e}_{x_{k}},\mathbf{e}_{u_{a_i-k+1}},(\mathbf{a}-a_i\mb{e}_i,\mathbf{b}-\sum_{j=1}^{k-1} \mathbf{e}_{x_{j}}-\sum_{j=1}^{a_i-k}\mathbf{e}_{u_{j}}),\mathbf{b})$ the out-degree of node $x_k$ may be $b_{x_k}-1$ and not $b_{x_k}$.) 
 
If this is the case, then consider the product 
\begin{equation}
\label{eq:indep}
f(\mathbf{e}_{x_{k}},\mathbf{e}_{u_{a_i-k+1}},\mathbf{b}-\sum_{j=1}^{k-1} \mathbf{e}_{x_{j}}-\sum_{j=1}^{a_i-k}\mathbf{e}_{u_{j}}) f(\mathbf{e}_{x_{m}},\mathbf{e}_{u_{a_i-m+1}},\mathbf{b}-\sum_{j=1}^{m-1} \mathbf{e}_{x_{j}}-\sum_{j=1}^{a_i-m}\mathbf{e}_{u_{j}})
\end{equation}
where we omitted the dependence on the in-degree sequence and the last argument $\mathbf{b}$.
Now when evaluating $f(\mathbf{e}_{x_{k}},\mathbf{e}_{u_{a_i-k+1}},\mathbf{b}-\sum_{j=1}^{k-1} \mathbf{e}_{x_{j}}-\sum_{j=1}^{a_i-k}\mathbf{e}_{u_{j}})$,  the degree of the $x_k$ is $b_{x_k}-1$ if and only if $a_i-m+1\leq a_i -k$,  so in both terms of the product the out-degree of node $x_k$ is $b_{x_k}-1$.  

We will now invoke assumption (\ref{eq:assumption}).  We will condense notation a bit:  in the function $h$, if $\mathbf{d}=(\mathbf{a},\mathbf{b})$ and $\mathbf{\sigma}$ selects the in-degree (out-degree) sequence, then we write $h(\mathbf{e}_i,\mathbf{a})$ ($h(\mathbf{e}_i,\mathbf{b})$).   We can thus write (\ref{eq:indep})  as   $$\left(\frac{h(x_k,\mathbf{b}-\sum_{j=1}^{k-1} \mathbf{e}_{x_{j}}-\sum_{j=1}^{a_i-k}\mathbf{e}_{u_{j}})}{h(u_{a_i-k+1},\mathbf{b}-\sum_{j=1}^{k-1} \mathbf{e}_{x_{j}}-\sum_{j=1}^{a_i-k}\mathbf{e}_{u_{j}})} \right) \left(\frac{h(x_m,\mathbf{b}-\sum_{j=1}^{m-1} \mathbf{e}_{x_{j}}-\sum_{j=1}^{a_i-m}\mathbf{e}_{u_{j}})}{h(u_{a_i-m+1},\mathbf{b}-\sum_{j=1}^{m-1} \mathbf{e}_{x_{j}}-\sum_{j=1}^{a_i-m}\mathbf{e}_{u_{j}})}\right).$$
But since $x_k = u_{a_i-m+1}$, $$\frac{h(x_k,\mathbf{b}-\sum_{j=1}^{k-1} \mathbf{e}_{x_{j}}-\sum_{j=1}^{a_i-k}\mathbf{e}_{u_{j}})}{h(u_{a_i-m+1},\mathbf{b}-\sum_{j=1}^{m-1} \mathbf{e}_{x_{j}}-\sum_{j=1}^{a_i-m}\mathbf{e}_{u_{j}})}=(1+O(S^{-2w\tau}))=\frac{h(x_k,\mathbf{z_1})}{h(u_{a_i-m+1},\mathbf{z_2})}(1+O(S^{-2w\tau}))$$
for any arbitrary vectors $\mathbf{z_1}$,$\mathbf{z_2}$ that satisfy assumption (\ref{eq:assumption}), as $x_k=u_{a_i-m+1}$. In particular, we can choose $\mathbf{z_1}$,$\mathbf{z_2}$ such that the out-degree of node $x_k$ is $b_{x_k}$ and not $b_{x_k}-1$  ,hence the dependence on the degree sequence can be ignored as initially claimed.
\end{proof}

Even with the many assumptions in the statement of Theorem \ref{thm:arborder1}, we still do not have our asymptotic enumeration result for counting the number of graphs realizing a given bidegree sequence.  The issue is that we would need to evaluate products of approximations $\Pi_{i=1}^{a_i} f(x_i)(1+O(S^{-2w\tau}))$ where $a_i=O(S^{\frac{1}{2}-\tau})$, which diverges as $S \to \infty$.
To avoid this problem, we note that Theorem \ref{thm:arborder1} does give us a way to shrink the error term in the product, decreasing the power of $S$ by 1 in each step.
Thus, by repeatedly applying Theorem  \ref{thm:arborder1}, we can obtain a product of the form $\Pi_{i=1}^{a_i} f(x_i)(1+O(S^{-\frac{1}{2}}))$, which does not yield divergence in the limit.  
To harness this strategy, we use a result that is analogous to Theorem  \ref{thm:arborder1} but starts with an approximation of $O(S^{-\frac{1}{2} - w\tau})$.  For this additional result, stated in Theorem \ref{thm:arborder2}, we no longer need the full assumptions made in Theorem \ref{thm:arborder1}, since starting from an improved approximation means that certain terms must stay bounded.

\begin{thm}
\label{thm:arborder2}
Consider an approximation $$\frac{\|G_{\mathbf{d_{-i}}}\|}{\|G_{\mathbf{d_{-j}}}\|}=f(\mathbf{e}_{i},\mathbf{e}_{j},\mathbf{d},\mathbf{\sigma})(1+O(S^{-\frac{1}{2}-w\tau}))$$ for some $w> 0$. 
Furthermore suppose that for $m=O(S^{\frac{1}{2}-\tau})$, $$ f(\mathbf{e}_{i},\mathbf{e}_{j},\mathbf{d}_0,\mathbf{\sigma})= f(\mathbf{e}_{i},\mathbf{e}_{j},\mathbf{d}_1,\mathbf{\sigma})+z(\mathbf{e}_{i},\mathbf{e}_{j},\mathbf{d}_0-\mathbf{d}_1,\mathbf{d}_0,\mathbf{\sigma})$$
where  $\|\mathbf{d}_1-\mathbf{d}_0\|_1\leq m$ and $z(\mathbf{e}_{i},\mathbf{e}_{j},\mathbf{d}_0-\mathbf{d}_1,\mathbf{d}_0,\mathbf{\sigma})\leq  O(S^{-\frac{1}{2}-\tau})f(\mathbf{e}_{i},\mathbf{e}_{j},\mathbf{d}_0,\mathbf{\sigma})$.
\textcolor{black}{If $\mathbf{\sigma}=\mathbf{a}$,} then we can construct a sharper approximation $$\frac{\|G_{\mathbf{d_{-i}}}\|}{\|G_{\mathbf{d_{-j}}}\|}=g(\mathbf{e}_{i},\mathbf{e_{j}},\mathbf{d},\mathbf{a})(1+O(S^{-\frac{1}{2}-(w+2)\tau}))$$ where
\[
\begin{array}{rcl}
g(\mathbf{e}_{i},\mathbf{e}_{j},\mathbf{d},\mathbf{a})&=&\frac{a_{i}}{a_{j}}\exp(\log(1+\frac{(a_{i}-1)\sum_{x_{1}\neq...\neq x_{a_{i-1}}=x_{a_i}} \Pi_{k=1}^{a_i} f(\mathbf{e}_{x_{k}},\mathbf{e}_{u_{k}},(\mathbf{a}-a_i\mb{e}_i,(\mathbf{b}-\sum_{j=1}^{k-1} \mathbf{e}_{x_{j}}+\sum_{j=1}^{a_i-k}\mathbf{e}_{u_{j}}),\mathbf{b}) }{\sum_{x_{1}\neq...\neq x_{a_{i-1}}\neq x_{a_i}} \Pi_{k=1}^{a_i} f(\mathbf{e}_{x_{k}},\mathbf{e}_{u_{k}},(\mathbf{a}-a_i\mb{e}_i,\mathbf{b}-\sum_{j=1}^{k-1} \mathbf{e}_{x_{j}}+\sum_{j=1}^{a_i-k}\mathbf{e}_{u_{j}}),\mathbf{b})})- \vspace{0.1in} \\
 & & \log(1+\frac{(a_{j}-1)\sum_{x_{1}\neq...\neq x_{a_{j-1}}=x_{a_j}} \Pi_{k=1}^{a_j} f(\mathbf{e}_{x_{k}},\mathbf{e}_{u_{k}},(\mathbf{a}-a_je_j,\mathbf{b}-\sum_{j=1}^{k-1} \mathbf{e}_{x_{j}}+\sum_{j=1}^{a_i-k}\mathbf{e}_{u_{j}}),\mathbf{b}) }{\sum_{x_{1}\neq...\neq x_{a_{j-1}}\neq x_{a_j}} \Pi_{k=1}^{a_j} f(\mathbf{e}_{x_{k}},\mathbf{e}_{u_{k}},(\mathbf{a}-a_je_j,\mathbf{b}-\sum_{j=1}^{k-1} \mathbf{e}_{x_{j}}+\sum_{j=1}^{a_i-k}\mathbf{e}_{u_{j}}),\mathbf{b})}))
\end{array}
\]
for an arbitrary choice of $u_{k}$. A similar result holds, with $g$ depending on $\mathbf{b}$, if $\mathbf{\sigma}=\mathbf{b}$.
\end{thm}

The proof is very similar to the prior theorem so we leave the details to Appendix C.  
So in order to construct asymptotics for the ratio of the number of graphs of two slightly different degree sequences, starting with our approximation $\frac{\|G_{\mathbf{d_{-i}}}\|}{\|G_{\mathbf{d_{-j}}}\|}=\frac{a_i}{a_j}(1+O(S^{-2\tau}))$, we apply Theorem \ref{thm:arborder1}, obtaining stronger approximations until we reach an approximation with a multiplicative error of $(1+O(S^{-\frac{1}{2}}))$ from the solution.  We can then apply Theorem \ref{thm:arborder2} to construct approximations that are arbitrarily accurate.  As mentioned before, this argument requires that the additional assumptions of Theorem \ref{thm:arborder1} are true, which we will prove in Section 6.  

We now provide a general method for constructing asymptotics for $\|G_{\mathbf{d}}\|$ from asymptotics for $\|G_{\mathbf{d_{}}}\|/\|G_{\mathbf{d_{*}}}\|$.
We start our derivation, as we did in Section 2, by considering a special case where it is particularly easy to count the number of graphs that realize a specified degree sequence.

\begin{cor}
\label{cor:count}  Consider the degree sequence $\mathbf{(a, b)}\in\mathbb{Z}^{N\times 2}$ where $1$ appears exactly $S$ times in $\mathbf{b}$ and $0$ appears exactly $N-S$ times in $\mathbf{b}$. Then $$\|G_{(\mathbf{a},\mathbf{b})}\|=\frac{S!}{\Pi_{i=1}^{N}a_{i}!}$$ 
\end{cor}
\begin{proof}
The result follows immediately from  Lemma 1 where $q = 0$ and $k = N$.
\end{proof}

\begin{thm}
\label{thm:graphcount}
Define a function $\rho:\mathbb{N}\rightarrow\mathbb{N}$ such that 
$\mathbf{b}=\sum_{k=1}^{S}\mathbf{e}_{\rho(k)}$.  Suppose that for each $i\in \{ 1, \ldots, S \}$, the following approximation holds:
 \begin{equation}
 \label{eq:rat_approx}
 \frac{\|G_{(\mathbf{a},\sum_{k=1}^{S-i}\mathbf{e}_{k}+\sum_{k=1}^{i}e_{\rho(k)})}\|}{\|G_{(\mathbf{a}, \sum_{k=1}^{S-i+1}\mathbf{e}_{k}+\sum_{k=1}^{i-1}\mathbf{e}_{\rho(k)})}\|}(1+O(S^{-1-\epsilon}))=f(i)
 \end{equation}
  for some $\epsilon>0$.
Then $$\|G_{\mathbf{d}}\| = (1+O(S^{-\epsilon}))\frac{S!}{\Pi_{i=1}^{N}a_i!}\Pi_{i=1}^{S}f(i).$$
\end{thm}

\begin{proof}
For simplicity, we first suppose that $S \leq N$. 
%Let $\mathbf{1}_{S}$ consist of $S$ 1's and $N-S$ 0's.  Consequently, we can let 
Write $\mathbf{1}_{S}=\sum_{k=1}^{S}\mathbf{e}_{k}$ where $\mathbf{e}_{k}$ is the kth standard unit vector.
From Corollary \ref{cor:count}, 
\begin{equation}
\label{eq:cor6}
\|G_{(\mathbf{a},\mathbf{1}_{S})}\|=\frac{S!}{\Pi_{i=1}^{N}a_{i}!}.
\end{equation}
Consequently,
\begin{equation}
\|G_{\mathbf{d}}\|=\|G_{(\mathbf{a},\mathbf{1}_{S})}\|\Pi_{i=1}^{S}\frac{\|G_{(\mathbf{a}, \sum_{k=1}^{S-i}\mathbf{e}_{k}+\sum_{k=1}^{i}\mathbf{e}_{\rho(k)})}\|}{\|G_{(\mathbf{a}, \sum_{k=1}^{S-i+1}\mathbf{e}_{k}+\sum_{k=1}^{i-1}\mathbf{e}_{\rho(k)})}\|}.
\end{equation}
But by assumption (\ref{eq:rat_approx}) and equation (\ref{eq:cor6}), we have that
\begin{equation}
\|G_{\mathbf{(\mathbf{a},\mathbf{1}_{S})}}\|\Pi_{i=1}^{S}\frac{\|G_{(\mathbf{a}, \sum_{k=1}^{S-i}\mathbf{e}_{k}+\sum_{k=1}^{i}\mathbf{e}_{\rho(k)})}\|}{\|G_{(\mathbf{a}, \sum_{k=1}^{S-i+1}\mathbf{e}_{k}+\sum_{k=1}^{i-1}\mathbf{e}_{\rho(k)})}\|}=\frac{S!}{\Pi_{i=1}^{N} a_i!}\Pi_{i=1}^{S}f(i)(1+O(S^{-1-\epsilon}))=(1+O(S^{-\epsilon}))\frac{S!}{\Pi_{i=1}^{N} a_i!}\Pi_{i=1}^{S}f(i).
\end{equation}

Finally, in the event that $S > N$,we extend $\mathbf{d}\in\mathbb{Z}^{N\times 2}$ to a bidegree sequence $\mathbf{d_{*}}\in\mathbb{Z}^{S\times 2}$ by appending zeros to $\mathbf{d}$. The proof then proceeds analogously.

\end{proof}

In the following section we will apply Theorems \ref{thm:arborder1}, \ref{thm:arborder2} and \ref{thm:graphcount} to construct asymptotics for the number of graphs of a given degree sequence.  The cases worked out in the following section will not only provide extra intution about the veracity of the assumptions of Theorem \ref{thm:arborder1}, but will also lay the groundwork for the rigorous proof that the assumptions are indeed true.

\section{Higher Order Approximations}
\label{section:higherorder}
We start with a technical lemma that will help us evaluate summations of the form $\sum_{x_1\neq...\neq x_r}^{N} \Pi_{i=1}^{r} f(x_i)$ where the notation $\sum_{x_1\neq...\neq x_r}^{N}$ tells us that for each $i \in \{1,2,..,r\}$,  we take $x_i\in \{1,2,...,N\}$ and we are summing over all $\binom{N}{r}$ choices of $r$ distinct elements in $\{1,2,...,N\}$. 

\begin{lem}\label{lem:count}
%Consider $\sum_{x_{1}\neq...\neq x_{r}}^{N}\Pi_{i=1}^{r} f(x_{i})$, where $r=O(N^{\frac{1}{2}-\tau})$, 
Suppose that $ f : \mathbb{N} \to [0,\infty)$, $\sum_{i=1}^{N}f(i)=S$,   %$f(\cdot)\geq 1$ and 
$\max f(\cdot)=O(S^{\frac{1}{2}-\tau})$, and that $x_1, \ldots, x_r$ is a sequence of natural numbers with $r=O(S^{\frac{1}{2}-\tau})$.  Let $k$ be a fixed (i.e., $O(1)$) natural number and %Let $k$ be bounded by a constant.
define 
\[
\xi = \sum_{\substack{x_{1},..,x_{k},\\x_{k+1}\neq...\neq x_{r}}}^{N}\Pi_{i=1}^{r} f(x_{i}). 
\]
Then $$\sum_{x_{1}\neq...\neq x_{r}}^{N}\Pi_{i=1}^{r} f(x_{i})=\xi(1+O( S^{-2\tau})).$$

\end{lem}
\begin{proof} We proceed by induction on $k$.  Start with $k=1$. So   
\begin{equation}
\label{eq:induct}
\sum_{x_{1}\neq...\neq x_{r}}^{N}\Pi_{i=1}^{r} f(x_{i})=\sum_{x_{1},x_{2}\neq...\neq x_{r}}^{N}\Pi_{i=1}^{r} f(x_{i})-(r-1)\sum_{x_{1}=x_{2}\neq...\neq x_{r}}^{N}\Pi_{i=1}^{r} f(x_{2}).
\end{equation}  
 But the bounds on $r$ and $\max f(\cdot)$ imply that 
\begin{equation}
\label{eq:upbound}
(r-1)\sum_{x_{1}=x_{2}\neq...\neq x_{r}}^{N}\Pi_{i=1}^{r} f(x_{i})\leq O(S^{1-2\tau})\sum_{x_{2}\neq...\neq x_{r}}^{N}\Pi_{i=2}^{r} f(x_{i}),
\end{equation}
 while 
\begin{equation}
\label{eq:lowbound}
 \sum_{x_{1},x_{2}\neq...\neq x_{r}}^{N}\Pi_{i=1}^{r} f(x_{i})\geq O(S)\sum_{x_{2}\neq...\neq x_{r}}^{N}\Pi_{i=1}^{r} f(x_{i}).
 \end{equation}
Using (\ref{eq:upbound}) and (\ref{eq:lowbound}) to express the right hand side of (\ref{eq:induct}) in terms of $\xi$ yields the desired result  for $k=1$.
 
   Now assume that the lemma is true when $k=m$ and consider the case $k=m+1$.  We have $$\sum_{x_{1}\neq...\neq x_{r}}^{N}\Pi_{i=1}^{r} f(x_{i})=\sum_{\substack{x_{1},..,x_{m},\\x_{m+1}\neq...\neq x_{r}}}^{N}\Pi_{i=1}^{r} f(x_{i})+O(S^{-2\tau})=$$ $$\sum_{\substack{x_{1},..,x_{m},x_{m+1},\\x_{m+2}\neq...\neq x_{r}}}^{N}\Pi_{i=1}^{r} f(x_{i})+O(S^{-2\tau})-(r-m-1)\sum_{\substack{x_{1},...,x_{m},\\x_{m+1}=x_{m+2}\\\neq...\neq x_{r}}}^{N}\Pi_{i=1}^{r}f(x_{i}).$$  By applying the estimates (\ref{eq:upbound}), (\ref{eq:lowbound}), it follows as before that $$(r-m-1)\sum_{\substack{x_{1},...,x_{m},\\x_{m+1}=x_{m+2}\\\neq...\neq x_{r}}}^{N}\Pi_{i=1}^{r}f(x_{i})=O(S^{-2\tau}),$$ and the proof is complete.

\end{proof}

\begin{cor}
\label{cor:count2}
With the notation of Lemma \ref{lem:count}, 
\textcolor{black}{let $k > m>1$ be fixed natural numbers, let $f, g:\mathbb{N}\rightarrow[0,\infty)$ with $g(\cdot)\leq f(\cdot)$, $\sum_{i=1}^{N}f(i)=S$, and $\max f(\cdot)=O(S^{\frac{1}{2}-\tau})$,  and let $r=O(S^{\frac{1}{2}-\tau})$}.
Define $$\zeta = \sum_{\substack{x_{1}=...=x_{m},\\x_{m+1},...,x_{k},\\x_{k+1}\neq...\neq x_{r}}}^{N}g(x_{1})\Pi_{i=2}^{r} f(x_{i}).$$
Then $$\sum_{x_{1}=...=x_{m}\neq...\neq x_{r}}^{N}g(x_{1})\Pi_{i=2}^{r} f(x_{i})=\zeta(1+O( S^{-2\tau})).$$

\end{cor}
\begin{proof}
The proof is identical to that of Lemma \ref{lem:count}.
\end{proof}
For use in later proofs, it is worth noting that while the error terms in Lemma \ref{lem:count} (and analogously in Corollary \ref{cor:count2}) are expressed as  
\[
O(S^{-2\tau})\sum_{\substack{x_{1},..,x_{k},\\x_{k+1}\neq...\neq x_{r}}}^{N}\Pi_{i=1}^{r} f(x_{i}),
\]
they can also be stated in terms of $O(S^{-2\tau})\sum_{x_{1}\neq...\neq x_{r}}^{N}\Pi_{i=1}^{r} f(x_{i})$ since asymptotically (as the above results imply), 
\[
\sum_{x_{1}\neq...\neq x_{r}}^{N}\Pi_{i=1}^{r} f(x_{i})\approx \sum_{\substack{x_{1},..,x_{k},\\x_{k+1}\neq...\neq x_{r}}}^{N}\Pi_{i=1}^{r} f(x_{i}).
\]

With Lemma \ref{lem:count} and Corollary \ref{cor:count2} at our disposal, we can now derive the first order approximation of our power series with relative ease, but first we introduce some additional notation.
\begin{mydef}
\label{def:alphabeta}
Given a bidegree sequence $\mathbf{d}=(\mathbf{a},\mathbf{b})\in \mathbb{Z}^{N \times 2}$, we denote $\alpha_k:=\sum_{i=1}^{N}a_{i}^{k}$ and 
$\beta_k:=\sum_{i=1}^{N}b_{i}^{k}$. 
\end{mydef}

\begin{thm}
\label{thm:S4}
If $\max_{i} \max(a_{i},b_{i})=O(S^{\frac{1}{2}-\tau})$ for $\tau>0$, then 
$$\frac{\|G_{d_{-i}}\|}{\|G_{d_{-j}}\|}= \frac{a_{i}}{a_{j}}e^{(a_{i}-a_{j})\epsilon+O(S^{-4\tau})}$$ where, $\epsilon=(\beta_{2}-\beta_{1})/\beta_{1}^2$ and $(a_{i}-a_{j})\epsilon=O(S^{-2\tau})$. 
\end{thm}
\begin{proof}
We start with the approximation $$\frac{\|G_{\mathbf{d}_{-i}}\|}{\|G_{\mathbf{d}_{-j}}\|}=f(\mathbf{e}_i,\mathbf{e}_j,\mathbf{d},\mathbf{a})(1+O(S^{-2\tau}))$$
for $f(\mathbf{e}_i,\mathbf{e}_j,\mathbf{d},\mathbf{a})=a_i/a_j$. We can apply Theorem \ref{thm:arborder1}, since our approximation depends only on the degrees of the nodes $i$ and $j$ and we can decompose our approximation, $f(\mathbf{e}_i,\mathbf{e}_j,\mathbf{d},\mathbf{a})=h(\mathbf{e}_i,\mathbf{d},\mathbf{a})/h(\mathbf{e}_j,\mathbf{d},\mathbf{a})$, where $h(\mathbf{e}_i,\mathbf{d},\mathbf{a})=a_i$;below, we shall also use $f(\mathbf{e}_i,\mathbf{e}_j,\mathbf{d},\mathbf{b})=h(\mathbf{e}_i,\mathbf{d},\mathbf{b})/h(\mathbf{e}_j,\mathbf{d},\mathbf{b})$, where $h(\mathbf{e}_i,\mathbf{d},\mathbf{b})=b_i$
For the remainder of the proof, we omit the dependence on $\mathbf{d}$ in our notation.

From the conclusion of Theorem \ref{thm:arborder1}, we need to evaluate $$\delta_{i} :=\frac{(a_{i}-1)\sum_{x_{1}\neq...\neq x_{a_{i-1}}=x_{a_i}} \Pi_{k=1}^{a_i} f(\mathbf{e}_{x_{k}},\mathbf{e}_{u_{k}},\mathbf{b}_k) }{\sum_{x_{1}\neq...\neq x_{a_{i-1}}\neq x_{a_i}} \Pi_{k=1}^{a_i} f(\mathbf{e}_{x_{k}},\mathbf{e}_{u_{k}},\mathbf{b}_k)}$$ 
and the analogously defined $\delta_j$, where 
\[
\mathbf{b}_k=\mathbf{b}-\sum_{j=1}^{a_i-k+1}\mathbf{e}_{u_j}-\sum_{j=1}^{k-1}\mathbf{e}_{x_j}
\]
 denotes the out-degree of the sequence $d_{k,i,j}$, as in the statement of Theorem \ref{thm:arborder1}.
Now, multiplying our numerator and denominator by $\Pi_{k=1}^{a_i} b_{u_{k}}$ removes the dependence of $f$ on $\mathbf{e}_{u_{k}}$, as $ f(\mathbf{e}_{x_{k}},\mathbf{e}_{u_{k}},\mathbf{b}_k)=b_{x_{k}}/b_{u_{k}}$  except for when $x_{a_{i}-1}=x_{a_i}$ and $k = a_i$,  in  which case $ f(\mathbf{e}_{x_{a_i}},\mathbf{e}_{u_{a_i}},\mathbf{b}_{a_i})=(b_{x_{a_i}}-1)/b_{u_{a_i}}$, since we have already subtracted an outgoing edge from the node $b_{x_{a_{i}-1}}$, reducing its out-degree by 1.  

By Corollary \ref{cor:count2}, with $ g=f(\mathbf{e}_{x_{a_i}},\mathbf{e}_{u_{a_i}},\mathbf{b}_{a_i})$  and a relabeling of  indices such that the case where $x_{a_{i}-1}=x_{a_i}$ becomes the case where $x_1=x_2$, we have 
$$\delta_{i}=\frac{(a_i-1)\sum_{x_1=x_2,x_3\neq...\neq x_{a_i}} \Pi_{k=1}^{a_i} f(\mathbf{e}_{x_{k}},\mathbf{b}_k) }{\sum_{x_1,x_2,x_3\neq...\neq x_{a_i}} \Pi_{k=1}^{a_i} f(\mathbf{e}_{x_{k}},\mathbf{b}_k)}(1+O(S^{-4\tau})).$$
Factoring out $\sum_{x_3\neq...\neq x_{a_i}}\Pi_{k=1}^{a_i} f(\mathbf{e}_{x_{k}},\mathbf{b})$ yields the following simplification:

$$\delta_i=\frac{(a_i-1)\sum_{x_1=x_2} \Pi f(\mathbf{e}_{x_{k}},\mathbf{b}_k) }{\sum_{x_1,x_2} \Pi f(\mathbf{e}_{x_{k}},\mathbf{b}_k)}(1+O(S^{-4\tau}))=\frac{(a_i-1)(\beta_2-\beta_1)}{\beta_1^2}(1+O(S^{-4\tau})).$$ 

where the last equality above follows from the fact that in the numerator where $x_1=x_2$, $f(\mathbf{e}_{x_{1}},\mathbf{b}_1)*f(\mathbf{e}_{x_{2}},\mathbf{b}_2)=f(\mathbf{e}_{x_{1}},\mathbf{b}_1)*f(\mathbf{e}_{x_{1}},\mathbf{b}_2)=b_{x_1}(b_{x_1}-1)=b_{x_1}^{2}-b_{x_1}$.  Summing over all choices for $x_1$ yields the expression $\beta_2-\beta_1$.
\newline 
We conclude that  $\log(1+\delta_i)=\log(1+ \frac{(a_i-1)[\beta_2-\beta_1]}{\beta_1^2})+O(S^{-4\tau})$ and analogously $\log(1+\delta_j)=\log(1+ \frac{(a_j-1)[\beta_2-\beta_1]}{\beta_1^2})+O(S^{-4\tau}).$
Hence from Theorem \ref{thm:arborder1} we know that  $\frac{\|G_{d_-i}\|}{\|G_{d_-j}\|}=\frac{a_i}{a_j}\exp[\log(1+ \frac{(a_i-1)[\beta_2-\beta_1]}{\beta_1^2})-\log(1+ \frac{(a_j-1)[\beta_2-\beta_1]}{\beta_1^2})+O(S^{-4\tau})]$ where writing the logarithms as a Taylor series yields $\frac{\|G_{d_-i}\|}{\|G_{d_-j}\|}=\frac{a_i}{a_j}\exp[\frac{(a_i-a_j)(\beta_2-\beta_1)}{\beta_1^2}+O(S^{-4\tau})].$

Finally, note that $$(a_i-a_j)\frac{\beta_2-\beta_1}{\beta_1^{2}}\leq \frac{d_{max}\beta_2}{\beta_1^{2}}\leq\frac{d_{max}^{2}\beta_1}{\beta_1^{2}} = \frac{d_{max}^{2}}{\beta_1}=O(S^{-2\tau})$$

\end{proof}

We can also prove the following statement with relative ease.

\begin{thm} 
Using the notation from Definition \ref{def:alphabeta} (with $\alpha_1=S$), 
if  $\max_{i}\max(a_{i},b_{i})=O(S^{\frac{1}{4}-\epsilon})$ for $\epsilon>0$, then 
 $$\|G_{\mathbf{d}}\|=[1+O(S^{-4\epsilon})]\frac{S!}{\Pi_{i=1}^{N}a_{i}!b_{i}!}\exp(\frac{-(\alpha_{2}-\alpha_{1})(\beta_{2}-\beta_{1})}{2S^{2}}).$$
\end{thm}

\begin{proof}
We know from Theorem \ref{thm:S4} that $$\frac{\|G_{d_{-i}}\|}{\|G_{d_{-j}}\|}= \frac{a_{i}}{a_{j}}e^{(a_{i}-a_{j})\epsilon+O(S^{-4\tau})}$$ where $\epsilon=(\beta_{2}-\beta_{1})/\beta_{1}^2$,$(a_{i}-a_{j})\epsilon=O(S^{-2\tau})$ when $d_{max} = O(S^{\frac{1}{2}-\tau})$ for $\tau \in (0,1/2)$.
If, in fact, $d_{max}=\max_{i}\max(a_{i},b_{i})=O(S^{\frac{1}{4}-\epsilon})$, then we have that $\tau = \frac{1}{4}+\epsilon$ and consequently $O(S^{-4\tau})=O(S^{-1-4\epsilon})$.  
By Theorem \ref{thm:graphcount} and the definition of $\rho$ given in the theorem statement, we know that 
$$\|G_{\mathbf{d}}\|=\frac{S!}{\Pi_{i=1}^{N}a_i!}\Pi_{i=1}^{S}f(i)(1+O(S^{-4\epsilon}))$$
where by equation (\ref{eq:rat_approx}) we can approximate $$f(i)=\frac{\|G_{(\mathbf{a},\sum_{k=1}^{S-i}e_{k}+\sum_{k=1}^{i}e_{\rho(k)})}\|}{\|G_{(\mathbf{a}, \sum_{k=1}^{S-i+1}e_{k}+\sum_{k=1}^{i-1}e_{\rho(k)})}\|}(1+O(S^{-1-4\epsilon})).$$

Now, we will proceed with a sequence of unit out-degree switches that transform an out-degree sequence consisting entirely of 1's to the out-degree sequence of $\mathbf{d}$, namely $\mathbf{b}$.  Without loss of generality, suppose that the first node in $\mathbf{b}$ has out-degree $b_1$, such that there are $b_1-1$ switches required, and that $\rho(1)= \cdots = \rho(b_{1}-1)=1$, which means that all of the switches will be applied to the first node, taking it from out-degree 1 to out-degree $b_1$.  Then it follows from Theorem \ref{thm:S4}  that 
$f(1) = \frac{1}{2}\exp((\alpha_2-\alpha_1)(1-2)/S^{2})$, since the first switch changes the out-degree of the first node from 1 to 2.  Similarly, it follows that 
$f(2) = \frac{1}{3}\exp((\alpha_2-\alpha_1)(1-3)/S^{2})$ and, more generally for $k\leq b_{1}-1$, $f(k) = \frac{1}{k+1}\exp((\alpha_2-\alpha_1)(1-[k+1])/S^{2})=\frac{1}{k+1}\exp((\alpha_2-\alpha_1)(-k)/S^{2})$.
Hence $$\Pi_{k=1}^{b_{i}-1}f(k) = \Pi_{k=1}^{b_{i}-1}\frac{1}{k+1}\exp((\alpha_2-\alpha_1)(-k)/S^{2})=$$
$$\frac{\exp((\alpha_2-\alpha_1)\sum_{k=1}^{b_{1}-1}(-k)/S^{2})}{b_{1}!}=\frac{\exp(-\frac{(\alpha_2-\alpha_1)(b_{1}^{2}-b_{1})}{2S^{2}})}{b_{1}!}.$$

Repeating this argument for all of the nodes in the degree sequence yields the result,

$$\|G_{\mathbf{d}}\|=(1+O(S^{-4\epsilon}))\frac{S!}{\Pi_{i=1}^{N}a_i!}\Pi_{i=1}^{S}f(i)=(1+O(S^{-4\epsilon}))\frac{S!}{\Pi_{i=1}^{N}a_{i}!b_{i}!}\exp(\frac{-(\alpha_{2}-\alpha_{1})(\beta_{2}-\beta_{1})}{2S^{2}}).$$

\end{proof}

We now explain the intuition behind some generalizations of Lemma \ref{lem:count} that we need to achieve higher order approximations.
Note that in Theorem \ref{thm:S4}, using Lemma \ref{lem:count}, we performed an approximation of the form  
$$ \frac{\sum_{x_{1}=x_{2}\neq...\neq x_{r}}\Pi_{n=1}^{r}f(x_{n})}{\sum_{x_{1}\neq x_{2}\neq...\neq x_{r}}\Pi_{n=1}^{r}f(x_{n})}\approx \frac{\sum_{x_{1}=x_{2}}\Pi_{n=1}^{2}f(x_{n})}{\sum_{x_{1},x_{2}}\Pi_{n=1}^{2}f(x_{n})} \left[ \frac{\sum_{x_{3}\neq...\neq x_{r}}\Pi_{n=3}^{r}f(x_{n})}{\sum_{x_{3}\neq...\neq x_{r}}\Pi_{n=3}^{r}f(x_{n})} \right]$$ $$=\frac{\sum_{x_{1}=x_{2}}\Pi_{n=1}^{2}f(x_{n})}{\sum_{x_{1},x_{2}}\Pi_{n=1}^{2}f(x_{n})}$$
and we showed that such an approximation yielded an $O(S^{-4\tau})$ error. 

More generally, we want to construct approximations of $$\sum_{x_{1}=x_{2}\neq...\neq x_{r}}^{N}\Pi_{n=1}^{r}f(x_{n}) \hspace{3pt} \mbox{and} \hspace{3pt} \sum_{x_{1}\neq x_{2}\neq...\neq x_{r}}^{N}\Pi_{n=1}^{r}f(x_{n})$$
For example, to attain a generalization of Lemma 2, we allow for the possibility (in the numerator) that $x_{1}=x_{2}=x_{3}$, but now we need to separate out three terms as opposed to two to achieve our desired cancellation; that is, we have 

$$ \frac{\sum_{x_{1}=x_{2}=x_{3}\neq...\neq x_{r}}\Pi_{n=1}^{r}f(x_{n})}{\sum_{x_{1}\neq x_{2}\neq...\neq x_{r}}\Pi_{n=1}^{r}f(x_{n})}\approx \frac{\sum_{x_{1}=x_{2}=x_{3}}\Pi_{n=1}^{3}f(x_{n})}{\sum_{x_{1},x_{2},x_{3}}\Pi_{n=1}^{3}f(x_{n})} \left[ \frac{\sum_{x_{4}\neq...\neq x_{r}}\Pi_{n=3}^{r}f(x_{n})}{\sum_{x_{4}\neq...\neq x_{r}}\Pi_{n=3}^{r}f(x_{n})} \right].$$
However, once we are computing $O(S^{-4\tau})$ terms, we also need to consider the case where $x_{1}=x_{2}$ and $x_{3}=x_{4}$, as such terms also turn out to contribute at $O(S^{-4\tau})$. To make this idea more rigorous, consider approximating 

$$
\begin{array}{l} 
\sum_{x_{1}=x_{2}\neq x_{3}\neq...\neq x_{r}}\Pi_{n=1}^{r}f(x_{n})= \\
\sum_{x_{1}=x_{2}, x_{3}\neq...\neq x_{r}}\Pi_{n=1}^{r}f(x_{n})-(r-2)\sum_{x_{1}=x_{2}= x_{3}\neq...\neq x_{r}}\Pi_{n=1}^{r}f(x_{n}).
\end{array}$$

To obtain a sufficiently high order estimate for the left hand side, we partition the right hand side into still more terms, motivated by the intuition that terms with more equal signs under a summation should be of higher order; thus, we consider 

$$ \sum_{x_{1}=x_{2}\neq x_{3}\neq...\neq x_{r}}\Pi_{n=1}^{r}f(x_{n})=\sum_{\substack{x_{1}=x_{2}, x_{3},\\x_{4}\neq...\neq x_{r}}}\Pi_{n=1}^{r}f(x_{n})-(r-3) \sum_{\substack{x_{1}=x_{2},\\ x_{3}=x_{4}\neq...\neq x_{r}}}\Pi_{n=1}^{r}f(x_{n})-$$ $$(r-2)\sum_{\substack{x_{1}=x_{2}= x_{3},\\x_{4}\neq...\neq x_{r}}}\Pi_{n=1}^{r}f(x_{n})+ \ldots $$  
where we have neglected to write out the final terms of still higher order.  

Now, since we want to keep the case where $x_{1}=x_{2}$ and $x_{3}=x_{4}$, we have to integrate out the $x_{4}$, which yields  

$$ \sum_{x_{1}=x_{2}\neq x_{3}\neq...\neq x_{r}}\Pi_{n=1}^{r}f(x_{n})=\sum_{\substack{x_{1}=x_{2},\\ x_{3},x_{4},x_{5}\neq...\neq x_{r}}}\Pi_{n=1}^{r}f(x_{n})-(r-3) \sum_{\substack{x_{1}=x_{2}, x_{3}=x_{4},\\x_{5}\neq...\neq x_{r}}}\Pi_{n=1}^{r}f(x_{n})$$ $$-(r-2)\sum_{\substack{x_{1}=x_{2}= x_{3},x_{4},\\x_{5}\neq...\neq x_{r}}}\Pi_{n=1}^{r}f(x_{n})-(r-4)\sum_{\substack{x_{1}=x_{2}, x_{3},\\x_{4}=x_{5}\neq...\neq x_{r}}}\Pi_{n=1}^{r}f(x_{n})+ \ldots $$ 

But we can simply relabel indices to observe that 
$$\sum_{x_{1}=x_{2}, x_{3},x_{4}=x_{5}\neq...\neq x_{r}}\Pi_{n=1}^{r}f(x_{n})=
%\sum_{x_{1}=x_{2},x_{4}=x_{5},x_{3}\neq x_{6}\neq...\neq x_{r}}\Pi_{n=1}^{r}f(x_{n})+\ldots= $$ $$
\sum_{\substack{x_{1}=x_{2},x_{3}=x_{4},\\x_{5}\neq...\neq x_{r}}}\Pi_{n=1}^{r}f(x_{n})+\ldots,$$ 
and we conclude that 
$$
\begin{array}{l}
 \sum_{x_{1}=x_{2}\neq x_{3}\neq...\neq x_{r}}\Pi_{n=1}^{r}f(x_{n})=\sum_{x_{1}=x_{2}, x_{3},x_{4},x_{5}\neq...\neq x_{r}}\Pi_{n=1}^{r}f(x_{n}) \\ \\
 -(2r-7) \sum_{\substack{x_{1}=x_{2}, x_{3}=x_{4},\\x_{5}\neq...\neq x_{r}}}\Pi_{n=1}^{r}f(x_{n})$$ $$-(r-2)\sum_{\substack{x_{1}=x_{2}= x_{3},x_{4},\\x_{5}\neq...\neq x_{r}}}\Pi_{n=1}^{r}f(x_{n})+\ldots
 \end{array}
 $$

So in fact, it suffices to factor out four indices.  Indeed, the general intuition is that if we are including $k$ equalities in our expansion, then we will need to factor out $2k$ indices, to consider the worst case scenario where $x_{1}=x_{2},x_{3}=x_{4},...,x_{2k-1}=x_{2k}$.  But we can then relabel indices to express terms with, say,  $x_{1}=x_{2},x_{3},x_{4}=x_{5},...,x_{2k}=x_{2k+1}$  using terms of the form 
$x_{1}=x_{2},x_{3}=x_{4},...,x_{2k-1}=x_{2k}$.

With these ideas in mind, we now proceed to derive the higher order approximation. 
We subdivide this task into a few steps, starting with an approximation lemma on sums of the form 
$\sum_{x_{1}\neq...\neq x_{r}}\Pi_{n=1}^{r}f(x_{n})$ 
and sums of the form  $\sum_{x_{1}=x_{2}\neq...\neq x_{r}}\Pi_{n=1}^{r}f(x_{n})$. 

\begin{lem}
\label{lem:reduce}
Let $f,g:\mathbb{N}\rightarrow[0,\infty)$ with $f(\cdot) \geq g(\cdot)$, $\sum_{i=1}^{N}f(i)=S$, $\max f(\cdot)=O(S^{\frac{1}{2}-\tau})$, and $r=O(S^{\frac{1}{2}-\tau})$, $r \geq 4$
and define $F=\Pi_{i=1}^{4}f(x_{i})$, $G=g(x_{1})\Pi_{i=2}^{4}f(x_{i})$,  
$$\tilde{\xi}=\sum_{\substack {x_{1},...,x_{4},\\ x_{5}\neq...\neq x_{r}}}^{N}\Pi_{i=1}^{r}f(x_{i})-(4r-10)\sum_{\substack {x_{1}=x_{2},x_{3},x_{4},\\ x_{5}\neq...\neq x_{r}}}^{N}\Pi_{i=1}^{r}f(x_{i}) \; \; \mbox{and} \; \; $$ 

$$
\begin{array}{rcl}
\tilde{\zeta}=\sum_{\substack {x_{1}=x_{2},x_{3},x_{4},\\ x_{5}\neq...\neq x_{r}}}^{N}g(x_{1})\Pi_{i=2}^{r}f(x_{i}) & - & (r-2)\sum_{\substack {x_{1}=x_{2}=x_{3},x_{4},\\ x_{5}\neq...\neq x_{r}}}^{N}g(x_{1})\Pi_{i=2}^{r}f(x_{i}) \vspace{0.1in} \\
& - & (2r-7)\sum_{\substack {x_{1}=x_{2},x_{3}=x_{4},\\ x_{5}\neq...\neq x_{r}}}^{N}g(x_{1})\Pi_{i=2}^{r}f(x_{i}).
\end{array}
     $$

Then 
\begin{equation}
\label{eq:tildexi}
\xi_0 := \sum_{x_{1}\neq...\neq x_{r}}^{N}\Pi_{i=1}^{r}f(x_{i})= \tilde{\xi}(1+O(S^{-4\tau}))
 \end{equation}
and 
\begin{equation}
\label{eq:tildezeta}
\zeta_0 := \sum_{x_{1}=x_{2}\neq ...\neq x_{r}}g(x_{1})\Pi_{i=2}^{r}f(x_{i})=\tilde{\zeta}(1+O(S^{-4\tau})).
\end{equation}
Furthermore, 
 \begin{equation}
 \label{eq:ratio0}
 \frac{\zeta_0}{\xi_0} = \frac{\sum_{x_{1}=x_{2},x_{3},x_{4}}G-(r-2)\sum_{x_{1}=x_{2}=x_{3},x_{4}}G-(2r-7)\sum_{x_{1}=x_{2},x_{3}=x_{4}}G}{\sum_{x_{1},..,x_{4}}F-(4r-10)\sum_{x_{1}=x_{2},x_{3},x_{4}}F}+{\color{black}O(S^{-\frac{1}{2}-5\tau})}.
 \end{equation}

\end{lem}

\begin{proof}
To derive the first equality, start by expressing $\xi_0$ as
\[
%    \begin{equation}\label{eq:neq}
\xi_0 = \sum_{x_{1}\neq...\neq x_{r}}^{N}\Pi_{i=1}^{r}f(x_{i})=\sum_{\substack {x_{1},\\ x_{2}\neq...\neq x_{r}}}^{N}\Pi_{i=1}^{r}f(x_{i})-(r-1)\sum_{\substack {x_{1}=x_{2}\neq\\ x_{3}\neq...\neq x_{r}}}^{N}\Pi_{i=1}^{r}f(x_{i}).
%\end{equation}
\]
Applying Lemma 2  to the final term yields
$$\xi_0=\sum_{\substack {x_{1},\\ x_{2}\neq...\neq x_{r}}}^{N}\Pi_{i=1}^{r}f(x_{i})-(r-1)\sum_{\substack {x_{1}=x_{2},x_{3},x_{4},\\ x_{5}\neq...\neq x_{r}}}^{N}\Pi_{i=1}^{r}f(x_{i})+O(S^{-4\tau})\sum_{x_{1}\neq...\neq x_{r}}^{N}\Pi_{i=1}^{r}f(x_{i})$$
where the error term comes from  
\begin{equation}
\label{eq:error}
O(S^{-2\tau})(r-1)\sum_{\substack {x_{1}=x_{2},x_{3},x_{4}\\ x_{5}\neq...\neq x_{r}}}^{N}\Pi_{i=1}^{r}f(x_{i})=O(S^{-4\tau})\sum_{x_{1}\neq...\neq x_{r}}^{N}\Pi_{i=1}^{r}f(x_{i}).
\end{equation}

We next repeat the same argument and relabel to obtain

\begin{equation}
\begin{array}{lll}
\xi_0 = \sum_{\substack {x_{1},x_{2},\\x_{3}\neq...\neq x_{r}}}^{N}\Pi_{i=1}^{r}f(x_{i})-(r-1)\sum_{\substack {x_{1}=x_{2},x_{3},x_{4},\\ x_{5}\neq...\neq x_{r}}}^{N}\Pi_{i=1}^{r}f(x_{i})-\\ \\(r-2)\sum_{\substack {x_{1},x_{2}=x_{3},x_{4},\\ x_{5}\neq...\neq x_{r}}}^{N}\Pi_{i=1}^{r}f(x_{i})+O(S^{-4\tau})\sum_{x_{1}\neq...\neq x_{r}}^{N}\Pi_{i=1}^{r}f(x_{i}) 
\\ \\ =  \sum_{\substack {x_{1},x_{2},\\x_{3}\neq...\neq x_{r}}}^{N}\Pi_{i=1}^{r}f(x_{i})-(2r-3)\sum_{\substack {x_{1}=x_{2},x_{3},x_{4},\\ x_{5}\neq...\neq x_{r}}}^{N}\Pi_{i=1}^{r}f(x_{i})+\\ \\O(S^{-4\tau})\sum_{x_{1}\neq...\neq x_{r}}^{N}\Pi_{i=1}^{r}f(x_{i}).
\end{array}
\end{equation}
Continuing in this fashion, 

$$\xi_0=\sum_{\substack {x_{1},x_{2},x_{3},\\x_{4}\neq...\neq x_{r}}}^{N}\Pi_{i=1}^{r}f(x_{i})-(3r-6)\sum_{\substack {x_{1}=x_{2},x_{3},x_{4},\\ x_{5}\neq...\neq x_{r}}}^{N}\Pi_{i=1}^{r}f(x_{i})+O(S^{-4\tau})\sum_{x_{1}\neq...\neq x_{r}}^{N}\Pi_{i=1}^{r}f(x_{i}).$$

The next, final step is a bit trickier so we include a little more detail.  First, we extract the $x_{4}$ out of the first summation of the above equation and apply (\ref{eq:error}), which yields 

$$
\begin{array}{rcl}
\\
\xi_0=\sum_{\substack {x_{1},x_{2},x_{3}x_{4},\\ x_{5}\neq...\neq x_{r}}}^{N}\Pi_{i=1}^{r}f(x_{i})-(3r-6)\sum_{\substack {x_{1}=x_{2},x_{3},x_{4},\\ x_{5}\neq...\neq x_{r}}}^{N}\Pi_{i=1}^{r}f(x_{i}) \\ \\
-(r-4)\sum_{\substack {x_{1},x_{2},x_{3},x_{4}=x_{5},\\ x_{6}\neq...\neq x_{r}}}^{N}\Pi_{i=1}^{r}f(x_{i})+O(S^{-4\tau})\sum_{x_{1}\neq...\neq x_{r}}^{N}\Pi_{i=1}^{r}f(x_{i}).
\end{array}
$$

But by another application of (\ref{eq:error}),

$$\sum_{\substack {x_{1},x_{2},x_{3},x_{4}=x_{5},\\ x_{6}\neq...\neq x_{r}}}^{N}\Pi_{i=1}^{r}f(x_{i})=\sum_{\substack {x_{2},x_{3},x_{4}=x_{5},\\ x_{1}\neq x_{6}\neq...\neq x_{r}}}^{N}\Pi_{i=1}^{r}f(x_{i})+O(S^{-4\tau})\sum_{x_{1}\neq...\neq x_{r}}^{N}\Pi_{i=1}^{r}f(x_{i}),$$ where by relabeling, the first term on the right hand side is the same as $$\sum_{\substack {x_{1}=x_{2},x_{3},x_{4},\\ x_{5}\neq...\neq x_{r}}}^{N}\Pi_{i=1}^{r}f(x_{i}).$$

So we conclude that

$$\xi_0=\sum_{\substack {x_{1},x_{2},x_{3},x_{4},\\ x_{5}\neq...\neq x_{r}}}^{N}\Pi_{i=1}^{r}f(x_{i})-(4r-10)\sum_{\substack {x_{1}=x_{2},x_{3},x_{4},\\ x_{5}\neq...\neq x_{r}}}^{N}\Pi_{i=1}^{r}f(x_{i})+O(S^{-4\tau})\sum_{x_{1}\neq...\neq x_{r}}^{N}\Pi_{i=1}^{r}f(x_{i}), $$
where the sum in the final term is itself $\xi_0$, such that the first part of the Lemma is established.

The second part of the Lemma follows analogously, while expression (\ref{eq:ratio0}) follows from cancellation of terms.
\end{proof}

\begin{thm}
Let $d_{max}=O(S^{\frac{1}{2}-\tau})$.  Then $$\frac{\|G_{d_{-i}}\|}{\|G_{d_{-j}}\|}=\frac{a_{i}}{a_{j}}\exp([a_{i}-a_{j}]\epsilon_{1}-[a_{i}^{2}-a_{j}^{2}]\epsilon_{2}+O(S^{\max(-6\tau,-\frac{1}{2}-3\tau)}))$$
where 
$$\epsilon_{1}=\frac{\beta_{2}+2\beta_{3}\alpha_{2}/\alpha_{1}^2}{(\beta_{1}+\beta_{2}\alpha_{2}/\alpha_{1}^2)^{2}} \; \mbox{and} \;  \epsilon_{2}=\frac{({\beta_{2}-\beta_{1}})^{2}}{2{\beta_{1}}^{4}}+\frac{{\beta_{3}\beta_{1}-2{\beta_{2}^{2}}}}{\beta_{1}^{4}}.$$

\end{thm}
\begin{proof}
From Theorem \ref{thm:S4} we know that 
\begin{equation}
\label{eq:exp}
\frac{\|G_{\mathbf{d_{-i}}}\|}{\|G_{\mathbf{d_{-j}}}\|}=\frac{a_i}{a_j}\exp((a_i-a_j)\frac{\beta_2-\beta_1}{\beta_1^2}=\frac{a_i\exp(a_i\frac{\beta_2}{\beta_1^2})}{a_j\exp(a_j\frac{\beta_2}{\beta_1^2})}(1+O(S^{-4\tau})).
\end{equation}
Note that the middle expression in (\ref{eq:exp}) satisfies the  assumption from Theorem \ref{thm:arborder1} that our approximation can be expressed as $h(i,\mathbf{d})/h(j,\mathbf{d})$ for some function $h$, and hence we can use Theorem \ref{thm:arborder1} to derive the next step on our expansion.
Similarly if we consider a different degree sequence such that $\|\mathbf{d}_0 -\mathbf{d}_1\|_\infty = 1$, $\|\mathbf{d}_0\|_1=\|\mathbf{d}_1\|_1$ and $\|\mathbf{d}_0-\mathbf{d}_1\|_1=O(S^{\frac{1}{2}-\tau})$, then we introduce a multiplicative error of $(1+O(S^{-6\tau}))$ as $$\frac{a_i}{a_j}\exp((a_i-a_j)\frac{\beta_2-\beta_1}{\beta_1^2}$$ only depends on the second moment of the out-degree sequence $\beta_2$.  We introduce the notation $\beta_2(\mathbf{d})$ to explicitly denote the dependence of $\beta_2$ on the degree sequence.  

Consequently, the difference between $|\beta_2(\mathbf{d}_1)-\beta_2(\mathbf{d}_0)|=O(S^{1-2\tau})$, hence 

$$\frac{\frac{a_i}{a_j}\exp((a_i-a_j)\frac{\beta_2(\mathbf{d_1})}{\beta_1^2})}{\frac{a_i}{a_j}\exp((a_i-a_j)\frac{\beta_2(\mathbf{d}_0)}{\beta_1^2})}=\exp((a_i-a_j)\frac{\beta_2(\mathbf{d_1})-\beta_2(\mathbf{d_0})}{\beta_1^2})=
\exp(O(S^{\frac{1}{2}-\tau}O(S^{1-2\tau})/O(S^2)) = \exp(O(S^{-\frac{1}{2}-3\tau})).$$

We now proceed as we did in Theorem \ref{thm:S4}.  We need to evaluate $$\delta_{i} =\frac{(a_{i}-1)\sum_{x_{1}\neq...\neq x_{a_{i-1}}=x_{a_i}} \Pi_{k=1}^{a_i} f(\mathbf{e}_{x_{k}},\mathbf{e}_{u_{k}},\mathbf{b}) }{\sum_{x_{1}\neq...\neq x_{a_{i-1}}\neq x_{a_i}} \Pi_{k=1}^{a_i} f(\mathbf{e}_{x_{k}},\mathbf{e}_{u_{k}},\mathbf{b})}.$$ 
We multiply our numerator and denominator by $\Pi_{k=1}^{a_i} b_{u_k}\exp(b_{u_k}\frac{\alpha_2}{\alpha_1^2})$, where $\alpha_1,\alpha_2$ are the moments of the degree sequence $\mathbf{d}$, we drop the dependence of $f$ on $\mathbf{e}_{u_{k}}$, and we define $f(\mathbf{e}_{x_{k}},\mathbf{b})=b_{x_{k}}\exp(b_{x_k}\frac{\alpha_2}{\alpha_1^2})$ to obtain (from (\ref{eq:exp}) 

$$\delta_{i}=\frac{(a_{i}-1)\sum_{x_{1}\neq...\neq x_{a_{i-1}}=x_{a_i}} \Pi_{k=1}^{a_i} f(\mathbf{e}_{x_{k}},\mathbf{b}) }{\sum_{x_{1}\neq...\neq x_{a_{i-1}}\neq x_{a_i}} \Pi_{k=1}^{a_i} f(\mathbf{e}_{x_{k}},\mathbf{b})}.$$
We can now invoke  Lemma \ref{lem:reduce} to simplify the above summation.  Further algebraic simplification yields the desired result.

\end{proof}

To indicate how the estimates that we derive can, in theory, be continued indefinitely,  we conclude this section by providing a higher order approximation for a ratio of the numbers of graphs that realize two different degree sequences that has an $O(S^{-8\tau})$ correction.
As before, we need a more refined version of Lemma 3 to derive this approximation.  We omit the proof, since the same proof technique that derived Lemma 3 yields the proof of Lemma 4.  On the  author's webpage, we provide code that computes (and proves) the refined variants of Lemma 3 up to arbitrary order.
%\href{http://www.pitt.edu/~dab176/publications.html}
\begin{lem} Suppose that $g:\mathbb{N} \to [0,\infty),  f : \mathbb{N} \to [1,\infty)$ with $f(\cdot)\geq  g(\cdot)$,  that $\sum_{i=1}^{N}f(i)=S$,   %$f(\cdot)\geq 1$ and 
$\max f(\cdot)=O(S^{\frac{1}{2}-\tau})$, and that $r=O(S^{\frac{1}{2}-\tau})$, {\color{black}$r\geq 6$}.
Define $\chi=\sum_{x_{1}\neq...\neq x_{r}}^{N}\Pi_{i=1}^{r}f(x_{i})$ and $\kappa=\sum_{x_{1}=x_{2}\neq...\neq x_{r}}^{N} g(x_{1})\Pi_{i=2}^{r}f(x_{i})$
Furthermore, for notational convenience let $F=\Pi_{i=1}^{6}f(x_{i})$ and $G=g(x_{1})\Pi_{i=2}^{6}f(x_{i})$.
Then \[\frac{\kappa}{\chi}=\frac{\kappa_{1}}{\chi_{1}}+O(S^{-\frac{1}{2}-7\tau})\]

where $$ \chi_{1}=\sum_{x_{1},..,x_{6}}F-(6r-21)\sum_{\substack{x_{1}=x_{2}\\x_{3}..,x_{6}}}F+(9r^{2}-58r+69)\sum_{\substack{x_{1}=x_{2},x_{3}=x_{4}\\x_{5},x_{6}}}F+(6r^{2}-48r+112)\sum_{\substack{x_{1}=x_{2}=x_{3}\\x_{4}..,x_{6}}}F$$
and $$\kappa_{1}=\sum_{x_{1}=x_{2},x_{3}..,x_{6}}G-(r-2)\sum_{\substack{x_{1}=x_{2}=x_{3}\\x_{4}..,x_{6}}}G-(4r-18)\sum_{\substack{x_{1}=x_{2},x_{3}=x_{4}\\x_{5},x_{6}}}G+(r^{2}-5r+6)\sum_{\substack{x_{1}=x_{2}=x_{3}=x_{4}\\x_{5},x_{6}}}G$$
$$+(3r^{2}-21r+30)\sum_{\substack{x_{1}=x_{2}=x_{3}\\x_{4}=x_{5},x_{6}}}G+(4r^{2}-40r+104)\sum_{\substack{x_{1}=x_{2}\\x_{3}=x_{4}=x_{5},x_{6}}}G+(2r^{2}-15r+21)\sum_{\substack{x_{1}=x_{2}\\x_{3}=x_{4},x_{5}=x_{6}}}G.     $$
\end{lem}
We conclude this section with a statement of the corresponding more refined approximation.
\begin{thm}
Let $d_{max}=O(S^{\frac{1}{2}-\tau})$ and recall that $\beta_{k}=\sum_{k=1}^{N}b_{i}^{k}$. We have $$\frac{\|G_{d_{-i}}\|}{\|G_{d_{-j}}\|}=\frac{a_{i}}{a_{j}}\exp([a_{i}-a_{j}]\epsilon_{1}-[a_{i}^{2}-a_{j}^{2}]\epsilon_{2}+[a_{i}^{3}-a_{j}^{3}]\epsilon_{3}+O(S^{\max(-1-2\tau,-8\tau)}))$$
where we define 
\[
\begin{array}{c}
\epsilon_{1}=\frac{E_{\mathbf{b}}[f(x)f(x-1)]}{E_{\mathbf{b}}[f(x)]^{2}}+\frac{E_{\mathbf{b}}[f(x)f(x-1)]^{2}-5E_{\mathbf{b}}[f(x)^{2}]E_{\mathbf{b}}[f(x)f(x-1)]+3E_{\mathbf{b}}[f(x)^{2}f(x-1)]E_{\mathbf{b}}[f(x)]}{E_{\mathbf{b}}[f(x)]^{4}}, \vspace{0.1in} \\
\epsilon_{2}=\frac{-2E_{\mathbf{b}}[f(x)^{2}]E_{\mathbf{b}}[f(x)f(x-1)]+\frac{1}{2}E_{\mathbf{b}}[f(x)f(x-1)]^{2}+E_{\mathbf{b}}[f(x)^{2}f(x-1)]E_{\mathbf{b}}[f(x)]}{E_{\mathbf{b}}[f(x)]^{4}},   \; \; \mbox{and} \; \; \vspace{0.1in}  \\
\epsilon_{3} = \frac{-\frac{107}{3}\beta_{2}^{3}-\frac{11}{2}\beta_{1}\beta_{2}\beta_{3}+\beta_{2}\beta_{4}}{\beta_{1}^{6}}
\end{array}
\]
for 
\[
\begin{array}{c}
E_{\mathbf{b}}[f(x)]:=\sum_{x\in\mathbf{b}}f(x)\; f(x)=x+x^{2}\eta_{1}+\frac{x^{3}\eta_{1}^{2}}{2}-x^{3}\eta_{2}, \vspace{0.1in} \\
\eta_{1} := \frac{\alpha_{2}+2\alpha_{3}\beta_{2}/\alpha_{1}^2}{(\alpha_{1}+\alpha_{2}\beta_{2}/\alpha_{1}^2)^{2}}, \; \; \mbox{and} \; \; 
\eta_{2} := \frac{({\alpha_{2}-\alpha_{1}})^{2}}{2{\alpha_{1}}^{4}}+\frac{{\alpha_{3}\alpha_{1}-2{\alpha_{2}^{2}}}}{\alpha_{1}^{4}}.
\end{array}
\]

\end{thm}

\section{On the Assumptions of Theorem \ref{thm:arborder1}}
\label{section:assumptions}
To prove that the assumptions made in Theorem \ref{thm:arborder1} (and in Theorem \ref{thm:arborder2}) hold, it is helpful to have  results to simplify the analysis of ratios of sums of the form $\sum_{x_1\neq...\neq x_r}\Pi_{i=1}^{r}f(x_i)$ that are more general than the results we proved in Section 5. First, it is helpful to introduce some notation.

\begin{mydef}
\label{def:alg}
Define the set of natural numbers $I = \{ 1, \ldots, r \}$.  For any collection $A_=$ of disjoint subsets of $I$, we say that a sequence $\mathbf{x}\in\mathbb{N}^{r}$ satisfies $A_=$ if for every $\{ i_1, \ldots, i_m \} \in A_=$, we have $x_{i_1} = x_{i_2} = \ldots = x_{i_m}$.
Given that $\mathbf{x}$ satisfies $A_=$ and that $A_{\neq}$ is a subset of $I$, we say that $\mathbf{x}$ satisfies $A_{\neq}$ if for all $i,j \in A_{\neq}$ such that $i,j$ are not both contained in any $A \in A_=$, we have $x_i \neq x_j$.
We denote the set of all $\mathbf{x}$ that satisfy both $A_{=}$ and $A_{\neq}$ by the notation $A_= \otimes A_{\neq}$.
\end{mydef}

Before proceeding with the theorem statements that enable us to construct approximations for ratios of sums of the form $\sum_{x_1\neq...\neq x_r}\Pi_{i=1}^{r}f(x_i)$, we wish to motivate some of the notation that we use.  To approximate $\sum_{x_1\neq...\neq x_r}\Pi_{i=1}^{r}f(x_i)$, we wish to  remove some (fixed) number of $x_i's$ from the list of required inequalities.  For example, we can express $\sum_{x_1\neq x_2} f(x_1)f(x_2) = \sum_{x_1,x_2}f(x_1)f(x_2) - \sum_{x_1=x_2}f(x_1)f(x_2)$, where in the first term on the right hand side, we no longer have to worry about the `does not equal' relationship of $x_1$ and $x_2$ found on the left hand side.  To help us identify the dominating terms, we introduce sets of the form $A_=^{(i,c)}\otimes A_{\neq}^{(i,c)}$, where the $i$ denotes the number of equal signs found under the summation and $c$ denotes a particular instance where  $i$ of these variables are equal.  For example, suppose we want to simplify $\sum_{x_1\neq x_2\neq x_3} f(x_1)f(x_2)f(x_3) = \sum_{x_1, x_2\neq x_3} f(x_1)f(x_2)f(x_3) - \sum_{x_1=x_2\neq x_3} f(x_1)f(x_2)f(x_3)- \sum_{x_1=x_3\neq x_2} f(x_1)f(x_2)f(x_3)$.  On the right hand side there are two summations involving one equality under the summation.  Hence we can define $A_=^{(1,0)}\otimes A_{\neq}^{(1,0)}=\{\{1,2\}\}\otimes \{1,2,3\}$, $A_=^{(1,1)}\otimes A_{\neq}^{(1,1)}=\{\{1,3\}\}\otimes \{1,2,3\}$ representing those two summations.  Note that in each of these examples $i=1$, but since there are two possible choices, $c$ can be either $0$ or $1$.  
Now we present two results regarding the sums of interest.  Since the proofs of the two theorems are essentially identical, we only prove one of them.

\begin{thm}
\label{thm:alg}
Suppose that $\sum_{x_m}f(x_m)=S$ and that $\max f(x_m) = O(S^{\frac{1}{2}-\tau})$ and $r = O(S^{\frac{1}{2}-\tau})$.
Consider $\sum_{x_1\neq...\neq x_r} \Pi_{m=1}^{r} f(x_m)$.  Fix a natural number $k$ such that $2k \leq r$ (where $2k$ is the number of variables that we remove from the list of inequalities so that they can equal other variables).  Then for all $j\leq k$ we can write
\begin{equation}
\label{eq:eqneqthm}
\sum_{x_1\neq...\neq x_r} \Pi_{m=1}^{r} f(x_m)=\sum_{i=0}^{j}\sum_{c=0}^{h(i,k)}p_{(i,c)}\sum_{A_=^{(i,c)}\otimes A_{\neq}^{(i,c)}}\Pi_{m=1}^{r} f(x_m)
\end{equation}
where for each $i$, %$h:\mathbb{N}^{2}\rightarrow \mathbb{N}$ 
$h(i,k)+1$ represents a number of arrangements of variables with $i$ variables equal to each other,  
%, such that $i$ variables are equal to each other for each of $h(i,k)+1$ arrangements
 and $h(0,k)=0$;  for all $i, c$, $p_{(i,c)}$ is a polynomial in $r$;  and 
for all $i<j$ and for all $c$, $A^{(i,c)}_=$ consists of a collection of subsets of $\{ 1, \ldots, 2k \}$ with $\sum_{A\in A^{(i,c)}_=}(|A|-1)=i$ and $A^{(i,c)}_{\neq}=\{2k+1,...,r\}$.  
In addition,
\begin{itemize}
\item for all $c$, $\sum_{A\in A^{(j,c)}_=}(|A|-1)=j$,
\item  $A^{(j,c)}_{\neq}=\{s,s+1,...,r\}$ for some $s\leq 3k$, and 
\item if $A\in A^{(j,c)}_=$ has a nontrivial intersection with $A^{(j,c)}_{\neq}$ then $A = \{s,s+1,...,s+t\}$ for some $t$ such that $s+t<r$; 
moreover, there can only be one such $A\in A^{(j,c)}_=$ that has a nontrivial intersection with $A^{(j,c)}_{\neq}$.
\end{itemize}
Finally, $$\frac{p_{(i,c)}\sum_{A_=^{(i,c)}\otimes A_{\neq}^{(i,c)}}\Pi_{m=1}^{r} f(x_m)}{\sum_{x_{1},...,x_{2k},x_{2k+1}\neq...\neq x_{r}}\Pi_{m=1}^{r} f(x_m)}=O(S^{-2i\tau}).$$

\end{thm}

\begin{thm}
\label{thm:alg2}
Suppose that $\sum_{x_m}f(x_m)=S$ and that $\max f(x_m) = O(S^{\frac{1}{2}-\tau})$ and $r = O(S^{\frac{1}{2}-\tau})$.
Consider $\sum_{x_1=x_2\neq...\neq x_r} g(x_1)\Pi_{m=2}^{r} f(x_m)$.  For simplicity, we assume that the elements in $A_=$ are disjoint.  Fix a number $k$ (where $2k$ is the number of variables removed from the list of inequalities). Then for all $j\leq k$ we can write
$$\sum_{x_1=x_2\neq...\neq x_r}g(x_1) \Pi_{m=2}^{r} f(x_m)=\sum_{i=0}^{j}\sum_{c=0}^{h(i,k)}p_{(i,c)}\sum_{A_=^{(i,c)}\otimes A_{\neq}^{(i,c)}}g(x_1)\Pi_{m=2}^{r} f(x_m)$$
where for each $i$, %$h:\mathbb{N}^{2}\rightarrow \mathbb{N}$ 
$h(i,k)+1$ represents a number of arrangements of variables with $i$ variables equal to each other, and 
$h(0,k)=0$; for all $i, c$, $p_{(i,c)}$ is a polynomial in $r$;  and for all $i<j$ and for all $c$, $A^{(i,c)}_=$ consists of a collection of subsets of $\{1,...,2k\}$, $A^{(i,c)}_{\neq}=\{2k+1,...,r\}$ with $\sum_{A\in A^{(i,c)}_=}(|A|-1)=i+1$.
In addition,
\begin{itemize}
\item for all $c$, $\sum_{A\in A^{(j,c)}_=}(|A|-1)=j+1$,
\item $A^{(j,c)}_{\neq}=\{s,s+1,...,r\}$  for some  $s\leq 3k$, and 
\item if $A\in A^{(j,c)}_=$ has a nontrivial intersection with $A^{(j,c)}_{\neq}$ then $A = \{s,s+1,...,s+t\}$ for some $t$ such that $s+t<r$; 
moreover, there can only be one such $A\in A^{(j,c)}_=$ that has a nontrivial intersection with $A^{(j,c)}_{\neq}$.
\end{itemize}

Finally, for all $i\leq j$, $$\frac{p_{(i,c)}\sum_{A_=^{(i,c)}\otimes A_{\neq}^{(i,c)}}g(x_1)\Pi_{m=2}^{r} f(x_m)}{\sum_{x_{1}=x_{2},x_{3},...,x_{2k},x_{2k+1}\neq...\neq x_{r}} g(x_1)\Pi_{m=2}^{r} f(x_m)}=O(S^{-2i\tau}).$$
\end{thm}

\begin{proof}[Proof of Theorem \ref{thm:alg}]  We proceed by induction.
For the base case of $j=1$,  we start by showing that the stated conditions hold in the case of $i=j$ (such that the bulleted list in the theorem statement applies) and then we come back to $(i,c)=(0,0)$.   We have seen previously that 
$$\sum_{x_1\neq...\neq x_r} \Pi_{m=1}^{r} f(x_m)=\sum_{x_1,x_2\neq...\neq x_r} \Pi_{m=1}^{r} f(x_m)-(r-1)\sum_{x_1=x_2\neq...\neq x_r} \Pi_{m=1}^{r} f(x_m).$$
We denote the set of $\mathbf{x}$ over which the final summation occurs as $A_=^{(1,0)}\otimes A_{\neq}^{(1,0)}$ where $A_=^{(1,0)}=\{\{1,2\}\}$, the variables that equal each other, and $A_{\neq}^{(1,0)}=\{1,2,...,r\}$, the variables that are required not to equal some (at least one) other variable.   The coefficient $p_{(1,0)}=-(r-1)$ is also a polynomial in $r$ and $\sum_{A\in A^{(1,0)}_=}(|A|-1)=1$ (where we use the superscript $(1,0)$ to denote that $\sum_{A\in A^{(1,0)}_=}(|A|-1)=1$ and the $0$ keeps track of this particular instance where $\sum_{A\in A_=}(|A|-1)=1$).  

We next repeat this argument for $$\sum_{x_1,x_2\neq...\neq x_r} \Pi_{m=1}^{r} f(x_m)=\sum_{x_1,x_2,x_3\neq...\neq x_r} \Pi_{m=1}^{r} f(x_m)-(r-2)\sum_{x_1,x_2=x_3\neq...\neq x_r} \Pi_{m=1}^{r} f(x_m)$$ and in context to the second summation on the right hand side, define $p_{(1,1)}=-(r-2)$,$A_=^{(1,1)}=\{\{2,3\}\}$, the variables that equal one another, and $A_{\neq}^{(1,1)}=\{2,3,4,..,r\}$, the variables that are required not to equal some other variable (where we use the superscript $(1,1)$ to denote that $\sum_{A\in A^{(1,1)}_=}(|A|-1)=1$ and the $1$ keeps track of this new instance where $\sum_{A\in A_=}(|A|-1)=1$).
We stop once we reach \begin{equation}\label{eq:finalbase}\sum_{x_1,x_2,..,x_{2k}\neq...\neq x_r} \Pi_{m=1}^{r} f(x_m)=\sum_{x_1,...,x_{2k},x_{2k+1}\neq...\neq x_r} \Pi_{m=1}^{r} f(x_m)-(r-2k)\sum_{x_1,...,x_{2k}=x_{2k+1}\neq...\neq x_r} \Pi_{m=1}^{r} f(x_m). 
\end{equation}

Finally, based on the the first sum on the right hand side of equation (\ref{eq:finalbase}),  we define $p_{(0,0)}=1$,$A_=^{(0,0)}=\emptyset$, as there are no variables that are required to equal one another, $A_{\neq}^{(0,0)}=\{2k+1,...,r\}$,the variables that are required not to equal some other variable, as $\sum_{A\in A^{(0,0)}_=}(|A|-1)=0$,  and note that only variables with an index of $2k+1$ or greater are constrained so that they cannot equal one another. 
Note that in the theorem statement, the number of instances where we have sets of the form $A_{=}^{(1,x)}\otimes A_{\neq}^{(1,x)}$ is $2k$, but since we start with $c=0$, the maximumum value of $c$ will be $2k-1$ (and hence define $h(1,k)=2k-1$.)

From the above inductive construction, we know that for all $c\in\{0,1,..,2k-1\}$,  
\begin{equation}\label{eq:polybound}
|p_{(1,c)}|\leq r.
\end{equation}.  In addition, since the range of $f$ consists of the non-negative real numbers, we know that for all $c$, \begin{equation}\label{eq:ineqalg}\sum_{A_=^{(1,c)}\otimes A_{\neq}^{(1,c)}}\Pi f(x_m)\leq \sum_{A_=^{(1,2k-1)}\otimes A_{\neq}^{(1,2k-1)}}\Pi f(x_m)
\end{equation} as the number of variables that are required not to equal some other variable in $A_{\neq}^{(1,2k-1)}$ is smaller than the analogous list for any other instance of $A_{\neq}^{(1,c)}$.  Hence by inequalities (\ref{eq:polybound}) and (\ref{eq:ineqalg}), we find that 

 $$\frac{|\sum_{c=0}^{2k-1}p_{(1,c)}\sum_{A_=^{(1,c)}\otimes A_{\neq}^{(1,c)}}\Pi f(x_m)|}{\sum_{x_1,...,x_{2k},x_{ 2k+1}\neq...\neq x_{r}}\Pi f(x_{m})}\leq 2kr\frac{\sum_{A_=^{(1,2k-1)}\otimes A_{\neq}^{(1,2k-1)}}\Pi f(x_m)}{\sum_{x_1,...,x_{2k},x_{ 2k+1}\neq...\neq x_{r}}\Pi f(x_{m})}=$$

 $$2kr\frac{\sum_{x_1,...,x_{2k-1},x_{2k}=x_{2k+1}\neq ...\neq x_{r}}\Pi f(x_m)}{\sum_{x_1,...,x_{2k},x_{2k+1}\neq...\neq x_{r}}\Pi f(x_{m})}=2kr\frac{S^{2k-1}\sum_{x_{2k}=x_{2k+1}\neq ...\neq x_{r}}\Pi f(x_m)}{S^{2k}\sum_{x_{2k+1}\neq...\neq x_{r}}\Pi f(x_{m})}\leq$$
 
$$2krf(x_{max})\frac{\sum_{x_{2k+1}\neq...\neq x_{r}}\Pi f(x_m)}{S\sum_{x_{2k+1}\neq...\neq x_{r}}\Pi f(x_{m})}
=\frac{2krf(x_{max})}{S}=2k\frac{O(S^{1-2\tau})}{S}=O(S^{-2\tau})$$ 
where all products without indices labeled are taken over the set of $m$ values in the index set of the summation that precedes them.  Thus, we have attained our desired results for $j=1$.
 
To proceed with induction, we now assume that the inductive statement holds for $j=n-1$ and prove that it is true for $j=n$, provided $n\leq k$.
The inductive hypothesis when $j=n-1$ yields 
$$\sum_{x_1\neq...\neq x_r} \Pi_{m=1}^{r} f(x_m)=\sum_{i=0}^{n-1}\sum_{c=0}^{h(i,k)}p_{(i,c)}\sum_{A_=^{(i,c)}\otimes A_{\neq}^{(i,c)}}\Pi_{m=1}^{r} f(x_m).$$

To prove that the inductive statement holds  for $j=n$ we manipulate the sets of the form $A_=^{(n-1,c)}\otimes A_{\neq}^{(n-1,c)}$ in our inductive hypothesis (where $j=n-1$).  From our inductive hypothesis we know that 
$\sum_{A\in A^{(n-1,c)}_=}(|A|-1)=n-1$;  for some $s\leq 3k$, $A^{(n-1,c)}_{\neq}=\{s,s+1,...,r\}$; and there exists at most one set in $A^{(n-1,c)}_=$, of the form $\{s,s+1,...,s+t\}$, that has a nontrivial intersection with $A^{(n-1,c)}_{\neq}$.

We proceed as in the base case. 
\newline 
\textbf{Case 1:} Suppose that indeed for some $s,t$, $\{s,s+1,...,s+t\}\in A^{(n-1,c)}_=$ and $A^{(n-1,c)}_{\neq}=\{s,s+1,...,r\}$. By assumption we can write $A_{=}^{(n-1,c)}=A_{1}\cup \{s,s+1,...,s+t\}$ where we define $A_{1}$ to be $A_{=}^{(n-1,c)}-\{s,s+1,...,s+t\}$. Since $x_s=x_{s+1}=...=x_{s+t}$, we have the following equality:

\begin{equation} \label{eq:alg}
p_{(n-1,c)}\sum_{\substack{A_{=}^{(n-1,c)}\otimes \\A_{\neq}^{(n-1,c)}}}\Pi f(x_m)=p_{(n-1,c)}\sum_{\substack{A_{=}^{(n-1,c)}\otimes\\ \{s+t+1,...,r\}}}\Pi f(x_m)-p_{(n-1,c)}(r-s-t)\sum_{\substack{A_{1}\cup\\ \{s,...,s+t+1\}\otimes \{s,...,r\}}}\Pi f(x_m).
\end{equation}

The last term on the right hand side of (\ref{eq:alg})  is a sum over an index set with $n$ equalities.
Analogous to the proof of the base case, we express this index set 
as $A_=^{(n,c_{*})}\otimes A_{\neq}^{(n,c_{*})}$ for some label $c_{*}$. Of course, the product of the two polynomial coefficients  of this term is a polynomial as well, which we can take as $p_{(n,c_{*})}$.
To see how this works, we proceed for each $c$ depending on whether $s+t\leq 2k$ or $s+t>2k$.  

\textbf{Case 1a}: $s+t\leq 2k$ \newline
Analogous to the proof in the base case, the expression (\ref{eq:alg}) can be decomposed as  
$$\sum_{\substack{A_{=}^{(n-1,c)}\otimes\\ \{s+t+1,...,r\}}}\Pi f(x_m)=\sum_{\substack{A_{=}^{(n-1,c)}\otimes \\ \{s+t+2,...,r\}}}\Pi f(x_m)-(r-s-t-1)\sum_{\substack{A_{=}^{(n-1,c)}\cup\\ \{s+t+1,s+t+2\}\otimes \{s+t+1,...,r\}}}\Pi f(x_m).$$
We use the last term to construct $A_=^{(n,c_{*})}\otimes A_{\neq}^{(n,c_{*})}$ for some $c_*$.
(Note that $\sum_{A\in A^{(n,c_*)}_=}(|A|-1)=n$).
We continue this process until we reach
\begin{equation}
\label{eq:n-1}
\sum_{A_{=}^{(n-1,c)}\otimes \{2k,...,r\}}\Pi f(x_m)=\sum_{A_{=}^{(n-1,c)}\otimes \{2k+1,...,r\}}\Pi f(x_m)-(r-2k)\sum_{A_{=}^{(n-1,c)}\cup\{2k,2k+1\}\otimes \{2k,...,r\}}\Pi f(x_m).
\end{equation}
We note that with $j=n$, the $i=n-1$ terms in (\ref{eq:eqneqthm}) may be different from the $i=n-1$ terms with $j=n-1$.
Indeed, based on (\ref{eq:n-1}), we now redefine (for $j=n$) $A_=^{(n-1,c)}\otimes A_{\neq}^{(n-1,c)} = A_=^{(n-1,c)}\otimes \{2k+1,...,r\}.$

\textbf{Case 1b}: $s+t >2k$ \newline
Now since by the assumption of the inductive hypothesis, $\sum_{A\in A_=^{(n-1,c)}}(|A|-1)=n-1\leq k-1\implies \sum_{A\in A_=^{(n-1,c)}}|A|\leq 2k-2$, there are at least $s+t-2k+2$ values in $\{1,...,s+t\}$ that are not in $A$ for all $A\in A_=^{(n-1,c)}$. We then create a relabeling where all values greater than $s+t$ are left alone (they map to themselves) and that the values that %do have the property that $t\notin A$ 
are not in $A$ for all $A\in A_=^{(n-1,c)}$ include $\{2k+1,...,s+t\}$.  (Hence the largest element in all of the sets in $A_=^{(n-1,c)}$ is bounded by $2k$.)

Expression (\ref{eq:alg}) can be decomposed as 

$$\sum_{A_{=}^{(n-1,c)}\otimes \{s+t+1,...,r\}}\Pi f(x_m)=\sum_{A_{=}^{(n-1,c)}\otimes \{s+t,...,r\}}\Pi f(x_m)+(r-s-t)\sum_{A_{=}^{(n-1,c)}\cup\{s+t,s+t+1\}\otimes \{s+t,...,r\}}\Pi f(x_m)$$

As in Case 1a, we use the last term to construct $A_=^{(n,c_{*})}\otimes A_{\neq}^{(n,c_{*})}$ for some $c_*$.  We repeat this equality until we attain

$$\sum_{A_{=}^{(n-1,c)}\otimes \{2k+2,...,r\}}\Pi f(x_m)=\sum_{A_{=}^{(n-1,c)}\otimes \{2k+1,...,r\}}\Pi f(x_m)+(r-2k-1)\sum_{A_{=}^{(n-1,c)}\cup\{2k,2k+1\}\otimes \{2k,...,r\}}\Pi f(x_m)$$

\textbf{Case 2:} Suppose that $A^{(n-1,c)}_{\neq}=\{s,s+1,...,r\}$ and for all $t$, $\{s,s+1,...,s+t\}\notin A^{(n-1,c)}_=$.  We again consider two subcases. \newline

\textbf{Case 2a:} $s\leq 2k$.

To achieve the desired form, we rewrite 

$$\sum_{A_{=}^{(n-1,c)}\otimes \{s,...,r\}}\Pi f(x_m)=\sum_{A_{=}^{(n-1,c)}\otimes \{s+1,...,r\}}\Pi f(x_m)-(r-s-1)\sum_{A_{=}^{(n-1,c)}\cup\{s,s+1\}\otimes \{s,...,r\}}\Pi f(x_m).$$
We use the last term to construct $A_=^{(n,c_{*})}\otimes A_{\neq}^{(n,c_{*})}$ for some $c_*$, proceeding inductively until we reach 

$$\sum_{A_{=}^{(n-1,c)}\otimes \{2k,...,r\}}\Pi f(x_m)=\sum_{A_{=}^{(n-1,c)}\otimes \{2k+1,...,r\}}\Pi f(x_m)-(r-2k-1)\sum_{A_{=}^{(n-1,c)}\cup\{2k,2k+1\}\otimes \{2k,...,r\}}\Pi f(x_m).$$

\textbf{Case 2b:} $s > 2k$. 

If $A_{=}^{(n-1,c)}$ contains sets with elements that are greater than $2k$, we can again perform a relabeling scheme as in Case 1b such that all of the elements in $A_{=}^{(n-1,c)}$ are guaranteed to be bounded above by $2k$:
$$\sum_{A_{=}^{(n-1,c)}\otimes \{s,...,r\}}\Pi f(x_m)=\sum_{A_{=}^{(n-1,c)}\otimes \{s-1,...,r\}}\Pi f(x_m)+(r-s-1)\sum_{A_{=}^{(n-1,c)}\cup\{s-1,s\}\otimes \{s-1,...,r\}}\Pi f(x_m).$$
We use the last term to construct $A_=^{(n,c_{*})}\otimes A_{\neq}^{(n,c_{*})}$ for some $c_*$ and we repeat this equality until we attain

$$\sum_{A_{=}^{(n-1,c)}\otimes \{2k+2,...,r\}}\Pi f(x_m)=\sum_{A_{=}^{(n-1,c)}\otimes \{2k+1,...,r\}}\Pi f(x_m)+(r-2k-1)\sum_{A_{=}^{(n-1,c)}\cup\{2k,2k+1\}\otimes \{2k,...,r\}}\Pi f(x_m).$$

To complete the proof, we need to verify the following properties from our theorem statement:
\begin{itemize}
\item $A^{(i,c)}_=$ is a collection of disjoint subsets of $\{ 1, \ldots, 2k \}$.
This fact follows from the conclusion of each of the cases in the above inductive argument.
\item $\sum_{A\in A^{(i,c)}_=}(|A|-1)=i$. 
For each $A^{(i,c_*)}_=$ constructed from an  $A^{(i-1,c)}_=$ ,  it is easy to show that $\sum_{A\in A^{(i-1,c)}_=}(|A|-1)=i-1$, and then the inductive step adds two more terms to the sum to yield  $\sum_{A\in A^{(i,c_*)}_=}(|A|-1)=i$.
\item For $i< j$, $A^{(i,c)}_{\neq}=\{2k+1,...,r\}$.
This fact is easily verified (and follows trivially) from the construction through the inductive hypothesis.
\item $A^{(j,c)}_{\neq}=\{s,s+1,...,r\}$ for some $s\leq 3k$. 
This fact is also easily verified (and follows trivially) from the proof by induction above.
\item If $A\in A^{(j,c)}_=$ has a nontrivial intersection with $A^{(j,c)}_{\neq}$ then $A = \{s,s+1,...,s+t\}$; 
in addition, there can only be one such $A\in A^{(j,c)}_=$ that has a nontrivial intersection with $A^{(j,c)}_{\neq}$.
This fact is also easily verified (and follows trivially) from the proof by induction above.
\item \begin{equation}\label{eq:algsize}\frac{p_{(i,c)}\sum_{A_=^{(i,c)}\otimes A_{\neq}^{(i,c)}}\Pi_{m=1}^{r} f(x_m)}{\sum_{x_{1},...,x_{2k},x_{2k+1}\neq...\neq x_{r}}\Pi_{m=1}^{r} f(x_m)}=O(S^{-2i\tau}).
\end{equation} \newline
The proof for this result is tricky but also follows from the inductive argument.  In the inductive hypothesis, assume for all sets $A_=^{(i-1,c)}\otimes A_{\neq}^{(i-1,c)}$ and corresponding polynomials $p_{(i-1,c)}$ that equation (\ref{eq:algsize}) holds.  Consider what happens when we construct sets of the form  $A_=^{(i,c_*)}\otimes A_{\neq}^{(i,c_*)}$ and the corresponding polynomials $p_{(i,c_*)}$ from $A_=^{(i-1,c)}\otimes A_{\neq}^{(i-1,c)}$.  We know from the above inductive proof that $ p_{(i,c_*)}\leq rp_{(i-1,c)}$.  Furthermore, there is an additional equal sign under the summation with the index  set $A_=^{(i,c_*)}\otimes A_{\neq}^{(i,c_*)}$.  Suppose that $x_m$ is a variable that does not appear in $A_=^{(i-1,c)}\otimes A_{\neq}^{(i-1,c)}$ but does appear in $A_=^{(i,c_*)}\otimes A_{\neq}^{(i,c_*)}$.  We bound the contribution of $f(x_m)$ in $A_=^{(i,c_*)}\otimes A_{\neq}^{(i,c_*)}$ by $d_{max}=O(S^{\frac{1}{2}-\tau})$.  

That is, \begin{equation}\label{eq:algclaim}\frac{p_{(i,c_*)}\sum_{A_=^{(i,c_*)}\otimes A_{\neq}^{(i,c_*)}} \Pi f(x_k)}{p_{(i-1,c)}\sum_{A_=^{(i-1,c)}\otimes A_{\neq}^{(i-1,c)}} \Pi f(x_k)}\leq \frac{d_{max}r\sum_{A_=^{(i,c_*)}\otimes A_{\neq}^{(i,c_*)}} \Pi_{k\neq m} f(x_k)}{\sum_{A_=^{(i-1,c)}\otimes A_{\neq}^{(i-1,c)}} \Pi f(x_k)}. 
\end{equation}    
Now that $f(x_m)$ has been `taken out', we can construct a crude lower bound on the contribution of $f(x_m)$ in $A_=^{(i-1,c)}\otimes A_{\neq}^{(i-1,c)}$, given the hypothesis that $\sum_{x_m} f(x_m) = O(S)$.  So we have from the inequality (\ref{eq:algclaim}) that     

\begin{equation}\label{eq:algclaim22}\frac{p_{(i,c_*)}\sum_{A_=^{(i,c_*)}\otimes A_{\neq}^{(i,c_*)}} \Pi f(x_k)}{p_{(i-1,c)}\sum_{A_=^{(i-1,c)}\otimes A_{\neq}^{(i-1,c)}}\Pi f(x_k)} \leq \frac{O(S^{1-2\tau})\sum_{A_=^{(i,c_*)}\otimes A_{\neq}^{(i,c_*)}} \Pi_{k\neq m} f(x_k)}{O(S)\sum_{A_=^{(i-1,c)}\otimes A_{\neq}^{(i-1,c)}} \Pi_{k\neq m} f(x_k)}. 
\end{equation}

But now that $x_m$ has been effectively removed, the summations over $A_=^{(i-1,c)}\otimes A_{\neq}^{(i-1,c)}$ and $A_=^{(i,c_*)}\otimes A_{\neq}^{(i,c_*)}$ are identical and we conclude from the right hand side of (\ref{eq:algclaim22}) that 

\begin{equation}
\label{eq:algclaim23} \frac{p_{(i,c_*)}\sum_{A_=^{(i,c_*)}\otimes A_{\neq}^{(i,c_*)}}\Pi_{k\neq m}f(x_k)}{p_{(i-1,c)}\sum_{A_=^{(i-1,c)}\otimes A_{\neq}^{(i-1,c)}} \Pi_{k\neq m} f(x_k)} \leq O(S^{-2\tau})
\end{equation}

Invoking the inductive hypothesis that
%completes the proof as 
\begin{equation}\frac{p_{(i-1,c)}\sum_{A_=^{(i-1,c)}\otimes A_{\neq}^{(i-1,c)}}\Pi_{m=1}^{r} f(x_m)}{\sum_{x_{1},...,x_{2k},x_{2k+1}\neq...\neq x_{r}}\Pi_{m=1}^{r} f(x_m)}=O(S^{-2(i-1)\tau}),
\end{equation} 
and  combining this with inequalities (\ref{eq:algclaim22}),(\ref{eq:algclaim23}) completes the proof.
\end{itemize}
\end{proof}

We now prove the desired result that validates the assumptions in Theorems \ref{thm:arborder1} and \ref{thm:arborder2} that enable us to extend our enumeration estimate to arbitrary order.  But before doing so, we would like to remind the reader that throughout this work we have been expressing our approximation $f$ of $\phi$ such that $f = \phi*(1+O(S^{-2k\tau}))$.  This statement implies that  $ f - \phi = \phi O(S^{-2k\tau})$  But since we know by Corollary \ref{cor:ratio} that $\phi(x_i,x_j,\mathbf{d},\mathbf{a}) = \frac{a_i}{a_j}(1+O(S^{-2\tau}))$, it follows that $f - \phi = \frac{a_i}{a_j}O(S^{-2k\tau})$.  In the theorem that follows we often interchange $f = \phi(1+O(S^{-2k\tau})$. $ f - \phi = \phi O(S^{-2k\tau})$ and $f - \phi = \frac{a_i}{a_j}O(S^{-2k\tau})$.

\begin{thm}
\label{thm:assumption}
 Given an approximation $f$ of $\phi$ (the true value for the ratio of the number of graphs that realize two slightly different degree sequences) such that $|f(x_i,x_j,\mathbf{d})-\phi(x_i,x_j,\mathbf{d})|=\frac{a_i}{a_j}O(S^{-2k\tau})$ for some $k<\frac{1}{4\tau}+1$,

\begin{equation}
\label{eq:decomp}
f(x_i,x_j,\mathbf{d},\mathbf{a})=\frac{h(x_i,\mathbf{d},\mathbf{a})}{h(x_j,\mathbf{d},\mathbf{a})}
\end{equation} 
where 
\begin{equation}
\label{eq:h}
h(x_i,\mathbf{d},\mathbf{a})=a_i(1+\sum_{v=1}^{r}\frac{\gamma_va_i^{k_v}\Pi_{q=1}^{s}\beta_q^{m(v,q)}\alpha_q^{n{(v,q)}}}{\alpha_1^{z}})
\end{equation}
for some constants $\gamma_v$,where $r,s,z$ are finite and each term in the summations in the numerator and denominator in (\ref{eq:decomp}) is $O(S^{-2\tau})$.

Furthermore if $\|\mathbf{d}_1-\mathbf{d}_2\|_{\infty}\leq 1$ and $\|\mathbf{d}_1-\mathbf{d}_2\|_{1}=O(S^{\frac{1}{2}-\tau})$, assuming the degrees of node $x_i$ are identical in $\mathbf{d}_1$ and $\mathbf{d}_2$, then $$|h(x_i,\mathbf{d}_1)-h(x_i,\mathbf{d}_2)|=a_i(O(S^{-\frac{1}{2}-3\tau})).$$

\end{thm}
\begin{proof}
We prove the result by induction.

The base case holds trivially when $k=1$ as by Corollary \ref{cor:expand},  $f(x_i,x_j,\mathbf{d})=a_i/a_j$ and it also follows trivially that differences in the degree sequence aside from the degrees of nodes $x_i$ or $x_j$ do not affect $f(x_i,x_j,\mathbf{d})$.  Similarly, equation (\ref{eq:decomp}) holds with $r=1$ and $\gamma_1=0$ in the expression (\ref{eq:h}) for both the numerator and denominator of (\ref{eq:decomp}).

So we proceed to the inductive step.

 Given our inductive hypothesis where we have an approximation $f$ where the dependence on the degree sequence in $f$ is sufficiently weak $O(S^{-\frac{1}{2}-\tau})$ and can be dropped as $|f(x_i,x_j,\mathbf{d})-\phi(x_i,x_j,\mathbf{d})|=\frac{a_i}{a_j}(O(S^{-2(k-1)\tau}))$ since $2(k-1)\tau<\frac{1}{2}+2\tau$, we want to show that the sharper approximation produced by Theorem \ref{thm:arborder1} also has the same property.

By applying Theorem \ref{thm:arborder1}, we get a stronger approximation (which we also denote by $f$) which is of the form $\frac{a_i}{a_j}\exp(\log(1+\frac{\|G_{X_{1i}}\|}{\|G_{X_{0i}}\|}))\exp(-\log(1+\frac{\|G_{X_{1j}}\|}{\|G_{X_{0j}}\|}))$.  From equations (\ref{eq:arbderive2})-(\ref{eq:arbderive3}) in Theorem \ref{thm:arborder1}, 
we have 
$$\frac{\|G_{X_{1i}}\|}{\|G_{X_{0i}}\|}= \frac{(a_i-1)\sum_{x_1\neq...\neq x_{a_i-1}=x_{a_i}}\Pi f(x_i,m_i,\mathbf{d},\mathbf{b})}{\sum_{x_1\neq...\neq x_{a_i}}\Pi f(x_i,m_i,\mathbf{d},\mathbf{b})}$$
where the choice of $m_i$ is arbitrary.  By the inductive hypothesis (\ref{eq:decomp}), we can drop the dependence on $m_i$ by multiplying the numerator and denominator by $\Pi_{i=1}^{a_i}b_{m_i}(1+\sum_{v=1}^{r}\frac{\gamma_vb_{m_i}^{k_v}\Pi_{q=1}^{s}\beta_q^{m{(v,q)}}\alpha_q^{n{(v,q)}}}{\alpha_1^{z}})$; that is, we have 
$$\frac{\|G_{X_{1i}}\|}{\|G_{X_{0i}}\|}=\frac{(a_i-1)\sum_{x_1\neq...\neq x_{a_i-1}=x_{a_i}}\Pi f(x_i,\mathbf{d},\mathbf{b})}{\sum_{x_1\neq...\neq x_{a_i}}\Pi f(x_i,\mathbf{d},\mathbf{b})}$$
where 
\begin{equation}
\label{eq:formularatio}
f(x_i,\mathbf{d},\mathbf{b})=b_i(1+\sum_{v=1}^{r}\frac{\gamma_vb_i^{k_v}\Pi_{q=1}^{s}\beta_q^{m(v,q)}\alpha_q^{n{(v,q)}}}{\alpha_1^{z}}).
\end{equation}
For simplicity, we will now drop the explicit dependence on $\mathbf{d}$ and $\mathbf{b}$ from $f$.
Note that by applying Theorems \ref{thm:alg} and \ref{thm:alg2} to the denominator and numerator respectively we have that 

$$\frac{\|G_{X_{1i}}\|}{\|G_{X_{0i}}\|}=\frac{(a_i-1)\sum_{j=0}^{k}\sum_{c=0}^{h(j,k)}p_{(j,c,1)}(a_i)\sum_{A_=^{(j,c,1)}\otimes A_{\neq}^{(j,c,1)}}\Pi f(x_i)}{\sum_{j=0}^{k}\sum_{c=0}^{h(j,k)}p_{(j,c,0)}(a_i)\sum_{A_=^{(j,c,0)}\otimes A_{\neq}^{(j,c,0)}}\Pi f(x_i)}$$
where the $p$'s are polynomials in $a_i$, summations denoted with $A^{(j,c,x)}_=$ are over a finite  number $2k$ of variables.
But by Theorems \ref{thm:alg} and \ref{thm:alg2}, we can drop the terms in the numerator and denominator involving $p_{(k,c,x)}\sum_{A_=^{(k,c,x)}\otimes A_{\neq}^{(k,c,x)}}\Pi f(x_i)$ as they only contribute a maximum of $O(S^{-2k\tau})$, hence

$$\frac{(a_i-1)\sum_{j=0}^{k-1}\sum_{c=0}^{h(j,k)}p_{(j,c,1)}(a_i)\sum_{A_=^{(j,c,1)}\otimes A_{\neq}^{(j,c,1)}}\Pi f(x_i)}{\sum_{j=0}^{k-1}\sum_{c=0}^{h(j,k)}p_{(j,c,0)}(a_i)\sum_{A_=^{(j,c,0)}\otimes A_{\neq}^{(j,c,0)}}\Pi f(x_i)}(1+O(S^{-2k\tau})).$$

(As an abuse of notation, when we write $\sum_{A_=}$, we are summing over all $\mathbf{x}\in\mathbb{N}^{2k}$, but when we consider $A_=$ without the sigma, we are specifying 
which dummy variables must equal one another.)
In addition, by Theorems \ref{thm:alg} and \ref{thm:alg2}, for all $j\leq k-1$, $\cup_{A\in A_=^{(j,c,x)}} A\cap A_{\neq}^{(j,c,x)}=\emptyset$ and $A_{\neq}^{(j,c,x)}=\{2k+1,...,a_i\}$.
Hence we can factor out a $\sum_{x_{2k+1}\neq...\neq x_{a_i}}\Pi_{i=2k+1}^{a_i} f(x_i)$ from both the numerator and denominator.  This yields

$$\frac{\|G_{X_{1i}}\|}{\|G_{X_{0i}}\|}=\frac{(a_i-1)\sum_{j=0}^{k-1}\sum_{c=0}^{h(j,k)}p_{(j,c,1)}(a_i)\sum_{A_=^{(j,c,1)}}\Pi f(x_i)}{\sum_{j=0}^{k-1}\sum_{c=0}^{h(j,k)}p_{(j,c,0)}(a_i)\sum_{A_=^{(j,c,0)}}\Pi f(x_i)}+O(S^{-2k\tau})).$$

Recall that $f(x_i,\mathbf{d},\mathbf{b})=b_i(1+\sum_{v=1}^{r}\frac{\gamma_vb_i^{k_v}\Pi_{q=1}^{s}\beta_q^{m(v,q)}\alpha_q^{n{(v,q)}}}{\alpha_1^{z}})$ where each of the finitely many terms in the summation is $O(S^{-2\tau})$.   Furthermore, note that only $A_=^{(0,0,0)}=\emptyset$, so let us multiply the numerator and denominator by $\frac{1}{\alpha^{2k}}$:  

\begin{equation}\label{eq:magnitude}\frac{\|G_{X_{1i}}\|}{\|G_{X_{0i}}\|}=\frac{(a_i-1)\sum_{j=0}^{k-1}\sum_{c=0}^{h(j,k)}p_{(j,c,1)}(a_i)\sum_{A_=^{(j,c,1)}} \frac{\Pi f(x_i)}{\alpha_1^{2k}}}{\sum_{j=0}^{k-1}\sum_{c=0}^{h(j,k)}p_{(j,c,0)}(a_i)\sum_{A_=^{(j,c,0)}} \frac{\Pi f(x_i)}{\alpha_1^{2k}}}+O(S^{-2k\tau}).
\end{equation}

Note that every term in the numerator $(a_i-1)p_{(j,c,1)}(a_i)\sum_{A_=^{(j,c,1)}}\frac{\Pi f(x_i)}{\alpha_1^{2k}}=O(S^{-2\tau})$ and except for $j=0$ and $c=0$,
$p_{(j,c,0)}(a_i)\sum_{A_=^{j,c,0)}}\frac{\Pi f(x_i)}{\alpha_1^{2k}}=O(S^{-2\tau})$.  And finally, since $A_=^{(0,0,0)}=\emptyset$ and $f(x_i,\mathbf{d},\mathbf{b})=b_i(1+\sum_{v=1}^{r}\frac{\gamma_vb_i^{k_v}\Pi_{q=1}^{s}\beta_q^{m(v,q)}\alpha_q^{n{(v,q)}}}{\alpha_1^{z}})$ where each term in the summation is $O(S^{-2\tau})$and $\sum f(x_i) = S + O(S^{1-2\tau})$, we note that 
$\sum_{A_=^{(0,0,0)}}\frac{\Pi f(x_i)}{\alpha_1^{2k}}=1+O(S^{-2\tau})$ as $\alpha_1 = S$.  We now show that the dependence in equation (\ref{eq:magnitude}) on the degree sequence is `small'.  

As noted before, $(a_i-1)^{\delta_{x,1}}p_{(j,c,x)}(a_i)\sum_{A_=^{(j,c,x)}} \frac{\Pi f(x_i)}{\alpha_1^{2k}}=\delta_{x,0}+O(S^{-2\tau})$ where $\delta_{x,y}=1$ if $x=y$ and $0$ otherwise and the $O(S^{-2\tau})$ term is some finite sum of terms of the form $\gamma a_i^{j}\frac{\Pi_{k=1}^{m}\alpha_k^{g(k)}\beta_k^{h(k)}}{\alpha_1^{2k}}$ each of which are $O(S^{-2\tau})$. 

The constraint that $\gamma a_i^{j}\frac{\Pi_{k=1}^{m}\alpha_k^{g(k)}\beta_k^{h(k)}}{\alpha_1^{2k}}=O(S^{-2\tau})$ holds for all degree sequences such that $d_{max}=O(S^{\frac{1}{2}-\tau})$.
Now we consider the perturbation analysis, such that $\| \mathbf{d}_1 -\mathbf{d}_0\|_1=O(S^{\frac{1}{2}-\tau})$ and $\|\mathbf{d}_1-\mathbf{d}_0\|_{\infty}\leq 1$, then $|\beta_j(\mathbf{d}_1)-\beta_j(\mathbf{d}_0)|=O(d_{max}^{j})=O(S^{\frac{j}{2}-j\tau})$.  But also note that $\max(\alpha_j,\beta_j)=O(S^{1+\frac{j-1}{2}-(j-1)\tau})$ as $\beta_j=\sum_i b_i^{j}\leq d_{max}^{j-1}\sum_i b_i=O(S^{1+\frac{j-1}{2}-(j-1)\tau})$.  Now if we measure the impact of considering different degree sequences $\mathbf{d}_0,\mathbf{d}_1$ on $\gamma a_i^{j}\frac{\Pi_{k=1}^{m}\alpha_k^{g(k)}\beta_k^{h(k)}}{\alpha_1^{2k}}$,this results in a $O(S^{-\frac{1}{2}-\tau})$ times smaller than $\gamma a_i^{j}\frac{\Pi_{k=1}^{m}\alpha_k^{g(k)}\beta_k^{h(k)}}{\alpha_1^{2k}}$.  But since $\gamma a_i^{j}\frac{\Pi_{k=1}^{m}\alpha_k^{g(k)}\beta_k^{h(k)}}{\alpha_1^{2k}}=O(S^{-2\tau})$, the contribution in equation (\ref{eq:magnitude}) from considering different degree sequences is $O(S^{-\frac{1}{2}-\tau-2\tau})$.

Now to verify that our new higher order approximation $f(x_i,x_j,\mathbf{d},\mathbf{a})=\frac{a_i}{a_j} \left( \frac{1+\sum_{v=1}^{r}\frac{\gamma_va_i^{k_v}\Pi_{q=1}^{s}\beta_q^{m(v,q)}\alpha_q^{n{(v,q)}}}{\alpha_1^{z}}}{1+\sum_{v=1}^{r}\frac{\gamma_va_j^{k_v}\Pi_{q=1}^{s}\beta_q^{m{(v,q)}}\alpha_q^{n{(v,q)}}}{\alpha_1^{z}}} \right)$, we merely perform a Taylor expansion in the denominator of
$$\frac{(a_i-1)\sum_{j=0}^{k-1}\sum_{c=0}^{h(j,k)}p_{(j,c,1)}(a_i)\sum_{A_=^{(j,c,1)}} \frac{\Pi f(x_i)}{\alpha_1^{2k}}}{\sum_{j=0}^{k-1}\sum_{c=0}^{h(j,k)}p_{(j,c,0)}(a_i)\sum_{A_=^{(j,c,0)}} \frac{\Pi f(x_i)}{\alpha_1^{2k}}}+O(S^{-2k\tau})$$ as the denominator can be rewritten in the form $1+O(S^{-2\tau})$ where the $O(S^{-2\tau})$ term is some finite sum of terms of the form $\gamma a_i^{j}\frac{\Pi_{k=1}^{m}\alpha_k^{g(k)}\beta_k^{h(k)}}{\alpha_1^{2k}}$ each of which are $O(S^{-2\tau})$ and each term in the numerator is $O(S^{-2\tau})$.  Hence we get that $$1+\frac{(a_i-1)\sum_{j=0}^{k-1}\sum_{c=0}^{h(j,k)}p_{(j,c,1)}(a_i)\sum_{A_=^{(j,c,1)}} \frac{\Pi f(x_i)}{\alpha_1^{2k}}}{\sum_{j=0}^{k-1}\sum_{c=0}^{h(j,k)}p_{(j,c,0)}(a_i)\sum_{A_=^{(j,c,0)}} \frac{\Pi f(x_i)}{\alpha_1^{2k}}}+O(S^{-2k\tau})=(1+\sum_{v=1}^{r_*}\frac{\gamma_va_i^{k_v}\Pi_{q=1}^{s}\beta_q^{m(v,q)}}{\alpha_1^{2k}})$$ for some finite $r_*$.

Repeating the same argument for evaluating the $\exp(-\log(1+\frac{\|G_{X1j}\|}{\|G_{X0j}\|}))$ term yields the desired result, that $f(x_i,x_j,\mathbf{d},\mathbf{a})=\frac{a_i}{a_j}\left(\frac{1+\sum_{v=1}^{r}\frac{\gamma_va_i^{k_v}\Pi_{q=1}^{s}\beta_q^{m(v,q)}\alpha_q^{n{(v,q)}}}{\alpha_1^{z}}}{1+\sum_{v=1}^{r}\frac{\gamma_va_j^{k_v}\Pi_{q=1}^{s}\beta_q^{m{(v,q)}}\alpha_q^{n{(v,q)}}}{\alpha_1^{z}}}\right)$.

\end{proof}

\section{Discussion}

In this work, we have established new results that allow for enumeration, up to arbitrary accuracy, of the graphs that satsify a fixed bidegree sequence.  These results are proven to hold asymptotically, as the number of nodes in the graphs goes to infinity, under the assumption that as the limit is taken, the degrees of all nodes remain below  the square root of the number of edges, $S$, in the graph.
The extension to the square root represents an advance over previous results of a similar spirit, which only allow for degrees that are smaller powers of the number of edges present \cite{McKay84,McKay91,Greenhill06}.   Clearly, the increase from $S^{1/3}$ to $S^{1/2}$ can become quite substantial in large graphs and will enhance the relevance of the results correspondingly.  Interestingly, we have recently established a convenient condition such that there is certain to be a graph that realizes each bidegree sequence that satisfies it (i.e., a graphicality condition), which is also based on a square root bound on the maximum degree \cite{B15}.  Indeed, as a sequence approaches the edge of graphicality, there will typically be fewer graphs that realize it \cite{Berger14}, and hence the types of arguments in the present paper, which are based on the presence of large collections of graphs, cannot be applied.    While other approaches do not require such sparsity conditions, they circumvent related obstacles by assuming sufficient homogeneity among nodes' degrees \cite{Barvinok,Canfield08}, whereas our results require no such restriction.

Although our findings are rigorously valid in an asymptotic sense, they are likely to be quite relevant in the generation of the large graphs that arise in many application areas \cite{Pomerance09,Zhao11}.
In practice, once a graphic bidegree sequence is specified, various algorithms can be used to instantiate one or more graphs that realize this sequence.  If the properties of a graph are to be studied, it is desirable not just to build such a graph or a collection of such graphs but to do so in an unbiased way.  Computational graph-generation algorithms  often start from a potentially biased sample of graphs and then repeatedly invoke some procedure, such as switching which nodes are connected by certain edges, to try to erase any initial bias.
Enumeration results such as ours can be used to provide estimates for the number of iterations or the size of samples that will be needed for these approaches to succeed.  In fact, many of our results concern enumeration of the ratio of the numbers of graphs compatible with two different bidegree sequences.  These types of results may be particularly useful for computational applications, since they allow for the calculation of the fraction of graphs that will contain a particular edge set or other features, which can then be compared to the corresponding fraction obtained computationally.  Our results generally pertain to sequences that are close to each other in some sense but, as we have done in some of our proofs, ratios for more disparate sequences can be obtained by taking  a series of small steps to gradually transform one of the sequences to the other.

Our findings can be extended to graphs with specific properties, such as graphs that exclude certain collections or types of edges, as briefly discussed in Appendix B.    On the other hand, our approach makes significant use of the fact that in large, sparse graphs, nodes rarely share common neighbors or collections of neighbors.  If we try to extend our ideas to smaller networks, then we may run into trouble as the numbers of potential targets for edges is reduced and shared neighbors become more important.   For sufficiently small degree sequences, the numbers of corresponding graphs could be counted by direct brute force methods \cite{Miller13}.  As sequences become large enough, however, such approaches would become more expensive.  A future computational direction would be to explore how large a degree sequence must be in order for asymptotic results such as ours to provide accurate enumeration estimates.   If there is a gap between sequence sizes where direct computational counting of graphs is practical and sizes where the application of  asymptotic results is useful, then alternative methods would need to be used for graph enumeration for sequences that fall into this gap.  

\section{Acknowledgements}
This work was partially supported by National Science Foundation Award DMS-1312508.  Any opinion, findings, and conclusions or recommendations expressed in this material are those of the authors(s) and do not necessarily reflect the views of the NSF.

\section{Appendix A}
\label{section:appendix1}
In this section, we  briefly mention a generalization of the  sparsity assumptions that will also yield a power series expansion for  the number of graphs realizing a degree sequence,  using the techniques of this paper.  

Specifically, up to this point, we have assumed that $d_{max}=O(S^{\frac{1}{2}-\tau})$.  Denote $a^{(k)}$ ($b^{(k)}$) as the $k$th entry in an in-degree (out-degree) sequence derived by sorting \textbf{a} and \textbf{b} into non-increasing order.  The key observation is that if  %{\color{red} notation unclear:} 
\begin{equation}\tag{A.1}\label{eq:condition} \max(\sum_{k=1}^{a^{(1)}}b^{(k)},\sum_{k=1}^{b^{(1)}}a^{(k)})=O(S^{1-\tau}),\end{equation} 
then the appropriate extensions of Corollary 4 and Lemma 2 still hold and we can derive the corresponding power expansions.  
Naturally, for example, condition (\ref{eq:condition}) holds if $a^{(1)}b^{(1)}=O(S^{1-\tau})$ or if more simply $d_{max}=O(S^{\frac{1}{2}-\tau})$.

Recall the proof of Theorem 2, where we count the number of common neighbors that receive an outward (inward) edge from both an arbitrary node $x$ and a node of bounded degree $y$.  In the worst case scenario, (\ref{eq:condition}) gives the number of outgoing edges from all of the neighbors of $x$.  Since there are $S$ edges in the graph, it is intuitive that it would be difficult for $x$ and $y$, which has bounded degree, to share a common neighbor.
This idea  forms the foundation for the appropriate extensions of Theorem 2 and Corollaries 3 and 4.

Similarly, with care, we can extend Lemma 2 as follows:

\begin{lem}
 
Suppose that $ f,g : I:=\{1,2,...,N\} \to [1,\infty)$ and for simplicity let $g(\cdot)\leq f(\cdot)$. 
 Let $\{x_{1},...,x_{r}\}$ be distinct inputs  from $I$ that yield the largest $r$ outputs for $f(\cdot)$ and assume that $\sum_{i=1}^{r} f(x_{i})=O(S^{1-\tau})$, $\sum_{i=1}^{N}f(i)=O(S)$. Furthermore, let $k$ be an $O(1)$ natural number and  let $c_1, \ldots, c_r$ be a sequence of natural numbers. 
Then $$\sum_{c_{1}\neq...\neq c_{r}}^{N}g(c_{1})\Pi_{i=2}^{r} f(c_{i})=\sum_{\substack{c_{1},..,c_{k},\\c_{k+1}\neq...\neq c_{r}}}^{N}g(c_{1})\Pi_{i=2}^{r} f(c_{i})(1+ O(S^{-\tau})).$$
\end{lem}

We omit the details of the proofs of these results.

\section{Appendix B}
\label{section:appendix2}
In this section, we discuss how to generalize results from previous sections to directed graphs without loops and to undirected graphs (with and without loops); similar ideas apply to graphs with other sets of prohibited edges besides loops.

\subsection{Directed Graphs without Loops}
As one may expect, when considering directed graphs without loops, the results regarding asymptotic enumeration for directed graphs with loops carry over as the likelihood that a fixed node has an edge to itself is small.  More specifically,

\begin{lem} 
\label{lem:appB}
Consider a degree sequence $\mathbf{d}=(\mathbf{a},\mathbf{b})\in \mathbb{Z}^{N\times 2}$ where
\[
\mathbf{a}=\{a_1,a_2,0,\ldots,0\}, \mathbf{b}=\{\delta,\delta,2,...,2,1,..,1,0,..,0\} \; \; \mbox{and} \; \; \sum_{i=1}^{N}a_{i}=\sum_{i=1}^{k}b_{i}=:S, 
\]
with $\delta$  either $0$ or $1$ and with 2 appearing $q$ times in $\mathbf{b}$.
If $a_1, a_2 \geq q+\delta$, then there are $\binom{a_1+a_2-2\delta-2q}{a_1-\delta-q}$ directed graphs without loops that realize the above degree sequence.
\end{lem}

The above Lemma allows for the appropriate generalization of Theorem \ref{thm:partition}, where we construct sets of residual degree sequences $X_{k,\delta_{1},\delta_{2}}$ for $k$ the number of common neighbors being considered.  \textcolor{black}{Here, $\delta_1=1$ if node 1 connects to node 2 and is otherwise 0; similarly, $\delta_2=1$ if node 2 connects to node 1 and is otherwise 0.}
   Since it should be unlikely asymptotically that either $\delta_{1}$ or $\delta_{2}$ is positive, it follows that $\|G_{X_{0,0,0}}\|$ becomes the dominating term and the analysis proceeds as in other cases.  (We can make this rigorous by using switching arguments as we did in Section 3.)  This will provide us with arbitrarily accurate asymptotics for the ratio of the number of directed graphs (without loops) with degree sequences that are distance 2 apart.  
   
Suppose we knew exactly the number of graphs of one degree sequence that summed to $S$.  Then we could multiply this quantity by a product of (fractional) terms, where each fractional term is a ratio of the number of directed graphs without loops, distance 2 apart.  Since from above we have asymptotics for the ratio of the number of graphs distance two apart, we only need one case where we can count the number of graphs of a given degree sequence exactly to derive the general formula. For simplicity, we can consider a degree sequence where all realizations can never have loops.  For instance, we can arbitrarily extend our degree sequence to an analogous degree sequence in $\mathbb{Z}^{2N\times 2}$ and we can consider the degree sequence $(\{a_{1},....,a_{N},0,....,0\},\{0,...,0,b_{N+1},...,b_{2N})\}\in\mathbb{Z}^{2N\times 2}$. 
If $b_{N+1},...,b_{2N}$ consists of only zeros and ones, then   Corollary \ref{cor:count} applies and all realizations are directed graphs without loops.

\subsection{Undirected Graphs}
Practically speaking, the primary difference between manipulating degree sequences for undirected graphs and degree sequences for directed graphs is that for directed graphs,  the in-degree sequence can be manipulated without impacting the out-degree sequence.  In contrast, with undirected graphs, the corresponding adjacency matrix is symmetric.  As such, while the same thematic ideas also follow through to undirected graphs, care should be taken.  More specifically, instead of partitioning two rows (or two columns) separately, as we were able to do in the directed case, we must partition the two rows and the two columns together.  \textcolor{black}{In the remainder of this section, we will present results in terms of undirected graphs without loops, noting that the techniques carry over to the case where loops are allowed.}

For the undirected case, we must consider two distinct cases.  
In the first case, the two partitioned nodes do not share an edge together.  Consequently, if the two nodes have $a_{1}+a_{2}$ edges, \textcolor{black}{then these edges must also show up in the degrees of $a_{1}+a_{2}$ other nodes in the graph.}

\begin{lem}
\label{lem:appB2}
Consider a degree sequence $\mathbf{a} \in \mathbb{Z}^{N}$ for an undirected graph without loops, given by  
\[
\mathbf{a}=\{a_1,a_2,2,...,2,1,...,1,0,...,0\} 
\]
where  the first two entires of $\mathbf{a}$ are followed by $q_{2}$ $2's$ and $q_{1}$ $1's$ such that $q_{1}+2q_{2}=a_{1}+a_{2}$.
Then there are $\binom{a_1+a_2-2q_2}{a_1-q_2}$ undirected graphs without loops that realize the above degree sequence where nodes 1 and 2 are not neighbors.
\end{lem}

Alternatively, if nodes $1$ and $2$ are neighbors then they make $a_{1}+a_{2}-2$ edges with other nodes in the graph.  We need to consider this case as well and the result is analogous to Lemma \ref{lem:appB2}.  (If we were considering undirected graphs with loops, we would also need to consider the case where there are self-loops too.)  Then as before, we construct sets $X_{k,\delta}$, where $k$ denotes the number of common neighbors \textcolor{black}{of nodes 1 and 2} and $\delta$ denotes whether nodes $1$ and $2$ are linked to each other.  Again, $\|G_{X_{0,0}}\|$ dominates and the analogous results hold. 
To switch from ratios to counts for a particular degree sequence, the following result, which is analogous to Corollary 6, is fundamental.  
\begin{lem}
For a degree sequence $\mathbf{a} \in \mathbb{Z}^{N}$ for an undirected graph without loops  given by $\mathbf{a}=\{1,...,1,0,...,0\}$ with $\sum a_i = S$, 
 $$\|G_{\mathbf{a}}\|=\frac{S!}{2^{\frac{S}{2}}(\frac{S}{2}!)}.$$
\end{lem}
\begin{proof}
The proof is constructive. Note that $S$ must be even as otherwise there are no graphs. For the first edge, there are $\binom{S}{2}$ possible choices of pairs of nodes to connect.  After we decide the initial pair to wire up, then there are $\binom{S-2}{2}$ possible choices to form the second edge.  This  reasoning implies that there are  $\Pi_{k=0}^{\frac{S}{2}-1}\binom{S-2k}{2}=\frac{S!}{2^{S/2}}$ choices for wirings.  In this procedure, however, suppose that all of the choices are the same except for the first two steps.  In the first step in one example, node 1 wires with node 2 and node 3 wires with node 4.  Alternatively in the other example, node 3 wires up with node 4 in the first step, but node 1 wires up with node 2 in the second step.  The output is the same graph, but we counted both of these instances as two distinct events (graphs).  We readily note that there are $\frac{S}{2}!$ possible ways of wiring up the same graph with this procedure, as there are $\frac{S}{2}$ edge pairs in the graph. 
\end{proof}
The technique for counting the number of undirected graphs with loops for the degree sequence $\mathbf{a}=\{1,...,1,0,...,0\}$ is similar and left as an exercise to the reader.   At this juncture, we note that the idea of prohibiting loops is a special case of prohibiting edges between (two) nodes.   Though it is outside the scope of this work, we strongly expect that the ideas can  be  extended to attain asymptotics for the ratio of the number of graphs of two distinct degree sequences distance 2 apart where we prohibit edges between certain nodes.  As long as generating a prohibited edge is unlikely for any node in the graph, the ideas of our work should carry over;  that is, the dominating term in the ratio will consist of realizations of graphs where the two partitioned nodes do not share a common neighbor and no prohibited edges appear.  

\section{Appendix C: Proof of Theorem \ref{thm:arborder2}}
\label{section:appendix3}
In this section, we provide a proof of Theorem \ref{thm:arborder2}.  From Theorem \ref{thm:arborder1}, we have an approximation $f$ such that $\frac{\|G_{\mathbf{d_{-i}}}\|}{\|G_{\mathbf{d_{-j}}}\|}=f(\mathbf{e}_{i},\mathbf{e}_{j},\mathbf{d},\mathbf{\sigma})(1+O(S^{-\frac{1}{2}-w\tau}))$ where $\mathbf{\sigma} = \mathbf{a}$ or $\mathbf{b}$, where we have assumed (and justified) up until this point that the dependence on the degree sequence $\mathbf{d}$ in $f$ is weak.  Eventually, we will reach a point where such an assumption is no longer reasonable.  In this case we decompose our approximation $f$ into one part that effectively ignores the dependence of $\mathbf{d}$ and a part of $f$ that takes into account $\mathbf{d}$.  That is, we can write \begin{equation}\tag{C.1} \label{eq:higherorder1}  f (\mathbf{e}_{i},\mathbf{e}_{j},\mathbf{d}_1,\mathbf{\sigma})=f (\mathbf{e}_{i},\mathbf{e}_{j},\mathbf{d}_0,\mathbf{\sigma})(1+O(S^{-\frac{1}{2}-\tau})),
\end{equation} where we have proven that for suitably chosen $\mathbf{d}_0$ the correction will yield a term of size at most  $O(S^{-\frac{1}{2}-\tau})$.  For simplicity, we will consider $\mathbf{\sigma} = \mathbf{a}$.  Denote the $O(S^{-\frac{1}{2}-\tau})$ term in equation (\ref{eq:higherorder1}) by $z_*$, such that 
 equation (\ref{eq:higherorder1}) becomes 

\begin{equation}\tag{C.2}\label{eq:higherorder2}f (\mathbf{e}_{i},\mathbf{e}_{j},\mathbf{d}_1,\mathbf{a})=f (\mathbf{e}_{i},\mathbf{e}_{j},\mathbf{d}_0,\mathbf{a})+ z_{*}f (\mathbf{e}_{i},\mathbf{e}_{j},\mathbf{d}_0,\mathbf{a}).
\end{equation}

In the Theorem below, we restate Theorem \ref{thm:arborder2}; note that we implicitly define $z = z_{*}f (\mathbf{e}_{i},\mathbf{e}_{j},\mathbf{d}_0,\mathbf{a})$.

\begin{duplicate}[Theorem~\ref{thm:arborder2}]
Consider an approximation $$\frac{\|G_{\mathbf{d_{-i}}}\|}{\|G_{\mathbf{d_{-j}}}\|}=f(\mathbf{e}_{i},\mathbf{e}_{j},\mathbf{d},\mathbf{\sigma})(1+O(S^{-\frac{1}{2}-w\tau}))$$ for some $w> 0$. 
Furthermore suppose that for $m=O(S^{\frac{1}{2}-\tau})$, $$ f(\mathbf{e}_{i},\mathbf{e}_{j},\mathbf{d}_0,\mathbf{\sigma})= f(\mathbf{e}_{i},\mathbf{e}_{j},\mathbf{d}_1,\mathbf{\sigma})+z(\mathbf{e}_{i},\mathbf{e}_{j},\mathbf{d}_0-\mathbf{d}_1,\mathbf{d}_0,\mathbf{\sigma})$$
where  $\|\mathbf{d}_1-\mathbf{d}_0\|_1\leq m$ and $z(\mathbf{e}_{i},\mathbf{e}_{j},\mathbf{d}_0-\mathbf{d}_1,\mathbf{d}_0,\mathbf{\sigma})\leq  O(S^{-\frac{1}{2}-\tau})f(\mathbf{e}_{i},\mathbf{e}_{j},\mathbf{d}_0,\mathbf{\sigma})$.
\textcolor{black}{If $\mathbf{\sigma}=\mathbf{a}$,} then we can construct a sharper approximation $$\frac{\|G_{\mathbf{d_{-i}}}\|}{\|G_{\mathbf{d_{-j}}}\|}=g(\mathbf{e}_{i},\mathbf{e_{j}},\mathbf{d},\mathbf{a})(1+O(S^{-\frac{1}{2}-(w+2)\tau}))$$ where
\[
\begin{array}{rcl}
g(\mathbf{e}_{i},\mathbf{e}_{j},\mathbf{d},\mathbf{a})&=&\frac{a_{i}}{a_{j}}\exp(\log(1+\frac{(a_{i}-1)\sum_{x_{1}\neq...\neq x_{a_{i-1}}=x_{a_i}} \Pi_{k=1}^{a_i} f(\mathbf{e}_{x_{k}},\mathbf{e}_{u_{k}},(\mathbf{a}-a_i\mb{e}_i,(\mathbf{b}-\sum_{j=1}^{k-1} \mathbf{e}_{x_{j}}+\sum_{j=1}^{a_i-k}\mathbf{e}_{u_{j}}),\mathbf{b}) }{\sum_{x_{1}\neq...\neq x_{a_{i-1}}\neq x_{a_i}} \Pi_{k=1}^{a_i} f(\mathbf{e}_{x_{k}},\mathbf{e}_{u_{k}},(\mathbf{a}-a_i\mb{e}_i,\mathbf{b}-\sum_{j=1}^{k-1} \mathbf{e}_{x_{j}}+\sum_{j=1}^{a_i-k}\mathbf{e}_{u_{j}}),\mathbf{b})})- \vspace{0.1in} \\
 & & \log(1+\frac{(a_{j}-1)\sum_{x_{1}\neq...\neq x_{a_{j-1}}=x_{a_j}} \Pi_{k=1}^{a_j} f(\mathbf{e}_{x_{k}},\mathbf{e}_{u_{k}},(\mathbf{a}-a_je_j,\mathbf{b}-\sum_{j=1}^{k-1} \mathbf{e}_{x_{j}}+\sum_{j=1}^{a_i-k}\mathbf{e}_{u_{j}}),\mathbf{b}) }{\sum_{x_{1}\neq...\neq x_{a_{j-1}}\neq x_{a_j}} \Pi_{k=1}^{a_j} f(\mathbf{e}_{x_{k}},\mathbf{e}_{u_{k}},(\mathbf{a}-a_je_j,\mathbf{b}-\sum_{j=1}^{k-1} \mathbf{e}_{x_{j}}+\sum_{j=1}^{a_i-k}\mathbf{e}_{u_{j}}),\mathbf{b})}))
\end{array}
\]
for an arbitrary choice of $u_{k}$. A similar result holds, with $g$ depending on $\mathbf{b}$, if $\mathbf{\sigma}=\mathbf{b}$.
\end{duplicate}
\begin{proof} Recall the notation from Theorem \ref{thm:arborder1}.  As before, we define $\Delta_i$ to be the difference between  $\frac{\|G_{X_{1_i}}\|}{a_{i}\|G_{X_{0_i}}\|}$  evaluated using the exact ratio $\phi$ and the same quantity evaluated using the approximation $f$.  In essence, we will show that using our crude approximation $f$ will yield an approximation of $\frac{\|G_{X_{1_i}}\|}{a_{i}\|G_{X_{0_i}}\|}$ by a factor of   $(1+O(S^{-\frac{1}{2}-(w+2)\tau}))$,  after which application of  Corollary \ref{cor5} yields the result.    Again we consider equation (\ref{eq:delta}), that is 

$$ \Delta_i = \frac{(a_{i}-1)\sum_{x_{1}\neq...\neq x_{a_{i-1}}=x_{a_i}} \Pi_{k=1}^{a_i} \phi(\mathbf{e}_{x_{k}},\mathbf{e}_{u_{a_i-k+1}},(\mathbf{a}-a_i\mb{e}_i-\mb{e}_j,\mathbf{b}-\sum_{j=1}^{k-1} \mathbf{e}_{x_{j}}-\sum_{j=1}^{a_i-k}\mathbf{e}_{u_{j}}),\mathbf{b}) }{\sum_{x_{1}\neq...\neq x_{a_{i-1}}\neq x_{a_i}} \Pi_{k=1}^{a_i} \phi(\mathbf{e}_{x_{k}},\mathbf{e}_{u_{a_i-k+1}},(\mathbf{a}-a_i\mb{e}_i-\mb{e}_j,\mathbf{b}-\sum_{j=1}^{k-1} \mathbf{e}_{x_{j}}-\sum_{j=1}^{a_i-k}\mathbf{e}_{u_{j}}),\mathbf{b})}-$$

\begin{equation} \tag{C.3} \frac{(a_{i}-1)\sum_{x_{1}\neq...\neq x_{a_{i-1}}=x_{a_i}} \Pi_{k=1}^{a_i} {f}(\mathbf{e}_{x_{k}},\mathbf{e}_{u_{a_i-k+1}},(\mathbf{a}-a_i\mb{e}_i-\mb{e}_j,(\mathbf{b}-\sum_{j=1}^{k-1} \mathbf{e}_{x_{j}}-\sum_{j=1}^{a_i-k}\mathbf{e}_{u_{j}}),\mathbf{b}) }{\sum_{x_{1}\neq...\neq x_{a_{i-1}}\neq x_{a_i}} \Pi_{k=1}^{a_i} {f}(\mathbf{e}_{x_{k}},\mathbf{e}_{u_{a_i-k+1}},(\mathbf{a}-a_i\mb{e}_i-\mb{e}_j,(\mathbf{b}-\sum_{j=1}^{k-1} \mathbf{e}_{x_{j}}-\sum_{j=1}^{a_i-k}\mathbf{e}_{u_{j}}),\mathbf{b})}.
\end{equation}

As in the proof of Theorem \ref{thm:arborder1}, we abuse notation by using $f_k(x_k)$, $\phi_k(x_k)$ in place of the full expressions for $f$ and $\phi$, even though $f_k$ and $\phi_k(x_k)$ do also depend on $x_1,...,x_{k-1}$. This reduces to a more tractable (but slightly misleading) notation:

$$\Delta_i = \frac{(a_{i}-1)\sum_{x_{1}\neq...\neq x_{a_{i-1}}=x_{a_i}} \Pi_{k=1}^{a_i} \phi_k({x_{k}}) }{\sum_{x_{1}\neq...\neq x_{a_{i-1}}\neq x_{a_i}} \Pi_{k=1}^{a_i} \phi_k({x_{k}})}-\frac{(a_{i}-1)\sum_{x_{1}\neq...\neq x_{a_{i-1}}=x_{a_i}} \Pi_{k=1}^{a_i} {f_k}({x_{k}}) }{\sum_{x_{1}\neq...\neq x_{a_{i-1}}\neq x_{a_i}} \Pi_{k=1}^{a_i} {f_k}({x_{k}})}.$$

\textcolor{black}{Let $D_{0}$ denote the set of sets of $a_i$ indices in $\{ 1, \ldots, N \}$ such that the variables associated with these indices are distinct and let $D_1$ denote the set of sets of $a_i$ indices in $\{ 1, \ldots, N \}$ such that the variables associated with the first $a_{i}-2$ are distinct and those with the final two are equal.}  Writing $\Delta_i$ as a single fraction, we obtain 

\[
\begin{array}{rcl}
\Delta_i &=& \frac{\textstyle (a_{i}-1)[\sum_{D_1} \Pi_{k=1}^{a_i} \phi_k({x_{k}})\sum_{D_0} \Pi_{k=1}^{a_i} {f_k}({x_{k}})-\sum_{D_1} \Pi_{k=1}^{a_i} {f_k}({x_{k}})\sum_{D_0} \Pi_{k=1}^{a_i} \phi_k({x_{k}}) ]}{\textstyle \sum_{D_0} \Pi_{k=1}^{a_i} \phi_k({x_{k}})\sum_{D_0} \Pi_{k=1}^{a_i} {f_k}({x_{k}})} \vspace{0.1in} \\
&=& \frac{\textstyle (a_{i}-1)[\sum_{D_1,D_0} \left( \Pi_{x_k\in D_1} \phi_k({x_{k}})\Pi_{x_k\in D_0} {f_k}({x_{k}})- \Pi_{x_k\in D_1}{f_k}({x_{k}}) \Pi_{x_k \in D_0} \phi_k({x_{k}})\right) ]}{\textstyle \sum_{D_0} \Pi_{k=1}^{a_i} \phi_k({x_{k}})\sum_{D_0} \Pi_{k=1}^{a_i} {f_k}({x_{k}})}.
\end{array}
\]

\textcolor{black}{We now apply Theorem \ref{thm:assumption} repeatedly to write $\phi_k=f_k(1+\xi_k)$ where $\xi_k$ depends only on $x_1,...,x_{k}$ (but we omit the dependence) and, by Theorem \ref{thm:assumption} and the definition of $f$,  $\xi_k=O(S^{-\frac{1}{2}-w\tau})$ for some positive $w$.} Furthermore, let $\delta_k =0$ if $k = a_i$ or $k=a_i -1$ and $\delta_k = 1$ otherwise.

These steps yield  

\begin{equation} \tag{C.4}
\label{eq:deltastarAPPENDIX}
\Delta_i =\frac{(a_{i}-1)[\sum_{D_1,D_0} \left( \Pi_{x_k\in D_1} f_k({x_{k}})(1+\xi_k\delta_k)\Pi_{x_k\in D_0} {f_k}({x_{k}})- \Pi_{x_k\in D_1}{f_k}({x_{k}}) \Pi_{x_k \in D_0} f_k({x_{k}})(1+\xi_k\delta_k)\right)+\epsilon ]}{\sum_{D_0} \Pi_{k=1}^{a_i} \phi_k({x_{k}})\sum_{D_0} \Pi_{k=1}^{a_i} {f_k}({x_{k}})} 
\end{equation}
where $\epsilon$ is the compensatory term for zeroing out certain terms by inserting the $\delta_k$ into equation (\ref{eq:deltastarAPPENDIX}), which we can express as  $\epsilon = \epsilon_1 + \epsilon_2 - \epsilon_3 - \epsilon_4$ for 

\begin{equation}\tag{C.5} \label{eq:epsilonAPPENDIX}
\epsilon_1 =  \sum_{D_1}\xi_{a_i}f_{a_i}(x_{a_i})f_{a_i-1}(x_{a_i-1})\Pi_{k\neq a_i-1,a_i}f_k({x_{k}})(1+\xi_k)\sum_{D_0}\Pi_{x_k\in D_0} {f_k}({x_{k}}), 
\end{equation}

\begin{equation}\tag{C.6}
\epsilon_2 =  \sum_{D_1}\xi_{a_i-1}f_{a_i-1}(x_{a_i-1})\Pi_{k\neq a_i-1}f_k({x_{k}})(1+\xi_k)\sum_{D_0}\Pi_{x_k\in D_0} {f_k}({x_{k}}), 
\end{equation}

\begin{equation}\tag{C.7}
\epsilon_3 =  \sum_{D_1}\Pi_{x_k\in D_1} {f_k}({x_{k}})\sum_{D_0}\xi_{a_i-1}f_{a_i}(x_{a_i})f_{a_i-1}(x_{a_i-1})\Pi_{k\neq a_i-1,a_i}f_k({x_{k}})(1+\xi_k), 
\end{equation}

\begin{equation}\tag{C.8}
\epsilon_4 =  \sum_{D_1}\Pi_{x_k\in D_1} {f_k}({x_{k}})\sum_{D_0}\xi_{a_i}f_{a_i}(x_{a_i})\Pi_{k\neq a_i}f_k({x_{k}})(1+\xi_k).
\end{equation}
The procedure for identifying that $\epsilon$ is indeed a `higher order term' is exactly identical to the proof in Theorem \ref{thm:arborder1} and hence we ignore the contribution from $\epsilon$.

It follows instantly from equation (\ref{eq:deltastarAPPENDIX}) that once we distribute $\Pi f_k(1+\xi_k\delta_k)$ and cancel identical terms, the contribution in the numerator only consists of terms where there is at least one $\xi_k$ in the product.  Hence we can define a vector $\mathbf{\eta}$ where each entry $\eta_k$ either equals $1$ or $0$ but $\mathbf{\eta}\neq \mathbf{0}$ and $\eta_{a_i-1}=\eta_{a_i} = 0$, which yields 

\begin{equation}
\tag{C.9}
\label{eq:deltastarAPPENDIX2}
\begin{split}
\Delta_i  \leq \hspace{470pt} \\ \sum_{\mathbf{\eta}\neq \mathbf{0}}\frac{(a_{i}-1)[\sum_{D_1,D_0} \left( \Pi_{x_k\in D_1} f_k({x_{k}})(1+\xi_k\delta_k-\eta_k)\Pi_{x_k\in D_0} {f_k}({x_{k}})- \Pi_{x_k\in D_1}{f_k}({x_{k}}) \Pi_{x_k \in D_0} f_k({x_{k}})(1+\xi_k\delta_k-\eta_k)\right) ]}{\sum_{D_0} \Pi_{k=1}^{a_i} \phi_k({x_{k}})\sum_{D_0} \Pi_{k=1}^{a_i} {f_k}({x_{k}})}.
\end{split}
\end{equation}

Now, for $k=a_i$ or $k=a_i-1$, $f_k(x_k)(1+\xi_k\delta_k)=f_k(x_k)$ by definition.  For these choices of $k$ in $D_0$ and $D_1$, we employ the relationship that $f_k = \mathfrak{i}_k+z_k$ where $\mathfrak{i}_k$ is independent of $x_1,...,x_{k-1}$ and $z_k=O(S^{-\frac{1}{2}-\tau})f_k$ depends on these values.  \textcolor{black}{The $z_k$ represent the new component here relative to Theorem \ref{thm:arborder1}, and we need to show that these terms are small enough that they do not make an important contribution.} We denote the first $a_i - 2$ variables in $D_0$ and $D_1$ as $F_0$ and $F_1$ respectively.   We now have the following (omitting the dependence on $x$):

\begin{equation}\tag{C.10} \label{eq:deltamain}
\Delta_i  \leq \sum_{\mathbf{\eta}\neq \mathbf{0}}\frac{(a_{i}-1)\sum_{D_1}[\mathfrak{i}_{a_i - 1}+z_{a_i - 1}]^{2}\Pi_{x_k\in F_1} f_k(1+\xi_k-\eta_k)\sum_{D_0}[\mathfrak{i}_{a_i - 1}+z_{a_i - 1}][\mathfrak{i}_{a_i}+z_{a_i}]\Pi_{x_k\in F_0} {f_k}}{\sum_{D_0} \Pi_{k=1}^{a_i} \phi_k\sum_{D_0} \Pi_{k=1}^{a_i} {f_k}} - 
\end{equation}
$$\frac{(a_{i}-1)\sum_{D_1}[\mathfrak{i}_{a_i - 1}+z_{a_i - 1}]^{2} \Pi_{x_k\in F_1}{f_k}\sum_{D_0} [\mathfrak{i}_{a_i - 1}+z_{a_i - 1}][\mathfrak{i}_{a_i}+z_{a_i}]\Pi_{x_k \in F_0} f_k(1+\xi_k-\eta_k)}{\sum_{D_0} \Pi_{k=1}^{a_i} \phi_k\sum_{D_0} \Pi_{k=1}^{a_i} f_k}. 
$$

With some algebra, we can re-express equation (\ref{eq:deltamain}) as 

\begin{equation} \tag{C.11}\label{eq:deltamain2}
\Delta_i  \leq \sum_{\mathbf{\eta}\neq \mathbf{0}}\frac{(a_{i}-1)\sum_{D_1}[\mathfrak{i}_{a_i - 1}]^{2}\Pi_{x_k\in F_1} f_k(1+\xi_k-\eta_k)\sum_{D_0}[\mathfrak{i}_{a_i - 1}][\mathfrak{i}_{a_i}]\Pi_{x_k\in F_0} {f_k}}{\sum_{D_0} \Pi_{k=1}^{a_i} \phi_k\sum_{D_0} \Pi_{k=1}^{a_i} {f_k}} - 
\end{equation}

$$\frac{(a_{i}-1)\sum_{D_1}[\mathfrak{i}_{a_i - 1}]^{2} \Pi_{x_k\in F_1}{f_k}\sum_{D_0} [\mathfrak{i}_{a_i - 1}][\mathfrak{i}_{a_i}]\Pi_{x_k \in F_0} f_k(1+\xi_k-\eta_k)}{\sum_{D_0} \Pi_{k=1}^{a_i} \phi_k\sum_{D_0} \Pi_{k=1}^{a_i} f_k} +\Omega
$$
where $\Omega$ is the compensatory term for omitting the $z_k$ terms (which we define later).   Now since $\mathfrak{i}_k$ is independent of $x_1,...,x_{k-1}$, we can employ the Mean Value Theorem to integrate out the last 2 variables, define $D_*$ as the set of sets of $a_{i}-2$ distinct indices, and express equation (\ref{eq:deltamain2}) as

\begin{equation}\tag{C.12} \label{eq:deltamain3}
\Delta_i  \leq \sum_{\mathbf{\eta}\neq \mathbf{0}}\frac{(a_{i}-1)\sum_{D_*,D_*}\lambda_1 \Pi_{x_k\in D_*}f_k(1+\xi_k-\eta_k)\lambda_{2}\Pi_{x_k\in D_*}{f_k}-\lambda_{1} \Pi_{x_k\in D_*}{f_k} \lambda_{2}\Pi_{x_k \in  D_*} f_k(1+\xi_k-\eta_k)}{\sum_{D_0} \Pi_{k=1}^{a_i} \phi_k\sum_{D_0} \Pi_{k=1}^{a_i} f_k} +\Omega =
\end{equation}

$$\sum_{\mathbf{\eta}\neq \mathbf{0}}\frac{(a_{i}-1)\lambda_{1}\lambda_{2}\sum_{D_*,D_*}[\Pi_{x_k\in D_*}f_k(1+\xi_k-\eta_k)\Pi_{x_k\in D_*}{f_k}- \Pi_{x_k\in D_*}{f_k} \Pi_{x_k \in  D_*} f_k(1+\xi_k-\eta_k)]}{\sum_{D_0} \Pi_{k=1}^{a_i} \phi_k\sum_{D_0} \Pi_{k=1}^{a_i} f_k} +\Omega =\sum_{\mathbf{\eta}\neq \mathbf{0}} \Omega. $$

Now we conclude the proof. From equation (\ref{eq:deltamain}) we can write $\Omega = \Omega_{+} - \Omega_{-}$ where $\Omega_{+}$ includes the positive compensatory terms from the (first part of) equation (\ref{eq:deltamain}) and $\Omega_{-}$ includes the negative compensatory terms from the (second part of) equation (\ref{eq:deltamain}). 
The following table  illustrates all of the subcases that constitute the $\Omega_{+,i}'s$ in equation (\ref{eq:deltamain}).
The columns indicate the term in the first summation found in equation (\ref{eq:deltamain})  ($\Omega_{+}$)  and the rows indicate the term in the second summation.  If the $i,j$th entry is a \checkmark, then this is a relevant subcase, while the case marked with \textbf{X} is not.   
\begin{center}
    \begin{tabular}{| l | l | l | l |}
    \hline
    & $\mathfrak{i}_{a_i-1}^{2}$ & $2\mathfrak{i}_{a_i-1}z_{a_i-1}$ & $z_{a_i-1}^{2}$ \\ \hline
    $\mathfrak{i}_{a_i-1}\mathfrak{i}_{a_i}$ & \textbf{X} & \checkmark & \checkmark \\ \hline
    $\mathfrak{i}_{a_i-1}z_{a_i}$ & \checkmark & \checkmark & \checkmark \\ \hline
    $\mathfrak{i}_{a_i}z_{a_i-1}$ &\checkmark & \checkmark & \checkmark \\ \hline
    $z_{a_i}z_{a_i-1}$ & \checkmark & \checkmark & \checkmark \\ \hline
    \end{tabular}
\end{center}

Note that only the first entry in the table is \textbf{X}; we already demonstrated through equation (\ref{eq:deltamain2}) that the contribution  from the
the case where both summations contribute $\mathfrak{i}'s$  without any $z's$ is $0$.
Hence, we have $11$ subcases and 
we can express $\Omega_{+}=\sum_{i=1}^{11} \Omega_{+,i}$ and $\Omega_{-}=\sum_{i=1}^{11} \Omega_{-,i}$.

As we have demonstrated in this work, the $z$'s represent terms that are much smaller than the $\mathfrak{i}'s$.  Consequently, {\em{interesting}} subcases will involve the fewest number of appearances fo $z$'s as possible.   Furthermore, by symmetry the terms corresponding to rows involving the $\mathfrak{i}_{a_i-1}z_{a_i}$ or $\mathfrak{i}_{a_i}z_{a_i-1}$ terms are identical due to symmetry. 
We therefore only need to examine two subcases, (a) the subcase where the first summation is $2\mathfrak{i}_{a_i-1}z_{a_i-1}$ and the second summation is $\mathfrak{i}_{a_i-1}\mathfrak{i}_{a_i}$, which we define to be $\Omega_{+,1},\Omega_{-,1}$ and (b) the subcase where the first summation is $\mathfrak{i}_{a_i-1}^{2}$ and the second summation is $\mathfrak{i}_{a_i-1}z_{a_i}$ which we define to be $\Omega_{+,3},\Omega_{-,3}$.  (We leave the subcase (b) as an exercise to the reader.)

So consider,

\begin{equation}\tag{C.13}
\Omega_{+,1} = \sum_{\mathbf{\eta}\neq \mathbf{0}}\frac{(a_{i}-1)\sum_{D_1}[2\mathfrak{i}_{a_i - 1}z_{a_i-1}]\Pi_{x_k\in F_1} f_k(1+\xi_k-\eta_k)\sum_{D_0}[\mathfrak{i}_{a_i - 1}][\mathfrak{i}_{a_i}]\Pi_{x_k\in F_0} {f_k}}{\sum_{D_0} \Pi_{k=1}^{a_i} \phi_k\sum_{D_0} \Pi_{k=1}^{a_i} {f_k}}.
\end{equation}
We will briefly consider the case where there is precisely one $\eta_k = 1$.  By symmetry we will consider $\eta_1 = 1$ and multiply this quantity by $a_i-2$ (as there are $a_i-2$ possible choices to let $\eta_k = 1$; recall that $\eta_{a_i-1}=\eta_{a_i}=0$), obtaining 

\begin{equation}\tag{C.14}
\Omega_{+,1}^{*}:=\frac{(a_{i}-2)(a_{i}-1)\sum_{D_1}[2\mathfrak{i}_{a_i - 1}z_{a_i-1}]f_1\xi_1\Pi_{x_k,k\neq 1\in F_1} f_k(1+\xi_k)\sum_{D_0}[\mathfrak{i}_{a_i - 1}][\mathfrak{i}_{a_i}]\Pi_{x_k\in F_0} {f_k}}{\sum_{D_0} \Pi_{k=1}^{a_i} \phi_k\sum_{D_0} \Pi_{k=1}^{a_i} {f_k}}.
\end{equation}

But by assumption $z_{a_i-1}\leq O(S^{-\frac{1}{2}-\tau})f_{a_i-1}$ and trivially $\mathfrak{i}_{a_i - 1}\leq Cf_{a_i - 1}$, for some constant $C$, so we have that 

\begin{equation}\tag{C.15}
\Omega_{+,1}^{*}\leq \frac{(a_{i}-2)(a_{i}-1)\sum_{D_1}[O(S^{-\frac{1}{2}-\tau})f_{a_i-1}^{2}]f_1\xi_1\Pi_{x_k,k\neq 1\in F_1} f_k(1+\xi_k)\sum_{D_0}[f_{a_i - 1}][f_{a_i}]\Pi_{x_k\in F_0} {f_k}}{\sum_{D_0} \Pi_{k=1}^{a_i} \phi_k\sum_{D_0} \Pi_{k=1}^{a_i} {f_k}}.
\end{equation}

Now applying the fact that $a_{i}\leq d_{max} = O(S^{\frac{1}{2}-\tau})$ and $\xi_1 = O(S^{-\frac{1}{2}-w\tau})$,

\begin{equation}\tag{C.16}
\Omega_{+,1}^{*}\leq O(S^{-(3+w)\tau})\frac{\sum_{D_1}[2f_{a_i-1}^{2}]f_1\Pi_{x_k,k\neq 1\in F_1} f_k(1+\xi_k)\sum_{D_0}[f_{a_i - 1}][f_{a_i}]\Pi_{x_k\in F_0} {f_k}}{\sum_{D_0} \Pi_{k=1}^{a_i} \phi_k\sum_{D_0} \Pi_{k=1}^{a_i} {f_k}}.
\end{equation}

Then employing the technique we used in the proof of Theorem \ref{thm:arborder1} by applying Corollary \ref{cor:expand} to reduce the domains $D_0$,$D_1$ to $D_*$ we get that 

\begin{equation}\tag{C.17}
\Omega_{+,1}^{*}\leq O(S^{-(3+w)\tau})O(S^{\frac{7}{2}-\tau})\frac{\sum_{D_*}\Pi_{x_k\in D_*} f_k(1+\xi_k)\sum_{D_*}\Pi_{x_k\in D_*} {f_k}}{O(S^{4})\sum_{D_*} \Pi_{k=1}^{a_i-2} \phi_k\sum_{D_*} \Pi_{k=1}^{a_i-2} {f_k}}=O(S^{-\frac{1}{2}-4\tau-w\tau}), 
\end{equation}
and hence this constribution is a higher order term.  Finally, we only considered the case where there is precisely one $\eta_k = 1$.  But cases where there are more positive $\eta_k's$  yield even smaller terms!  For example, when we had precisely one $\eta_k = 1$, there were at most $a_i$ choices for selecting $k$, but for each $\eta_k=1$, we end up multiplying the term $f_k$ by  $\xi_k$ (as opposed to $(1+\xi_k)$), and hence if there are $m$ $\eta_k's$ that equal $1$, then this results in a contribution bounded by $\Omega_{+,1}\leq \sum_{m=1}^{a_i}(a_i)^{m-1}\xi^{m-1}\Omega_{+,1}^{*}= \sum_{m=1}^{a_i}O(S^{-w(m-1)\tau-(m-1)\tau})\Omega_{+,1}^{*}=O(S^{-\frac{1}{2}-4\tau-w\tau})$.   Since these contributions are now much smaller than $O(S^{-\frac{1}{2}-2\tau-w\tau})$, we conclude that 

$$\Delta_i = O(S^{-\frac{1}{2}-2\tau-w\tau})$$ and from the beginning part of the proof of Theorem \ref{thm:arborder1} conclude that our new approximation $\frac{\|G_{\mathbf{d_{-i}}}\|}{\|G_{\mathbf{d_{-j}}}\|}=g(\mathbf{e}_{i},\mathbf{e_{j}},\mathbf{d},\mathbf{a})(1+O(S^{-\frac{1}{2}-(w+2)\tau}))$.

\end{proof}

\end{document}